\documentclass[10pt]{amsart}

\usepackage[french, english]{babel}
\usepackage[T1]{fontenc}
\usepackage{epigraph}
\usepackage{amsmath,amssymb,stmaryrd}
\usepackage{tikz}
\usepackage[all]{xy}
\usepackage{etoolbox}
\usepackage{option_keys}
\usepackage{fancyhdr}

\usepackage{hyperref} 
\patchcmd{\thebibliography}{*}{}{}{}
\pretocmd\thebibliography{\csname c@secnumdepth\endcsname=-2 }{}{}

\hypersetup{
colorlinks=black, 
breaklinks=true, 
urlcolor= blue, 
linkcolor= blue, 
citecolor=blue, 
pdftitle={Approximation de Artin}, 
pdfauthor={Guillaume Rond}, 
pdfsubject={Survey} 
} 
\usepackage{etoolbox}
\makeatletter
\patchcmd{\@thm}{\thm@headfont{\scshape}}{\thm@headfont{\scshape\bfseries}}{}{}
\patchcmd{\@thm}{\thm@notefont{\fontseries\mddefault\upshape}}{}{}{}
\makeatother

 \theoremstyle{plain}
 \newtheorem{lemma}{Lemma}[section]
\newtheorem{theorem}[lemma]{Theorem }
\newtheorem{corollary}[lemma]
{Corollary}
\newtheorem{proposition}[lemma]{Proposition}

\theoremstyle{definition}
\newtheorem{definition}[lemma]{Definition }

\newtheorem{example}[lemma]
{Example }

\newtheorem{remark}[lemma]{Remark}

\newcommand{\lgw}{\longrightarrow}
\newcommand{\lgm}{\longmapsto}
\newcommand{\s}{\sigma}
\newcommand{\ovl}{\overline}
\newcommand{\Frac}{\operatorname{Frac}}
\newcommand{\p}{\mathfrak{p}}
\renewcommand{\P}{\mathcal{P}}
\renewcommand{\deg}{\text{deg}\,}
\newcommand{\ord}{\text{ord}}
\newcommand{\Ker}{\operatorname{Ker}}

\newcommand{\wdh}{\widehat}
\newcommand{\I}{\mathcal{I}}
\newcommand{\ini}{\text{in}}
\newcommand{\wdt}{\widetilde}
\newcommand{\la}{\lambda}
\newcommand{\LL}{\mathbb{L}}
\renewcommand{\O}{\mathcal{O}}

\newcommand{\m}{\mathfrak{m}}

\newcommand{\Z}{\mathbb{Z}}
\newcommand{\Spec}{\operatorname{Spec}}

\renewcommand{\k}{\Bbbk}
\newcommand{\D}{\Delta}
\newcommand{\Supp}{\operatorname{Supp}}
\newcommand{\lb}{\llbracket}
\newcommand{\rb}{\rrbracket}
\newcommand{\G}{\Gamma}
\renewcommand{\dim}{\operatorname{dim}}

\newcommand{\R}{\mathbb{R}}
\newcommand{\K}{\mathbb{K}}

\newcommand{\N}{\mathbb{N}}

\newcommand{\C}{\mathbb{C}}
\newcommand{\Q}{\mathbb{Q}}
\newcommand{\E}{\mathcal{E}}
\newcommand{\F}{\mathcal{F}}
\renewcommand{\t}{\tau}
\renewcommand{\a}{\alpha}
\renewcommand{\b}{\beta}
\newcommand{\g}{\gamma}

\renewcommand{\phi}{\varphi}
\renewcommand{\d}{\delta}
\renewcommand{\o}{\omega}
\newcommand{\e}{\varepsilon}

\makeindex

\begin{document}
\selectlanguage{english}
%
\title{Artin Approximation}
\author{Guillaume ROND}

\begin{abstract} In 1968, M. Artin proved that any formal power series solution of a system of analytic equations may be approximated by convergent power series solutions. Motivated by this result and a similar result of A. P\l oski, he conjectured that this remains true when the ring of convergent power series is replaced by a more general kind of  ring.\\
This paper  presents the state of the art on this problem and its extensions. An extended introduction is aimed at non-experts. Then we present three main aspects of the subject: the classical Artin Approximation Problem, the Strong Artin Approximation Problem and  the Artin Approximation Problem with constraints. Three appendices present the algebraic material used in this paper (The Weierstrass Preparation Theorem, excellent rings and regular morphisms, \'etale and smooth morphisms and Henselian rings). \\
The goal is to review most of the known results and to give a list of references as complete as possible. We do not give the proofs of all the results presented in this paper but, at least, we always try to outline the proofs and give the main arguments together with precise references.

\end{abstract}
%
%
%
%
 \address{Aix Marseille UniversitŽ, CNRS, Centrale Marseille, I2M UMR 7373
13453, Marseille, France}

 \email{guillaume.rond@univ-amu.fr}


\thanks{The author was partially supported by ANR projects STAAVF (ANR-2011 BS01 009) and SUSI (ANR-12-JS01-0002-01)}
\maketitle

\epigraph{"On the whole, divergent series are the work of the Devil and it is a shame that one dares base any demonstration on them. You can get whatever result you want when you use them, and they have given rise to so many disasters and so many paradoxes."}{N. H. Abel, letter to Holmboe, January 1826}
%
%
\tableofcontents 


 \section{Introduction}
The aim of this text is to present the Artin Approximation Theorem and some related results. The  problem we are interested in is to find analytic solutions of some system of equations when this system admits formal power series solutions and the Artin Approximation Theorem yields a positive answer to this problem. This topic has been an important subject of study in the 70s and 80s having a wide variety of applications in algebraic and analytic geometry as well as in commutative algebra. But this is still an important subject of study that has been stimulated by several more recent applications (for instance in real algebraic geometry \cite{CRS}, to the theory of arc spaces \cite{Ve} or in CR geometry \cite{MMZ}). \\
\\
The main object of study are power series, i.e. formal sums of the form 
$$\sum_{\a\in\N^n} c_\a x_1^{\a_1}\ldots x_n^{\a_n}$$
and the relations that they satisfy. We will be mainly interested in the nature of series that satisfy some given relations: convergent power series, divergent power series, algebraic power series, \ldots . But we will also be interested in studying the existence of power series satisfying some given relations.\\
\\
We begin this paper by giving several examples explaining what are precisely the different problems we aim to study. Then we will present the state of the art on this problem. \\
This text contains essentially three parts: the first part is dedicated to present the Artin Approximation Theorem and its generalizations; the second part presents a stronger version of the Artin Approximation Theorem; the last part is mainly devoted to exploring the Artin Approximation Problem in the case of constraints. The end of the text contains 3 appendices: the first one concerns the Weierstrass  Preparation and Division Theorems, the second one concerns the notions of excellent ring and regular morphism, and the last one reviews the main definitions and results concerning the \'etale morphisms and Henselian rings.  \\
\\
This paper is an extended version of the habilitation thesis of the author. I wish to thank the members of the jury of my habilitation thesis who encouraged me to improve the first version of this writing: Edward Bierstone, Charles Favre, Herwig Hauser, 
Michel Hickel, Adam Parusi\'nski, Anne Pichon and Bernard Teissier. \\
I wish to thank especially Herwig Hauser for his  encouragement  on the very first stage of this writing. In particular the idea of writing an extended introduction with examples was his proposal (and he also suggested some examples).\\
I also wish to thank the participants of the Chair Jean Morlet at CIRM for the fruitful discussions that helped me to improve the text, in particular the participants of the doctoral school and more specifically Francisco Castro-Jim\'enez, Christopher Chiu, Alberto Gioia,  Dorin Popescu, Kaloyan Slavov,  Sebastian Woblistin.\\
Last but not least I thank  Mark Spivakovsky who introduced me to this subject 15 years ago and from whom I learned very much.\\
Finally I wish to thank the referee for their careful reading of the paper and for their useful and relevant remarks.\\
\\

\begin{example}\label{exe1}
Let us consider the following curve $C:=\{(t^3,t^4,t^5), t\in\C\}$ in $\C^3$. This curve is an algebraic set which means that it  is the zero locus of polynomials in three variables. Indeed, we can check that $C$ is the zero locus of the polynomials $f:=y^2-xz$, $g:=yz-x^3$ and $h:=z^2-x^2y$. If we consider the zero locus of any two of these polynomials we always get a set larger than $C$. 
The complex dimension of the zero locus of one non-constant polynomial in three variables is 2 (such a set is called a hypersurface of $\C^3$). Here $C$ is the intersection of the zero locus of three  hypersurfaces and not of two of them, but its complex dimension is 1. \\
In fact we can see this phenomenon as follows: we call a linear algebraic relation between $f$, $g$ and $h$ any element of the kernel of the linear  map $\phi : \C[x,y,z]^3\lgw \C[x,y,z]$ defined by $\phi(a,b,c):=af+bg+ch$. Obviously $r_1:=(g,-f,0)$, $r_2:=(h,0,-f)$ and $r_3:=(0,h,-g)\in\Ker(\phi)$. These are called the trivial linear relations between $f$, $g$ and $h$. But in our case there are two more linear relations which are $r_4:=(z,-y,x)$ and $r_5:=(x^2,-z,y)$ and $r_4$ and $r_5$ cannot be written as $a_1r_1+a_2r_2+a_3r_3$ with $a_1$, $a_2$ and $a_3\in\C[x,y,z]$, which means that $r_4$ and $r_5$ are not in the sub-$\C[x,y,z]$-module of $\C[x,y,z]^3$ generated by $r_1$, $r_2$ and $r_3$.\\
On the other hand we can prove, using the theory of Gr\"obner bases,  that $\Ker(\phi)$ is generated by $r_1$, $r_2$, $r_3$, $r_4$ and $r_5$.\\
Let $X$ be the common zero locus of $f$ and $g$. If $(x,y,z)\in X$ and $x\neq 0$, then $h=\frac{zf+yg}{x}=0$ thus $(x,y,z)\in C$. If $(x,y,z)\in X$ and $x=0$, then $y=0$. Geometrically this means that $X$ is the union of $C$ and the $z$-axis, i.e. the union of two curves.\\
\\
Now let us denote by $\C\lb x,y,z\rb$ the ring of formal power series with coefficients in $\C$. We can also consider linear formal relations between $f$, $g$ and $h$ that is, elements of the kernel of the map $\C\lb x,y,z\rb^3\lgw \C\lb x,y,z\rb$ induced by $\phi$.
Of course any element of the form $a_1r_1+a_2r_2+a_3r_3+a_4r_4+a_5r_5$, where  $a_1,\ldots,a_5\in\C\lb x,y,z\rb$, is a linear formal relation between $f$, $g$ and $h$.\\
In fact any linear formal relation is of this form, i.e. the linear algebraic relations generate the linear formal relations. We can show this as follows: we can assign the weights 3 to $x$, 4 to $y$ and 5 to $z$. In this case $f$, $g$, $h$ are weighted homogeneous polynomials of weights $8$, $9$ and $10$ and $r_1$, $r_2$, $r_3$, $r_4$ and $r_5$ are weighted homogeneous relations of weights $(9,8,-\infty)$, $(10, -\infty, 8)$, $(-\infty,10,9)$, $(5,4,3)$, $(6,5,4)$ ($-\infty$ is by convention  the degree of the zero polynomial). If $(a,b,c)\in\C\lb x,y,z\rb^3$ is a linear formal relation then we can write $a=\sum_{i=0}^{\infty}a_i$, $b=\sum_{i=0}^{\infty}b_i$ and $c=\sum_{i=0}^{\infty}c_i$ where $a_i$, $b_i$ and $c_i$ are weighted homogeneous polynomials  of degree $i$ with respect to the above weights. Then saying that $af+bg+ch=0$ is equivalent to $$a_if+b_{i-1}g+c_{i-2}h=0 \ \ \ \forall i\in\N$$ with the assumption $b_i=c_i=0$ for $i<0$. Thus $(a_0,0,0)$, $(a_1,b_0,0)$ and any $(a_i,b_{i-1},c_{i-2})$, for $2\leq i$, are in $\Ker(\phi)$, thus are weighted homogeneous linear combinations of $r_1,\ldots,r_5$. Hence $(a,b,c)$ is a linear combination of  $r_1,\ldots,r_5$ with coefficients in $\C\lb x,y,z\rb$.\\
\\
 Now we can investigate the same problem by replacing the ring of formal power series by $\C\{x,y,z\}$, the ring of \index{convergent power series} convergent power series with coefficients in $\C$, i.e. 
$$\C\{x,y,z\}:=\left\{\sum_{i,j,k\in\N}a_{i,j,k}x^iy^jz^k\ /\ \exists \rho>0, \sum_{i,j,k}|a_{i,j,k}|\rho^{i+j+k}<\infty\right\}.$$
We can also consider linear analytic relations between $f$, $g$ and $h$, i.e. elements of the kernel of the map $\C\{x,y,z\}^3\lgw \C\{x,y,z\}$ induced by $\phi$.
From the linear formal case we see that any linear analytic relation $r$ is of the form $a_1r_1+a_2r_2+a_3r_3+a_4r_4+a_5r_5$ with $a_i\in\C\lb x,y,z\rb$ for $1\leq i\leq 5$. And there is no reason for the $a_i$ to be convergent. For instance we can check that
$$hr_1-gr_2+fr_3=0.$$
So if $\wdh b$ is a divergent power series we see that $(\wdh bh)r_1-(\wdh bg)r_2+(\wdh b f)r_3$ is a linear formal relation but not an analytic relation. But we can prove that every linear analytic relation has the form $a_1r_1+a_2r_2+a_3r_3+a_4r_4+a_5r_5$ where the   $a_i$ can be chosen to be convergent power series, and our goal is to describe how to do this.\\
 Let us remark that the equality $r=a_1r_1+\cdots+a_5r_5$ is equivalent to saying that $a_1,\ldots,a_5$ satisfy a system of  three linear equations with analytic coefficients. This is the first example of the problem we are interested in: if  some equations with analytic coefficients have formal solutions do they have analytic solutions? The Artin Approximation Theorem yields an answer to this problem. Here is the first theorem proven by M. Artin in 1968:

\begin{theorem}[Artin Approximation Theorem]\label{Ar68_intro}\cite{Ar68}
  Let $F(x,y)$ be a  
  vector of convergent power series over $\C$ in two sets of variables $x$ and $y$. Assume given a formal power series solution $\wdh{y}(x)$,
  $$F(x,\wdh{y}(x))=0.\footnote{To be rigorous one needs to assume that $\wdh y(0)=0$ to ensure that $F(x,\wdh y(x))$ is well defined. But we can drop this assumption if $F(x,y)$ is polynomial in $y$ which is the case of most of the examples given in the introduction.}$$
  Then, for any $c\in\N$,  there exists a convergent power series solution $y(x)$,
  $$F(x,y(x))=0$$
  which coincides with $\wdh{y}(x)$ up to degree $c$,
  $$y(x)\equiv \wdh{y}(x) \text{ modulo } (x)^c.$$
  
  \end{theorem}
We can define a topology on $\C\lb x\rb$, $x=(x_1,\ldots,x_n)$ being a set of variables, by saying that two power series are close to each other if  they are equal up to a high degree. Thus we can reformulate Theorem \ref{Ar68_intro} as follows: formal power series solutions of a system of analytic equations may be approximated by convergent power series solutions (see Remark \ref{rem_topology} in the next part for a precise definition of this topology).
\end{example}


\begin{example}\label{flat_ana}
A special case of Theorem \ref{Ar68_intro} and a generalization of Example \ref{exe1}  occurs when $F$ is homogeneous linear in $y$, say $F(x,y)=\sum f_i(x)y_i$, where $f_i(x)$ is a vector of convergent power series with $r$ coordinates for each $i$ and $x$ and $y$ are two sets of variables.  A solution $y(x)$ of $F(x,y)=0$ is a linear relation between the $f_i(x)$. In this case the linear formal relations are linear combinations of linear analytic relations with coefficients in $\C\lb x\rb$. In term of commutative algebra, this is expressed as the \index{flatness} flatness of the ring of formal power series over the ring of convergent powers series, a result which can be proven via  Artin-Rees Lemma (see Remark \ref{flat} in the next part and Theorems 8.7 and 8.8 \cite{Ma}) and which is much more elementary than Theorem \ref{Ar68}.\\
It means that if $\wdh{y}(x)$ is a formal solution of $f(x,y)=0$, then there exist analytic solutions of $F(x,y)=0$ denoted by  $\wdt{y}_i(x)$, $1\leq i\leq s$, and formal power series $\wdh{b}_1(x),\ldots,\wdh{b}_s(x)$, such that $\wdh{y}(x)=\sum_i\wdh{b}_i(x)\wdt{y}_i(x)$. Thus, by replacing in the previous sum the $\wdh{b}_i(x)$ by their truncation at order $c$, we obtain an analytic solution of $f(x,y)=0$ coinciding with $\wdh{y}(c)$ up to degree $c$.\\
If the $f_i(x)$  are vectors of polynomials then  the linear formal relations are also linear combinations of linear algebraic relations since the ring of formal power series is flat over the ring of polynomials, and Theorem \ref{Ar68_intro} remains true if $F(x,y)$ is linear in $y$ and $\C\{x\}$ is replaced by $\C[x]$.

\end{example}


\begin{example}\label{faithfully_flat}
A  slight generalization of the previous example is when $F(x,y)$ is a vector of polynomials in $y$ of degree one with coefficients in $\C\{x\}$ (resp. $\C[x]$), say 
$$F(x,y)=\sum_{i=1}^mf_i(x)y_i+b(x)$$
where the $f_i(x)$  and $b(x)$ are vectors of convergent power series (resp. polynomials). Here $x$ and $y$ are multivariables. If $\wdh{y}(x)$ is a formal power series solution of $F(x,y)=0$, then $(\wdh{y}(x),1)$ is a formal power series solution of $G(x,y,z)=0$ where
$$G(x,y,z):=\sum_{i=1}^mf_i(x)y_i+b(x)z$$
and $z$ is a single variable. Thus using the flatness of $\C\lb x\rb$ over $\C\{x\}$ (resp. $\C[x]$), as in Example \ref{flat_ana}, we can approximate $(\wdh{y}(x),1)$ by a convergent power series (resp. polynomial) solution $(\wdt{y}(x),\wdt{z}(x))$ which coincides with $(\wdh{y}(x),1)$ up to degree $c$. In order to obtain a solution of $F(x,y)=0$ we would like to be able to divide $\wdt{y}(x)$ by $\wdt{z}(x)$ since $\wdt{y}(x)\wdt{z}(x)^{-1}$ would be a solution of $F(x,y)=0$ approximating $\wdh{y}(x)$. We can remark that, if $c\geq 1$, then $\wdt{z}(0)=1$ thus $\wdt{z}(x)$ is not in the ideal  $(x)$. But $\C\{x\}$ is  a \index{local ring} local ring. We call a \emph{local ring} any ring $A$ that has only one maximal ideal. This is equivalent to saying that $A$ is the disjoint union of one ideal (its only maximal ideal) and of the set of units in $A$. Here the units of $\C\{x\}$ are exactly the power series $a(x)$ such that $a(x)$ is not in the ideal $(x)$, i.e. such that $a(0)\neq 0$. In particular  $\wdt{z}(x)$ is invertible in $\C\{x\}$, hence we can approximate formal power series solutions of $F(x,y)=0$ by convergent power series solutions.\\
\\
In the case $(\wdt{y}(x),\wdt{z}(x))$ is a polynomial solution of $g(x,y,z)=0$, $\wdt{z}(x)$ is not invertible in general in $\C[x]$ since it is not a local ring.  For instance set 
$$F(x,y):=(1-x)y-1$$
where $x$ and $y$ are single variables. Then  $y(x):=\displaystyle\sum_{n=0}^{\infty}x^n=\frac{1}{1-x}$ is the only formal power series solution of $F(x,y)=0$, but $y(x)$ is not a polynomial. Thus we cannot approximate the roots of $F$ in $\C\lb x\rb$ by roots of $F$ in $\C[x]$.\\
\\
But instead of working in $\C[x]$ we can work in $\C[x]_{(x)}$ which is the ring of rational functions whose denominator does not vanish at 0. This ring is a local ring. Since $\wdt{z}(0)\neq 0$,  $\wdt{y}(x)\wdt{z}(x)^{-1}$ is a  rational function belonging to $\C[x]_{(x)}$.  In particular any system of polynomial equations of degree one with coefficients in $\C[x]$ which has solutions in $\C\lb x\rb$ has solutions in $\C[x]_{(x)}$.  \\
\\
In term of commutative algebra, the fact that degree 1 polynomial equations satisfy Theorem \ref{Ar68_intro} is expressed as the \index{faithful flatness} faithful flatness  of the ring of formal power series over the ring of convergent powers series, a result that follows from the flatness and the fact that the ring of convergent power series is a local ring (see also Remark \ref{flat_ploski} in the next part). 
 \end{example}


\begin{example}[Implicit Function Theorem]
Let $F(x,y)$ be a polynomial in two sets of variables $x=(x_1,\ldots, x_n)$ and $y=(y_1,\ldots, y_m)$,  with $F(0,0)=0$. Let us assume that there exists a vector of formal power series $\wdh y(x)\in\C\lb x\rb^m$ solution of $F=0$ and vanishing at 0:
$$F(x,\wdh y(x))=0.$$
So by Theorem \ref{Ar68_intro} there exists a vector of convergent power series solution of $F=0$. But in general the equation $F=0$ has a infinite solution set and the given formal solution has no reason to be already convergent. For instance if $F(x,y)=y_1-y_2$, every vector of the form $(\wdh z(x), \wdh z(x))$ where $\wdh z\in\C\lb x\rb$ is a solution.\\
\\
But now let us assume that $m=1$ and  that $\frac{\partial F}{\partial y}(0,0)\neq 0$. Since $\wdh y(0)=0$ and $F(0,0)=0$ we can apply the Implicit Function Theorem for convergent power series to see that the equation $F=0$ has a unique convergent power series solution vanishing at 0. In particular, the uniqueness of this solution implies that $\wdh y$ is already a vector of convergent power series.\\
\\
In fact the proof of Theorem \ref{Ar68_intro} consists in reducing the problem to a situation where the Implicit Function Theorem applies.

\end{example}


\begin{example}\label{ex_2}
The next example we are looking at is the following one: set $f\in A$ where $A=\C[x]$ or $\C[x]_{(x)}$ or $\C\{x\}$ with $x$  a finite set of variables. When do there exist $g$, $h\in A$ such that $f=gh$?\\
First of all, we can take $g=1$ and $h=f$ or, more generally,  $g$  a unit in $A$ and $h=g^{-1}f$. These are trivial cases and thus we are looking for non-units $g$ and $h$. \\
Of course, if there exist non-units $g$ and $h$ in $A$ such that $f=gh$, then $f=(\wdh{u}g)(\wdh{u}^{-1}h)$ for any unit $\wdh{u}\in\C\lb x\rb$. But is the following true: let us assume that there exist $\wdh{g}$, $\wdh{h}\in\C\lb x\rb$ such that $f=\wdh{g}\wdh{h}$, then do there exist non-units $g$, $h\in A$ such that $f=gh$?\\
Let us remark that this question is equivalent to the following: if $\frac{A}{(f)}$ is an integral domain, is $\frac{\C\lb x\rb}{(f)\C\lb x\rb}$ still an integral domain?\\
\\
The answer to this question is no in general: for example set $A:=\C[x,y]$ where $x$ and $y$ are single variables
 and  $f:=x^2-y^2(1+y)$. The polynomial $f$ is irreducible since $y^2(1+y)$ is not a square in $\C[x,y]$. But as a power series we can factor $f$ as  
 $$f=(x+y\sqrt{1+y})(x-y\sqrt{1+y})$$
  where $\sqrt{1+y}$ is a formal power series such that $(\sqrt{1+y})^2=1+y$. Thus $f$ is not irreducible in $\C\lb x,y\rb$ nor in $\C\{x,y\}$ but it is irreducible in $\C[x,y]$ or $\C[x,y]_{(x,y)}$.\\
 \\
  In fact it is easy to see that $x+y\sqrt{1+y}$  and $x-y\sqrt{1+y}$ are power series which are algebraic over $\C[x,y]$, i.e. they are roots of polynomials with coefficients in $\C[x,y]$ (here they are roots of the polynomial $(z-x)^2-y^2(1+y)$). The set of such algebraic power series is a subring of $\C\lb x,y\rb$ denoted by $\C\langle x,y\rangle$. In general if $x$ is a multivariable the ring of algebraic power series  \index{algebraic power series} $\C\langle x\rangle$ is the following:
  $$\C\langle x\rangle:=\left\{f\in\C\lb x\rb\ / \ \exists P(z)\in\C[x][z], \ P(z)\neq 0,\ \ P(f)=0\right\}.$$
  It is not difficult to prove that the ring of algebraic power series is a subring of the ring of convergent power series and is a local ring.
   In 1969, M. Artin proved an analogue of Theorem \ref{Ar68} for the rings of algebraic power series \cite{Ar69} (see Theorem \ref{Ar69} in the next part). Thus if $f\in\C\langle x\rangle$ (or  $\C\{x\}$) is irreducible then it remains irreducible in $\C\lb x\rb$, this is a  consequence of this Artin Approximation Theorem for algebraic power series applied to the equation $y_1y_2-f$. Indeed if $f$ were reducible in $\C[[x]]$ there would exist non-units $\wdh y_1$, $\wdh y_2\in\C[[x]]$ (i.e. $y_1(0)=y_2(0)=0$) such that $\wdh y_1\wdh y_2-f=0$. Then by Artin Approximation Theorem of algebraic power series applied with $c\geq 1$ there would exist $\wdt y_1$, $\wdt y_2\in\C\langle x\rangle$ such that $\wdt y_1\wdt y_2=f$ and $\wdt y_i-\wdh y_i\in (x)^c$ for $i=1,2$. In particular $\wdt y_1(0)=\wdt y_2(0)=0$ so $\wdt y_1$ and $\wdt y_2$ are non-unit which contradicts the fact that $f$ is assumed to be irreducible.\\
   From this theorem we  can also deduce that if $f\in\frac{\C\langle x\rangle}{I}$ (or $\frac{\C\{x\}}{I}$) is irreducible, for some ideal $I$, it remains irreducible in $\frac{\C\lb x\rb}{I\C\lb x\rb}$.  
 \end{example}
 
 
 \begin{example}
 Let us strengthen the above question. Let us assume that there exist non-units $\wdh{g}$, $\wdh{h}\in\C\lb x\rb$ such that $f=\wdh{g}\wdh{h}$ with $f\in A$ with $A=\C\langle x\rangle$ or $\C\{x\}$. Then does there exist a unit $\wdh{u}\in\C\lb x\rb$ such that $\wdh{u}\wdh{g}\in A$ and $\wdh{u}^{-1}\wdh{h}\in A$ ?\\
  The answer to this question is  positive if $A=\C\langle x\rangle$ or $\C\{x\}$, this is a  corollary of the  Artin Approximation Theorem (see Corollary \ref{ufd_fact}). But it is negative in general for $A=\frac{\C\langle x\rangle}{I}$ or $\frac{\C\{x\}}{I}$ if $I$ is an ideal as shown by the  following example  due to S. Izumi \cite{Iz92} (nevertheless it remains true when $A$ is normal - see Corollary \ref{Iz92}):\\ 
  \\
   Set $A:=\frac{\C\{x,y,z\}}{(y^2-x^3)}$. Set $\wdh{\phi}(z):=\sum_{n=0}^{\infty}n!z^n$ (this is a divergent power series) and set 
 $$\wdh{f}:=x+y\wdh{\phi}(z),\ \ \ \wdh{g}:=(x-y\wdh{\phi}(z))(1-x\wdh{\phi}(z)^2)^{-1}\in\C\lb x,y,z\rb.$$
 Then we can  check that $x^2=\wdh{f}\wdh{g}$ modulo ${(y^2-x^3)}$. Now let us assume that there exists a unit $\wdh{u}\in \C\lb x,y,z\rb$ such that $\wdh{u}\wdh{f}\in \C\{x,y,z\}$ modulo $(y^2-x^3)$. Thus the element $P:=\wdh{u}\wdh{f}-(y^2-x^3)\wdh{h}$ is a convergent power series for some well chosen $\wdh h\in\C\lb x,y,z\rb$. We can check easily that $P(0,0,0)=0$ and $\frac{\partial P}{\partial x}(0,0,0)=\wdh{u}(0,0,0)\neq 0$. Thus by the Implicit Function Theorem for analytic functions there exists $\psi(y,z)\in\C\{y,z\}$, such that $P(\psi(y,z),y,z)=0$ and $\psi(0,0)=0$. This yields
 $$\psi(y,z)+y\wdh{\phi}(z)-(y^2-\psi(y,z)^3)\wdh{h}(\psi(y,z),y,z)\wdh{u}^{-1}(\psi(y,z),y,z)=0. $$
 By substituting 0 for $y$  we obtain $\psi(0,z)+\psi(0,z)^3\wdh{k}(z)=0$ for some power series $\wdh{k}(z)\in\C\lb z\rb$. Since $\psi(0,0)=0$, the order of the power series $\psi(0,z)$ is positive, hence the previous equality shows that $\psi(0,z)=0$. Thus $\psi(y,z)=y\theta(y,z)$ with $\theta(y,z)\in\C\{y,z\}$. Thus we obtain
 $$\theta(y,z)+\wdh{\phi}(z)-(y-y^2\theta(y,z)^3)\wdh{h}(\psi(y,z),y,z)\wdh{u}^{-1}(\psi(y,z),y,z)=0$$ 
 and by substituting 0 for $y$, we see that $\wdh{\phi}(z)=\theta(0,z)\in\C\{z\}$ which is a contradiction.\\
 Thus $x^2=\wdh{f}\wdh{g}$ modulo $(y^2-x^3)$ but there is no unit $\wdh{u}\in\C\lb x,y,z\rb$ such that $\wdh{u}\wdh{f}\in\C\{x,y,z\}$ modulo $(y^2-x^3)$. \end{example}
 

\begin{example}
Using the same notation as in Example \ref{ex_2} we can ask a stronger question: set $A=\C\langle x\rangle$ or $\C\{x\}$ and let $f$ be in $A$. If there exist $\ovl{g}$ and $\ovl{h}\in \C[x]$, vanishing at 0,  such that $f=\ovl{g}\ovl{h}$ modulo a large power of the ideal $(x)$, do there exist $g$ and $h$ in $A$, vanishing at 0, such that $f=gh$? We just remark that by  Example \ref{ex_2} there is no hope, if $f$ is a polynomial and $g$ and $h$ exist, to expect that $g$ and $h\in \C[x]$.\\
Nevertheless we have the following theorem that gives a precise answer to this question:
\begin{theorem}[Strong Artin Approximation Theorem]\label{SAP_intro}\cite{Ar69}
  Let $F(x,y)$ be a vector of convergent power series over $\C$ in two sets of variables $x$ and $y$. Then  for any integer $c$ there exists an integer $\b$ such that for any  given approximate solution  $\ovl{y}(x)$ at order $\b$,
  $$F(x,\ovl{y}(x))\equiv 0\text{ modulo } (x)^{\b},$$
 there exists a convergent  power series solution $y(x)$,
  $$F(x,y(x))=0$$
  which coincides with $\ovl{y}(x)$ up to degree $c$,
  $$y(x)\equiv \ovl{y}(x) \text{ modulo } (x)^c.$$\end{theorem}
  In particular we can apply this theorem  to the polynomial  $y_1y_2-f$ with $c=1$. It shows that there exists an integer $\b$ such that if $\ovl{g}\ovl{h}-f\equiv 0$ modulo $(x)^{\b}$ and if $\ovl{g}(0)=\ovl{h}(0)=0$,  there exist non-units $g$ and $h\in\C\{ x\}$ such that $gh-f=0$.\\
  \\
 For a given $F(x,y)$ and a given $c$ let $\b(c)$ denote the smallest integer $\b$ satisfying the above theorem. A natural question is: how to compute or bound the function $c\lgm\b(c)$ or, at least, some values of $\b$? For instance when $F(x,y)=y_1y_2-f(x)$,  $f(x)\in\C[x]$, what is the value or a bound of $\b(1)$? That is, up to what order do we have to check that the equation $y_1y_2-f=0$ has an approximate solution in order to be sure that this equation has non-trivial  solutions (i.e. which are not vanishing at 0)? For instance, if $f:=x_1x_2-x_3^d$ then $f$ is irreducible but $x_1x_2-f\equiv 0$ modulo $(x)^d$ for any $d\in\N$, so obviously $\b(1)$ really depends on $f$. \\ 
 \\
 In fact in Theorem \ref{SAP_intro} M. Artin proved that $\b$ can be chosen to depend only on the degree of the components of the  vector $F(x,y)$. But it is still an open problem to find effective bounds on $\b$ (see Section \ref{effective}).

\end{example}


\begin{example}[Arc Space and Jet Spaces]
Let $X$ be an affine algebraic subset of $\C^m$, i.e. $X$ is the zero locus of some polynomials in $m$ variables: $f_1,\ldots,f_r\in\C[y_1,\ldots,y_m]$ and let $F$ denote the vector $(f_1,\ldots,f_r)$. Let $t$ be a single variable. For any integer $n$, let us define $X_n$ to be the set of  vectors $y(t)$ whose coordinates are polynomials of degree $\leq n$ and  such that $F(y(t))\equiv 0$ modulo $(t)^{n+1}$. The elements of $X_n$ are called \index{jets} $n$-jets of $X$.  \\
If $y_i(t)=y_{i,0}+y_{i,1}t+\cdots+y_{i,n}t^n$ and if we consider each $y_{i,j}$ as an indeterminate, saying that $F(y(t))\in (t)^{n+1}$ is equivalent to the vanishing of the coefficient of $t^k$, for $0\leq k\leq n$, in the expansion of every $f_i(y(t))$. Thus this is equivalent to the vanishing of $r(n+1)$ polynomial equations involving the $y_{i,j}$. This shows that the jet spaces of $X$ are algebraic sets (here $X_n$ is an algebraic subset of $\C^{m(n+1)}$).\\
For instance if $X$ is a cusp (this example is taken from \cite{Ve}), i.e. the plane curve defined as $X:=\{y_1^2-y_2^3=0\}$ we have 
$$X_0:=\{(a_0,b_0)\in\C^2\ /\ a_0^2-b_0^3=0\}=X.$$ We have 
$$X_1=\{(a_0,a_1,b_0,b_1)\in\C^4\ / \ (a_0+a_1t)^2-(b_0+b_1t)^3\equiv 0\text{ modulo } t^2\}$$
$$=\{(a_0,a_1,b_0,b_1)\in\C^4\ / \ a_0^2-b_0^3=0\text{ and } 2a_0a_1-3b_0^2b_1=0\}.$$\\
The morphisms $\frac{\C[t]}{(t)^{k+1}}\lgw \frac{\C[t]}{(t)^{n+1}}$, for $k\geq n$, induce truncation maps $\pi_n^k:X_{k}\lgw X_n$ by reducing $k$-jets modulo $(t)^{n+1}$. In the example we are considering, the fibre of $\pi_0^1$ over the point $(a_0,b_0)\neq (0,0)$ is the line in the $(a_1,b_1)$-plane whose equation is $2a_0a_1-3b_0^2b_1=0$. This line is exactly the  tangent space at $X$ at the point $(a_0,b_0)$.
The Zariski tangent space at $X$ in $(0,0)$ is the whole plane since this point is a singular point of the plane curve $X$. This corresponds to the fact that the fibre of $\pi_0^1$ over $(0,0)$ is the whole plane.\\
On this example we show that $X_1$ is isomorphic to the tangent bundle of $X$, which is a general fact.\\
We can easily see that $X_2$ is given by the following equations:
$$\left\{\begin{aligned}
 & a_0^2-b_0^3=0\\ & 2a_0a_1-3b_0^2b_1=0\\
 & a_1^2+2a_0a_2-3b_0b_1^2-3b_0^2b_2=0 \end{aligned}\right.$$
 In particular, the fibre of $\pi_0^2$ over the point $(0,0)$ is the set of points of the form $(0,0,a_2,0,b_1,b_2)$ and the image of this fibre by $\pi_1^2$ is the line $a_1=0$. This shows that $\pi_1^2$ is not surjective.\\
 But, we can show that above the smooth part of $X$, the maps $\pi_n^{n+1}$ are surjective and the fibres are isomorphic to $\C$.\\
 \\
 The  \index{arcs space} space of arcs of $X$, denoted by $X_{\infty}$, is the set of vectors $y(t)$ whose coordinates are formal power series satisfying $F(y(t))=0$. For such a general vector of formal power series $y(t)$, saying that $F(y(t))=0$ is equivalent to saying that the coefficients of all the powers of $t$ in the Taylor expansion of $F(y(t))$ are equal to zero. This shows that $X_{\infty}$ may be defined by a countable number of equations in a countable number of variables. For instance, in the previous example, $X_{\infty}$ is the subset of $\C^{\N}$ with coordinates $(a_0,a_1,a_2,\ldots,b_0,b_1,b_2,\ldots)$ defined by the infinite following equations:
 $$\left\{\begin{aligned}
 & a_0^2-b_0^3=0\\ & 2a_0a_1-3b_0^2b_1=0\\
 & a_1^2+2a_0a_2-3b_0b_1^2-3b_0^2b_2=0\\
 & \cdots \cdots\cdots\end{aligned}\right.$$
The morphisms $\C\lb t\rb\lgw \frac{\C[t]}{(t)^{n+1}}$ induce truncations maps $\pi_n:X_{\infty}\lgw X_n$ by reducing arcs modulo $(t)^{n+1}$.\\
In general it is a difficult problem to compare $\pi_n(X_{\infty})$ and $X_n$. It is not even clear if $\pi_n(X_{\infty})$ is finitely defined in the sense that it is defined by a finite number of equations involving a finite number of $y_{i,j}$. But we have the following theorem due to M. Greenberg which is a particular case of Theorem \ref{SAP_intro} in which $\b$ is bounded by an affine function:
\begin{theorem}[Greenberg's  Theorem]\cite{Gr}\label{Gr_intro}
Let $F(y)$ be a vector of polynomials in $m$ variables and let $t$ be a single variable. Then there exist two positive integers  $a$ and $b$, such that for any integer $n$ and any polynomial solution $\ovl{y}(t)$ modulo $(t)^{an+b+1}$,
$$F(\ovl{y}(t))\equiv 0\text{ modulo } (t)^{an+b+1},$$
there exists a formal power series solution $\wdt{y}(t)$,
$$F(\wdt{y}(t))=0$$
which coincides with $\ovl{y}(t)$ up to degree $n+1$, that is
$$\ovl{y}(t)\equiv\wdt{y}(t)\text{ modulo } (t)^{n+1}.$$\\
\end{theorem}
We can reinterpret this result as follows: Let  $X$ be the zero locus of $F$ in $\C^m$ and let $y(t)$ be a $(an+b)$-jet on $X$. Then the truncation of $y(t)$ modulo $(t)^{n+1}$ is the truncation of a formal power series solution of $F=0$. Thus we have
\begin{equation}\label{popop}\pi_n(X_{\infty})=\pi_n^{an+b}(X_{an+b}),\ \ \forall n\in\N.\end{equation}\\
A constructible subset of  $\C^n$ is  a set defined by the vanishing of some polynomials and the non-vanishing of other polynomials, i.e. a set which is a finite union of sets of the form $$\{x\in\C^n\ /\ f_1(x)=\cdots=f_r(x)=0, g_1(x)\neq 0,\ldots, g_s(x)\neq0\}$$ for some polynomials $f_i$, $g_j$. In particular algebraic sets are constructible sets. A \index{theorem of Chevalley}  theorem of Chevalley asserts that the projection of an algebraic subset of $\C^{n+k}$ onto $\C^k$ is a constructible subset of $\C^n$, so \eqref{popop} shows that $\pi_n(X_{\infty})$ is a \index{constructible subset} constructible subset of $\C^n$ since $X_{an+b}$ is an algebraic set. In particular $\pi_n(X_{\infty})$ is finitely defined (see  \cite{GL} for an introduction to the study of these sets).\\
\\
A difficult problem in singularity theory is to understand the behavior of $X_n$ and $\pi_n(X_{\infty})$ and to relate them to the geometry of $X$. One way to do this is to define the (motivic) measure of a constructible subset of $\C^n$, that is an additive map $\chi$ from the set of constructible sets to a commutative ring $R$ such that:\\
 $\bullet$ $\chi(X)=\chi(Y)$ as soon as $X$ and $Y$ are isomorphic algebraic sets,\\
 $\bullet$  $\chi(X\backslash U)+\chi(U)=\chi(X)$ as soon as $U$ is an open set of an algebraic set $X$,\\
 $\bullet$  $\chi(X\times Y)=\chi(X).\chi(Y)$ for any algebraic sets $X$ and $Y$.\\
 Then we are interested in understanding the following generating series: 
 $$\sum_{n\in\N}\chi(X_n)T^n \text{ and } \sum_{n\in\N}\chi(\pi_n(X_{\infty}))T^n\in R\lb T\rb.$$
The reader may consult \cite{De-Lo,Lo,Ve} for an introduction to  these problems.
\end{example}


\begin{example}\label{ex_art}
Let $f_1,\ldots, f_r\in\k[x,y]$ where $\k$ is an algebraically closed field and $x:=(x_1,\ldots,x_n)$ and $y:=(y_1,\ldots,y_m)$ are multivariables.  Moreover we will assume here that $\k$ is  uncountable. As in the previous example let us define the following sets:
$$X_l:=\{y(x)\in\k[x]^m\ / \ f_i(x,y(x))\in (x)^{l+1}\  \forall i\}.$$
As we have done in the previous example with the introduction of the variables $y_{i,j}$, for any $l$ we can embed $X_l$ in $\k^{\N(l)}$ for some integer $N(l)\in\N$. Moreover $X_l$ is an algebraic subset of $\K^{N(l)}$ and the morphisms $\frac{\k[x]}{(x)^{k+1}}\lgw \frac{\k[x]}{(x)^{l+1}}$ for $k\leq l$ induce truncation maps $\pi^k_l : X_k\lgw X_l$ for any $k\geq l$.\\

By the theorem of Chevalley mentioned in the previous example (this is where we need $\k$ to be algebraically closed), for any $l\in\N$, the sequence $(\pi^k_l(X_k))_k$ is a decreasing sequence of constructible subsets of $X_l$. Thus the sequence $(\ovl{\pi_l^k(X_k)})_k$ is a decreasing sequence of algebraic subsets of $X_l$, where $\ovl{Y}$ denotes the Zariski closure of a subset $Y$, i.e. the smallest algebraic set containing $Y$. By Noetherianity this sequence stabilizes: $\ovl{\pi_l^k(X_k)}=\ovl{\pi^{k'}_l(X_{k'})}$ for all $k$ and $k'$ large enough (say for any $k$, $k'\geq k_l$ for some integer $k_l$). Let us denote by $F_l$ this algebraic set. 

Let us assume that $X_k\neq \emptyset$ for any $k\in\N$. This implies that $F_l\neq \emptyset$. Set
$C_{k,l}:=\pi_l^k(X_k)$. It  is a constructible set whose Zariski closure is $F_l$ for any $k\geq k_l$.
Thus $C_{k,l}$ is a finite union of sets of the form $F\backslash V$ where $F$ and $V$ are algebraic sets. Let $F'_l$ be one of the irreducible components of $F_l$ and $C_{k,l}':=C_{k,l}\cap F_l'$. Then $C_{k,l}'$  contains a set of the form $F'_l\backslash V_k$ where $V_k$ is an algebraic proper subset of $F'_l$, for any $k\geq k_l$. 

The set  $U_{l}:=\cap_kC_{k,l}$ contains $\cap _k F'_l\backslash V_k=F'_l\backslash\cup_kV_k$ and the latter set is not empty since $\k$ is uncountable, hence $U_l\neq \emptyset$. By construction $U_{l}$ is exactly the set of points of $X_l$ that can be lifted to points of $X_k$ for any $k\geq l$. In particular $\pi_l^k(U_{k})=U_l$. If $y_0\in U_0$ then $y_0$ may be lifted to $U_1$, i.e.  there exists $y_1\in U_1$ such that $\pi^1_0(y_1)=y_0$. By induction we may construct a sequence of points $y_l\in U_l$ such that $\pi^{l+1}_l(y_{l+1})=y_l$ for any $l\in\N$. At the limit we obtain a point $y_{\infty}$ in $X_{\infty}$, which corresponds to a power series $y(x)\in\k\lb x\rb^m$ solution of $f(x,y)=0$.\\

We have proven here the following result really weaker than Theorem \ref{SAP_intro} but whose proof is very easy (in fact it is given as an exercise in \cite{Ar69} p. 52. See also Lemme 1.6.7 \cite{Ro-thesis} where the above proof is given):

\begin{theorem}\label{ex_Artin}
 If $\k$ is an uncountable algebraically closed field and if $F(x,y)=0$ has solutions modulo $(x)^k$ for every $k\in\N$, then there exists a power series solution $y(x)$: 
$$F(x,y(x))=0.$$
\end{theorem}

This kind of argument using asymptotic constructions (here the Noetherianity is the key point of the proof) may be nicely formalized using ultraproducts. Ultraproducts methods can be used to prove easily  stronger results such as Theorem \ref{SAP_intro} (See Part \ref{ultraproducts} and Proposition \ref{prop_mod}).

\end{example}


\begin{example}[Ideal Membership Problem]\label{IMP}
Let $f_1,\ldots, f_r\in \C\lb x\rb$ be formal power series where $x=(x_1,\ldots,x_n)$. Let us denote by $I$ the ideal of $\C\lb x\rb$ generated by $f_1,\ldots,f_r$. If $g$ is a power series, how can we detect that $g\in I$ or $g\notin I$? Because a power series  is determined by its coefficients, saying that $g\in I$ will depend in general on an infinite number of conditions and it will not be possible to check that all these conditions are satisfied in finite time. Another problem is to find  canonical representatives of power series modulo the ideal $I$ that will enable us to make computations in the quotient ring $\frac{\C\lb x\rb}{I}$.\\
\\
One way to solve these problems is the following. Let us consider the following order on $\N^n$:
for any $\a$, $\b\in\N^n$, we say that $\a\leq \b$ if $(|\a|,\a_1,\ldots,\a_n)\leq_{lex}(|\b|,\b_1,\ldots,\b_n)$ where $|\a|:=\a_1+\cdots+\a_n$ and $\leq_{lex}$ is the lexicographic order. For instance
$$(1,1,1)\leq  (1,2,3)\leq (2,2,2) \leq (3,2,1) \leq (2,2,3).$$
This order induces an order on the sets of monomials $x_1^{\a_1}\ldots x_n^{\a_n}$ as follows: we say that $x^{\a}\leq x^{\b}$ if $\a\leq \b$. Thus 
$$x_1x_2x_3\leq x_1x_2^2x_3^3\leq x_1^2x_2^2x_3^2  \leq x_1^3x_2^2x_3\leq x_1^2x_2^2x_3^3.$$
If $f:=\sum_{\a\in\N^n}f_{\a}x^{\a}\in\C\lb x\rb$, the initial exponent of $f$ with respect to the above order is
$$\exp(f):=\min\{\a\in\N^n\ / \ f_{\a}\neq 0\}=\inf \Supp(f)$$ 
where the support of $f$ is  the set $\Supp(f):=\{\a\in\N^n\ / \ f_{\a}\neq 0\}$. The initial term of $f$ is $f_{\exp(f)}x^{\exp(f)}$. This is the smallest non-zero monomial in the  expansion of $f$ with respect to the above order.\\
If $I$ is an ideal of $\C\lb x\rb$, we define $\G(I)$ to be the subset of $\N^n$ of all the initial exponents of elements of $I$. Since $I$ is an ideal, for any $\b\in\N^n$ and any $f\in I$, $x^{\b}f\in I$. This means that $\G(I)+\N^n=\G(I)$. Then we can prove (this statement is known as Dickson's Lemma) that there exists a finite number of elements $g_1,\ldots, g_s\in I$ such that 
$$\{\exp(g_1),\ldots,\exp(g_s)\}+\N^n=\G(I).$$
Let us mention that Dickson's Lemma is an immediate consequence of the Noetherianity of the polynomial ring in $n$ variables since it translates into saying that the monomial ideal defined by $\Gamma(f)$ is finitely generated.\\
Set
$$\Delta_1:=\exp(g_1)+\N^n\text{ and } \ \Delta_i=(\exp(g_i)+\N^n)\backslash \bigcup_{1\leq j <i}\Delta_j, \text{ for } 2\leq i\leq s.$$
Finally, set
$$\Delta_0:=\N^n\backslash\bigcup_{i=1}^s\Delta_i.$$
For instance, if $I$ is the ideal of $\C\lb x_1,x_2\rb$ generated by $g_1:=x_1x_2^3$ and $g_2:=x_1^2x_2^2$, we can check that 
$$\G(I)=\{(1,3), (2,2)\}+\N^2$$
and the sets $\Delta_0$, $\Delta_1$ and $\Delta_2$ are the following ones:

$$
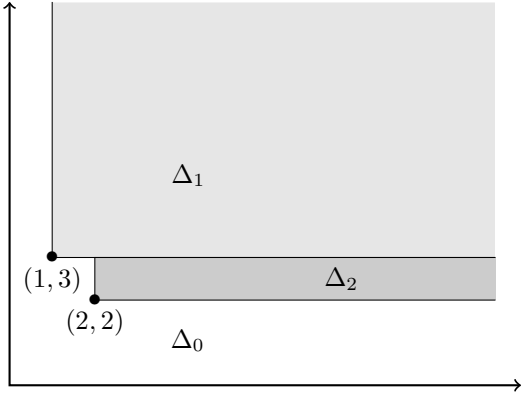
\begin{tikzpicture}[scale=1.7]
   
    \draw [<->,thick] (0,3) node (yaxis) {}
        |- (4,0) node (xaxis) {} ;
        \draw [white, fill=gray!20] (3.8,1) -- (1/3,1) -- (1/3,3) -- (3.8,3) -- (3.8,1) {} ;
   \draw [black]  (3.8,1) -- (1/3,1) -- (1/3,3)  {}  ;
\draw (1.4,1.8) node[below]{$\Delta_1$} ;
     
     \draw [black, fill=gray!40]  (3.8,2/3) -- (2/3,2/3) -- (2/3,1) -- (3.8,1)  {}  ;
\draw (2.8,5/6) node[left]{$\Delta_2$} ;
    \draw (1.6,0.35) node[left]{$\Delta_0$} ; 
    
  \draw (1/3,1)   node {$\bullet$}; 
   \draw (1/3,1)   node[below] {$(1,3)$} ;

      \draw (2/3,2/3)   node {$\bullet$}; 
  \draw (2/3,2/3)   node[below] {$(2,2)$} ;
      \end{tikzpicture}
$$\\

Set $g\in\C\lb x\rb$. Then by  the Galligo-Grauert-Hironaka Division Theorem \cite{Gal} there exist unique power series  $q_1,\ldots,q_s$, $r\in\C\lb x\rb$ such that
\begin{equation}\label{GD1}g=g_1q_1+\cdots+g_sq_s+r\end{equation}
\begin{equation}\label{GD2}\exp(g_i)+\Supp(q_i)\subset \D_i \text{ and } \Supp(r)\subset \Delta_0.\end{equation}
The uniqueness of the division comes from the fact the $\D_i$  are disjoint subsets of $\N^n$. The existence of such a decomposition is proven through the following division algorithm:\\
\\
Set $\a:=\exp(g)$. Then there exists an integer $i_1$ such that $\a\in \D_{i_1}$.\\
$\bullet$ If $i_1=0$, then set $r^{(1)}:=\ini(g)$ and $q_i^{(1)}:=0$ for all $i$.\\
$\bullet$  If $i_1\geq 1$, then set $r^{(1)}:=0$, $q_i^{(1)}:=0$ for $i\neq i_1$ and $q_{i_1}^{(1)}:=\frac{\ini(g)}{\ini(g_{i_1})}$.\\
 Finally set $\displaystyle g^{(1)}:=g-\sum_{i=1}^sg_iq_i^{(1)}-r^{(1)}$. Thus we have $\exp(g^{(1)})>\exp(g)$. Then we replace $g$ by $g^{(1)}$ and we  repeat the above process.\\
  In this way we construct a sequence $(g^{(k)})_k$ of power series such that, for any $k\in\N$, $\exp(g^{(k+1)})>\exp(g^{(k)})$ and $\displaystyle g^{(k)}=g-\sum_{i=1}^sg_iq_i^{(k)}-r^{(k)}$ with
  $$\exp(g_i)+\Supp(q_i^{(k)})\subset \D_i \text{ and } \Supp(r^{(k)})\subset \Delta_0.$$
  At the limit $k\lgw\infty$ we obtain the desired decomposition.\\
  \\   
  In particular since  $\{\exp(g_1),\ldots,\exp(g_s)\}+\N^n=\G(I)$ we deduce from this that $I$ is  generated by $g_1,\ldots,g_s$.\\ 
  \\
  This algorithm  implies that for any $g\in \C\lb x\rb$ there exists a unique power series $r$ whose support is included in $\D_0$ and such that $g-r\in I$ and the division algorithm yields a way to obtain this representative $r$.\\
 Moreover, saying that $g\notin I$ is equivalent to $r\neq 0$ and this is equivalent to saying that, for some integer $k$, $r^{(k)}\neq 0$. But $g\in I$ is equivalent to $r=0$ which is equivalent to $r^{(k)}=0$ for all $k\in\N$. Thus by applying the division algorithm, if for some integer $k$ we have $r^{(k)}\neq0$ we can conclude that $g\notin I$. But this algorithm will not enable us to determine if $g\in I$ since  it requires an infinite number of computations.\\
 \\
Now a natural question is what happens if we replace $\C\lb x\rb$ by $A:=\C\langle x\rangle$ or $\C\{x\}$? Of course we can  proceed with the division algorithm but we do not know if $q_1,\ldots,q_s$, $r\in A$. In fact by controlling the size of the coefficients of $q_1^{(k)},\ldots,q_s^{(k)}$, $r^{(k)}$ at each step of the division algorithm, we can prove that if $g\in\C\{x\}$ then $q_1,\ldots,q_s$ and $r$ remain in $\C\{x\}$ (see \cite{Hi64, Gr72,Gal,JP}).\\
 But if $g\in \C\langle x\rangle$ is an algebraic power series then it may happen that $q_1,\ldots,q_s$ and $r$ are not algebraic power series (see Example \ref{K-G} of Section \ref{constraint}). This is exactly an Artin Approximation problem with constraints in the sense that Equation \eqref{GD1} has formal solutions satisfying the constraints \eqref{GD2} but no algebraic power series solutions satisfying the same constraints. On the other hand, still in the case where $g$ and the $g_i$ are algebraic power series, there exists a computable bound depending on the complexity of $g$ and of the $g_i$ on the number of steps required to conclude if $g\in I$ or not, even if this bound is too large to be used in practice (see \cite{Ro15}).

\end{example}


\begin{example}[Linearization of germs of biholomorphisms]\label{biholo}
Given $f\in\C\{x\}$, $x$ being a single variable, let us assume that $f(0)=0$ and $f'(0)=\lambda \neq0$. By the inverse function theorem $f$ defines a biholomorphism  from a neighborhood of $0$ in $\C$ onto a neighborhood of 0 in $\C$ preserving the origin. The linearization problem, firstly investigated by C. L. Siegel, is the following: is $f$ conjugated to its linear part? That is: does there exist $g(x)\in\C\{x\}$, with $g(0)=0$ and  $g'(0)\neq 0$, such that $f(g(x))=g(\lambda x)$ or $g^{-1}\circ f\circ g(x)=\lambda x$ (in this case we say that $f$ is analytically linearizable)?\\
 This problem is difficult  and the following cases may occur: $f$ is not linearizable, $f$ is formally linearizable but not analytically linearizable (i.e. $g$ exists but $g(x)\in\C\lb x\rb\backslash\C\{x\}$), $f$ is analytically linearizable (see \cite{Ce}).\\
 \\
Let us assume that $f$ is formally linearizable, i.e. there exists $\wdh{g}(x)\in\C\lb x\rb$ such that $f(\wdh{g}(x))-\wdh{g}(\lambda x)=0$. 
By considering the Taylor expansion of $\wdh{g}(\lambda x)$:
$$\wdh{g}(\lambda x)=\wdh{g}(y)+\sum_{n=1}^{\infty}\frac{(\lambda x-y)^n}{n!}\wdh g^{(n)}(y)$$ we see that there  exists $\wdh{h}(x,y)\in\C\lb x,y\rb$ such that 
$$\wdh{g}(\lambda x)=\wdh{g}(y)+(y-\lambda x)\wdh{h}(x,y).$$
Thus if  $f$ is formally linearizable  there exists $\wdh{h}(x,y)\in\C\lb x,y\rb$ such that
$$f(\wdh{g}(x))-\wdh{g}(y)+(y-\lambda x)\wdh{h}(x,y)=0.$$
On the other hand if there exists such an $\wdh h(x,y)$, by replacing $y$ by $\lambda x$ in the previous equation we see that $f$ is formally linearizable.
This former equation is equivalent to the existence of $\wdh{k}(y)\in\C\lb y\rb$ such that
$$\left\{\begin{aligned}
& f(\wdh{g}(x))-\wdh{k}(y)+(y-\lambda x)\wdh{h}(x,y)=0\\
& \wdh{k}(y)-\wdh{g}(y)=0
\end{aligned}\right.$$
Using the same trick as before (Taylor expansion), this is equivalent to the existence of $\wdh{l}(x,y,z)\in\C\lb x,y,z\rb$ such that
\begin{equation}\label{eq_intro}\left\{\begin{aligned}
& f(\wdh{g}(x))-\wdh{k}(y)+(y-\lambda x)\wdh{h}(x,y)=0\\
& \wdh{k}(y)-\wdh{g}(x)+(x-y)\wdh{l}(x,y,z)=0
\end{aligned}\right.\end{equation}
Hence, we see that, if $f$ is formally linearizable, there exists a formal solution $$\left(\wdh{g}(x),\wdh{k}(z),\wdh{h}(x,y),\wdh{l}(x,y,z)\right)$$ of the  system (\ref{eq_intro}). Such a solution is called a solution with constraints. On the other hand, if the system (\ref{eq_intro}) has a convergent solution 
$$(g(x),k(z),h(x,y),l(x,y,z)),$$
 then $f$ is analytically linearizable.\\
\\
We see  that the problem of linearizing analytically $f$ when $f$ is formally linearizable is equivalent to find convergent power series solutions of the system (\ref{eq_intro}) with constraints on the support of the components of the solutions. Because in some cases $f$ may be analytically linearizable but not formally linearizable, such a system (\ref{eq_intro}) may have formal solutions with constraints but no  analytic solutions with the same constraints.\\
In Section \ref{constraint} we will give some results about the Artin Approximation Problem with constraints. 
\end{example}


\begin{example}\label{euler}
Another related problem is the following: if a differential equation with convergent power series coefficients has a formal power series solution, does it have convergent power series solutions? We can also ask the same question by replacing "convergent" by "algebraic".\\
For instance  let us consider the (divergent) formal power series $\displaystyle\wdh{y}(x):=\sum_{n\geq 0}n!x^{n+1}$. It is straightforward to check that it is a solution of the equation
 $$x^2y'-y+x=0\text{ (Euler Equation)}.$$\\
 On the other hand if $\displaystyle\sum_na_nx^n$ is a solution of the Euler Equation then the sequence $(a_n)_n$ satisfies the following recursion:
 $$a_0=0,\ \ a_1=1$$
 $$a_{n+1}=na_n \ \ \forall n\geq 1.$$
 Thus $a_{n+1}=(n+1)!$ for any $n>0$ and $\wdh{y}(x)$ is the only solution of the Euler Equation. Hence we have an example of a differential equation with polynomial coefficients having a formal power series solution but no convergent power series solution. We will discuss in Section \ref{constraint} how to relate this phenomenon to an Artin Approximation problem for polynomial equations with constraints (see Example \ref{diff}).\\
 \\
 \\
 \\
 \\
 \\
 \\
 \\
 \\
 \\
 
\end{example}


 \textbf{Notation:} If $A$ is a local ring, then $\m_A$ will denote its maximal ideal. For any $f\in A$, $f\neq 0$, 
 $$\ord(f):=\max\{n\in\N\ \backslash\ f\in\m_A^n\}.$$
 If $A$ is an integral domain, $\Frac(A)$ denotes its field of fractions.\\
 If no other indication is given the letters  $x$ and $y$ will always denote multivariables, $x:=(x_1,\ldots,x_n)$ and $y:=(y_1,\ldots,y_m)$, and $t$ will denote a single variable. In this case the notation $x'$ will be used to denote the vector $(x_1,\ldots, x_{n-1})$.\\
 If $f(y)$ is a vector of polynomials with coefficients in a ring $A$, $$f(y):=(f_1(y),\ldots,f_r(y))\in A[y]^r,$$   if $\I$ is an ideal of $A$ and $\ovl{y}\in A^m$,  $f(\ovl{y})\in \I$ (resp. $f(\ovl{y})=0$) means $f_i(\ovl{y})\in \I$ (resp. $ f_i(\ovl{y})=0$) for $1\leq i\leq r$.

 \section {Classical Artin Approximation}
 In this  part we review the main results concerning the Artin Approximation Property. We give four results that are the most characteristic ones in the story: the classical Artin Approximation Theorem in the analytic case, its generalization by A. P\l oski, a result of J. Denef and L. Lipshitz  concerning rings with the Weierstrass Division Property and, finally,  the General NŽron Desingularization Theorem.
 
 \subsection{The analytic case}
 In the analytic case the first result is due to Michael Artin in 1968 \cite{Ar68}. His result asserts that the set of convergent solutions is dense in the set of formal solutions of a system of implicit analytic equations. This result is particularly useful, because if we have some analytic problem that we can express in a system of analytic equations, in order to find solutions of this problem we only need to find formal solutions and this may be done in general  by an inductive process. Another way to use this result is the following: let us assume that we have some algebraic problem and that we are working over a ring of the form $A:=\k\lb x\rb$, where $x:=(x_1,\ldots,x_n)$ and $\k$ is a characteristic zero field. If the problem involves only a countable number of data (which is often the case in this context),  since $\C$ is algebraically closed and the transcendence degree of $\Q\lgw \C$ is uncountable, we may assume that we work over $\C\lb x\rb$. Using Theorem \ref{Ar68}, we may, in some cases, reduce the problem to $A=\C\{x\}$. Then we can use powerful methods of complex analytic geometry to solve the problem. This kind of method is used, for instance, in the recent proof of the Nash Conjecture for algebraic surfaces (see  \cite[Theorem A]{FB} and the crucial use of this theorem in \cite{FB-PP}) or in the proof of the Abhyankar-Jung Theorem given  in \cite{P-R}.\\
 Let us mention that  C. Chevalley had apparently  proven this theorem some years before M. Artin but he did not publish it because he did not find applications of it \cite{Ra?}.
  \subsubsection{Artin's  result}

  \begin{theorem}[Analytic Artin Approximation Theorem]\label{Ar68}\cite{Ar68}\index{analytic Artin approximation Theorem}
 Let $\k$ be a valued field, i.e. a field equipped with a multiplicative norm, of characteristic zero and let $f(x,y)$ be a vector of convergent power series in two sets of variables $x$ and $y$. Assume given a formal power series solution $\wdh{y}(x)$ vanishing at $0$,
  $$f(x,\wdh{y}(x))=0.$$
  Then,  for any $c\in\N$,  there exists a  convergent power series \index{convergent power series} solution $\wdt{y}(x)$,
  $$f(x,\wdt{y}(x))=0$$
  which coincides with $\wdh{y}(x)$ up to degree $c$,
  $$\wdt{y}(x)\equiv \wdh{y}(x) \text{ modulo } (x)^c.$$
  
  \end{theorem}
  
  \begin{remark}
  This theorem has been conjectured by S. Lang in \cite{Lan2} (last paragraph p. 372) when $\k=\C$.
  \end{remark}
 
\begin{remark}\label{rem_topology}
The ideal $(x)$ defines a topology on $\k\lb x\rb$, called the \index{Krull topology} \emph{Krull topology},  induced by the following ultrametric norm: $|a(x)|:=e^{-\ord(a(x))}$. In this case the small elements of $\k\lb x\rb$ are the elements of high order. Thus Theorem \ref{Ar68} asserts that the set of solutions in $\k\{x\}^m$ of $f(x,y)=0$ is dense in the set of solutions in $\k\lb x\rb^m$ of $f(x,y)=0$ for the Krull topology.
\end{remark} 

\begin{remark}\label{artin_ideal}
Let $f_1(x,y),\ldots,f_r(x,y)\in\k\{x,y\}$ denote the components of the vector $f(x,y)$. Let $I$ denote the ideal of $\k\{x,y\}$ generated by the $f_i(x,y)$. It is straightforward to see that
$$f_1(x,y(x))=\cdots=f_r(x,y(x))=0 \Longleftrightarrow g(x,y(x))=0\ \ \forall g\in I$$
for any vector $y(x)$ of formal power series vanishing at 0. This shows that Theorem \ref{Ar68} is a statement concerning the ideal generated by the components of the vector $f(x,y)$ and not only these components themselves.
\end{remark}

\begin{proof}[Sketch of proof of Theorem \ref{Ar68}]
Before giving the complete proof of Theorem \ref{Ar68} let us explain the  strategy. \\
The proof is made by induction on $n$. Instead of trying to find directly an approximating convergent solution we will construct a convergent power series $\ovl y(x)$ such that
\begin{equation}\label{intermediate}f(x,\ovl y)\in (\d^2(x,\ovl y))\end{equation}
with $\ovl y_i(x)- \wdh y_i(x)\in (x)^c$ for every $i$ and where $\d$ is a well chosen minor of the Jacobian matrix  $\frac{\partial f}{\partial y}$. Then we will use a generalized version of the Implicit Function Theorem due to J.-C. Tougeron (see Proposition \ref{TougeronIFT} given below) to obtain a convergent solution $\wdt y(x)$ of $f(x,y)=0$ close to $\ovl y(x)$ (and thus close to $\wdh y(x)$).\\
The choice of the minor $\d$ will require some technical reductions involving the Jacobian Criterion.\\
Then to reduce the problem to the situation of \eqref{intermediate} we do the following:\\
We apply a linear change of coordinates $x_j$ in order to assume that $\d^2(x,\wdh y)$ is $x_n$-regular of order $d$ (see Appendix \ref{weierstrass} and Remark \ref{change_x_n-regular}). Thus, by the Weierstrass Preparation Theorem \ref{WDT},  the ideal $(\d^2(x,\wdh y))$ is generated by a Weierstrass polynomial of the form
$$\wdh a(x):=x_n^d+\wdh{a}_1(x')x_{n}^{d-1}+\cdots+\wdh{a}_d(x')$$
where $x'=(x_1,\ldots,x_{n-1})$ and the $\wdh a_i(x')$ are formal power series.
Then we divide each component  $\wdh y_i(x)$ by $\wdh a(x)$. The remainders of these divisions, denoted by 

$$\wdh y_i^*(x):=\sum_{j=0}^{d-1}\wdh y_{i,j}(x')x_n^j,$$
 are polynomials in $x_n$ of degree $<d$ with coefficients in $\k\lb x'\rb$. Since $f(x,\wdh y)=0\in (\wdh a)$  and $\wdh a(x)$ divides the components of $\wdh y(x)-\wdh y^*(x)$,  we have $f(x,\wdh y^*)\in (\wdh a)$.\\
Now we introduce new variables $a_0,\ldots, a_{d-1}$ and $y_{i,j}$ for $1\leq i\leq m$ and $0\leq j <d$. Then we divide the
$$f_k\left(x,\sum_{j=0}^{d-1}y_{i,j}x_n^j\right)$$
by $A(a_i,x_n):=x_{n}^d+a_1x_{n}^{d-1}+\cdots+a_d\in\k[x_{n},a_1,\ldots,a_d]$. We denote by $\sum_{l=0}^{d-1}F_{k,l}x_n^l$, where $F_{k,l}\in\k\{x',y_{i,j},a_p\}$, the remainders of these divisions, so that the relation
$$f(x,\wdh y^*)\in (\wdh a)$$
is equivalent to 
\begin{equation}\label{reduction_intro}F_{k,l}(x',\wdh y_{i,j}(x'),\wdh a_p(x'))=0\ \ \forall k, l.\end{equation}
Hence, by the inductive assumption, we can find a convergent power series solution $(\ovl y_{i,j}(x'),\ovl a_p(x'))$ of \eqref{reduction_intro} close to the given formal solution. This one yields a vector  $\ovl{y}(x)$ of convergent power series such that \eqref{intermediate} holds. In order to prove this, one important point is  to show that the ideal generated by $\d^2(x,\ovl y)$ is equal to the ideal generated by $A(\ovl a_i(x'),x_n)$. This requires to modify a bit the previous argument by dividing $\d^2\left(x,\sum_{j=0}^{d-1}y_{i,j}x_n^j\right)$ by $\wdh a(x)$ and by adding to \eqref{reduction_intro} the condition that the remainder of this division is zero. The details are given in the proof given below.
\end{proof}

 \begin{proof}[Proof of Theorem \ref{Ar68}] 

The proof is done by induction on $n$, the case $n=0$ being obvious since the rings of formal or convergent power series in 0 variables are the field $\k$.\\
 Let us assume that the theorem is proven for  $n-1$ and let us prove it for $n$.\\
 Let us remark that if Theorem \ref{Ar68} is proven for a given integer $c$ then it is obviously true for every integer $\leq c$. This allows us to replace the integer $c$ by any larger integer. In particular at several steps we will  assume that $c$ is an integer larger that some given data independent of $c$.\\
  \\
$\bullet$ Let $I$ be the ideal of $\k\{x,y\}$ generated by $f_1(x,y),\ldots,f_r(x,y)$. Let $\phi$ be the $\k\{x\}$-morphism $\k\{x,y\}\lgw \k\llbracket x\rrbracket$ sending $y_i$ onto $\wdh{y}_i(x)$. Then $\Ker(\phi)$ is a prime ideal containing $I$ and if the theorem is true for generators of $\Ker(\phi)$ then it is true for $f_1,\ldots,f_r$. Thus we can assume that $I=\Ker(\phi)$. \\
\\
$\bullet$  The local ring $\k\{x,y\}_{I}$ is regular by a theorem of Serre (see  \cite[Theorem 19.3]{Ma}). Set $h:=$height$(I)$.  By  the \index{Jacobian Criterion}  Jacobian Criterion (see  \cite[ThŽorme 3.1]{To} or  \cite[Lemma 4.2]{Rui})  there exists a $h\times h$ minor of the Jacobian matrix $\frac{\partial(f_1,\ldots,f_r)}{\partial(x,y)}$, denoted by $\d(x,y)$, such that $\d\notin I=\Ker(\phi)$. In particular we have $\d(x,\wdh{y}(x))\neq 0$.\\
By considering the partial derivative of $f_i(x,\wdh{y}(x))=0$ with respect to $x_j$ we get
$$\frac{\partial f_i}{\partial x_j}(x,\wdh{y}(x))=-\sum_{k=1}^r\frac{\partial \wdh{y}_k(x)}{\partial x_j}\frac{\partial f_i}{\partial y_k}(x,\wdh{y}(x)).$$
Thus there exists a $h\times h$ minor of the Jacobian matrix $\frac{\partial(f_1,\ldots,f_r)}{\partial(y)}$, still denoted by $\d(x,y)$, such that $\d(x,\wdh{y}(x))\neq 0$. In particular $\d\notin I$ and $m\geq h$. From now on we will assume that $\d$ is the determinant of $\frac{\partial(f_1,\ldots,f_h)}{\partial(y_1,\ldots,y_h)}$.\\
If we denote $J:=(f_1,\ldots,f_h)$,  ht$(J\k\{x,y\}_I)\leq h$ by the Krull Haupidealsatz  \cite[Theorem 13.5]{Ma}. On the other hand the Jacobian Criterion  \cite[Proposition 4.3]{Rui} shows that ht$(J\k\{x,y\}_I)\geq $ rk$(\frac{\partial(f_1,\ldots,f_h)}{\partial(y_1,\ldots,y_h)})$ mod. $I$, and $h=$ rk$(\frac{\partial(f_1,\ldots,f_h)}{\partial(y_1,\ldots,y_h)})$ mod. $I$ because $\d(x,\wdh{y}(x))\neq0$.  Hence ht$(J\k\{x,y\}_I)= h$ and $\sqrt{J\k\{x,y\}_I}=I\k\{x,y\}_I$. This means that there exists $q\in\k\{x,y\}$, $q\notin I$, and $e\in\N$ such that $qf_i^e\in J$ for $h+1\leq i\leq r$. In particular $q(x,\wdh{y}(x))\neq 0$. \\
\\
$\bullet$ Let $\wdt y(x)$ be a given convergent solution of $f_1=\cdots=f_h=0$ such that
$$\wdt y(x)-\wdh y(x)\in (x)^c.$$
If $c>\ord(q(x,\wdh{y}(x)))$, then $q(x,\wdt{y}(x))\neq 0$ by Taylor formula. Since $qf_i^e\in J$ for $h+1\leq i\leq r$, this proves that $f_i(x,\wdt{y}(x))=0$ for all $i$ and Theorem \ref{Ar68} is proven. \\
So we can replace $I$ by the ideal generated by $f_1,\ldots,f_h$. \\
Thus from now on we assume that $r=h$ and that  there exists a $h\times h$ minor of the Jacobian matrix $\frac{\partial(f_1,\ldots,f_r)}{\partial(y)}$, denoted by $\d(x,y)$, such that $\d(x,\wdh{y}(x))\neq 0$. We also fix the integer $c$ and assume that $c>\ord(q(x,\wdh{y}(x)))$.
\\
\\
 \begin{lemma}\label{lemma1}
Let us assume that  Theorem \ref{Ar68} is true for  the integer $n-1$. Let $g(x,y)$ be a convergent power series  and let $f(x,y)$ be a vector of convergent power series.\\ Let $\wdh{y}(x)$ be in $(x)\k\lb x\rb^m$ such that $g(x,\wdh{y}(x))\neq 0$ and $f_i(x,\wdh{y}(x))\in (g(x,\wdh{y}(x)))$ for every $i$.\\
Let $c'$ be an integer. Then there exists $\ovl{y}(x)\in(x)\k\{x\}^m$ such that 
$$f_i(x,\ovl{y}(x))\in (g(x,\ovl{y}(x)))\ \ \forall i$$
 and $\ovl{y}(x)-\wdh{y}(x)\in(x)^{c'}$.
\end{lemma}
 $\bullet$ We  apply this lemma to $g(x,y):=\d^2(x,y)$ with the integers $c':=c+d+1$ and $d:=\ord(\d^{2}(x,\wdh{y}(x)))$. Indeed since $f(x,\wdh y(x))=0$ we have $f_i(x,\wdh y(x))\in (\d^2(x,\wdh y(x)))$ for every integer $i$. \\
 Thus we may assume that there are $\ovl{y}_{i}(x)\in\k\{x\}$, $1\leq i\leq m$, such that $f(x,\ovl{y})\in(\d^2(x,\ovl{y}))$ and $\ovl{y}_{i}(x)-\wdh{y}_{i}(x)\in (x)^{c+d+1}$, $1\leq i\leq m$. Because $\ord(\d^2(x,\ovl{y}))=d$ we have that $f(x,\ovl{y})\in \d^2(x,\ovl{y})(x)^c$ by Taylor formula.
 Then we use the following generalization of the Implicit Function Theorem (with $m=h$) to show that there exists $\wdt{y}(x)\in\k\{x\}^m$ with $\wdt{y}(0)=0$ such that $\wdt{y}_j(x)-\wdh{y}_j(x)\in(x)^c$, $1\leq j\leq m$,  and $f_i(x,\wdt{y}(x))=0$ for $1\leq i\leq h$. This proves Theorem \ref{Ar68}.  \end{proof}

\begin{proposition}[Tougeron  Implicit Function Theorem]\label{TougeronIFT}\cite{To}\index{Tougeron  Implicit Function Theorem}
Let $f(x,y)$ be a vector of  $\k\{x,y\}^h$ with $m\geq h$ and let $\d(x,y)$ be  a $h\times h$ minor of the Jacobian matrix $\frac{\partial(f_1,\ldots,f_h)}{\partial(y_1,\ldots,y_m)}$. Let us assume that there exists $y(x)\in\k\{x\}^m$ such that
$$f(x,y(x))\in (\d(x,y(x)))^2(x)^c \text{ for all }1\leq i\leq h$$
and for some $c\in\N$. Then there exists $\wdt{y}(x)\in\k\{x\}^m$ such that 
$$f_i(x,\wdt{y}(x))=0 \text{ for all } 1\leq i\leq h$$
$$\wdt{y}(x)-y(x)\in (\d(x,y(x)))(x)^c.$$
Moreover $\wdt{y}(x)$ is unique if we impose $\wdt{y}_j(x)=y_j(x)$ for $h<j\leq m$.
\\
\end{proposition}
\noindent

Its remains to prove Lemma \ref{lemma1} and Proposition \ref{TougeronIFT}.
\begin{proof}[Proof of Lemma \ref{lemma1}]
If $g(x,\wdh{y}(x))$ is invertible, the result is obvious (just take for $\wdt{y}_i(x)$ any truncation of $\wdh{y}_i(x)$). Thus let us assume that $g(x,\wdh{y}(x))$ is not invertible. By making a linear change of variables  we may assume that $g(x,\wdh{y}(x))$ is $x_n$-regular  (see Remark \ref{change_x_n-regular}), and by the  Weierstrass Preparation Theorem 
$g(x,\wdh{y}(x))=\wdh{a}(x)\times \text{unit}$ where 
$$\wdh{a}(x):=x_{n}^d+\wdh{a}_1(x')x_{n}^{d-1}+\cdots+\wdh{a}_d(x')$$
where $x':=(x_1,\ldots,x_{n-1})$, $d$ is a positive integer and $a_i(x')\in(x')\k\lb x'\rb$, $1\leq i\leq d$.\\
Let us perform the Weierstrass division of $\wdh{y}_i(x)$ by $\wdh{a}(x)$:
\begin{equation}\label{w_i}\wdh{y}_i(x)=\wdh{a}(x)\wdh{w}_i(x)+\sum_{j=0}^{d-1}\wdh{y}_{i,j}(x')x_{n}^j\end{equation}
for $1\leq i\leq m$. We set
$$\wdh{y}^*_i(x):=\sum_{j=0}^{d-1}\wdh{y}_{i,j}(x')x_{n}^j,\ \ 1\leq i\leq m.$$
Then by the Taylor formula 
$$g(x,\wdh{y}(x))=g(x,\wdh{y}^*(x)) \text{ mod. }\wdh{a}(x)$$ and 
$$f_k(x,\wdh{y}(x))=f_k(x,\wdh{y}^*(x)) \text{ mod. }\wdh{a}(x)$$ for $1\leq k\leq r$. Thus
\begin{equation}\label{Artin3}g(x,\wdh{y}^*(x))=f_k(x,\wdh{y}^*(x))=0 \text{ mod. } \wdh a(x).\end{equation}
\\
Let $y_{i,j}$, $1\leq i\leq m$, $0\leq j\leq d-1$, be new variables. We define $y^*_i:=\sum_{j=0}^{d-1}y_{i,j}x_{n}^j$, $1\leq i\leq m$. Let us define the polynomial
$$A(a_i,x_{n}):=x_{n}^d+a_1x_{n}^{d-1}+\cdots+a_d\in\k[x_{n},a_1,\ldots,a_d]$$
where $a_1,\ldots,a_d$ are new variables. Let us perform the Weierstrass division of $g(x,y^*)$ and $f_i(x,y^*)$ by $A$:
\begin{equation}\label{Artin1}g(x,y^*)=A.Q+\sum_{l=0}^{d-1}G_lx_{n}^l\end{equation}
\begin{equation}\label{Artin2}f_k(x,y^*)=A.Q_k+\sum_{l=0}^{d-1}F_{k,l}x_{n}^l,\ \ 1\leq k\leq r\end{equation}
 where $Q$, $Q_k\in\k\{x, y_{i,j}, a_p\}$ and $G_l$, $F_{k,l}\in\k\{x',y_{i,j},a_p\}$. \\
 Then we have 
 $$g(x,\wdh{y}^*(x))=\sum_{l=0}^{d-1}G_l(x',\wdh{y}_{i,j}(x'),\wdh{a}_p(x'))x_{n}^l \text{ mod. } (\wdh{a}(x))$$
 $$f_k(x,\wdh{y}^*(x))=\sum_{l=0}^{d-1}F_{k,l}(x',\wdh{y}_{i,j}(x'),\wdh{a}_p(x'))x_{n}^l \text{ mod. } (\wdh{a}(x)),\ \ 1\leq k\leq r.$$
Hence \eqref{Artin3} shows that 
$$G_l(x',\wdh{y}_{i,j}(x'),\wdh{a}_p(x'))=0$$ and $$F_{k,l}(x',\wdh{y}_{i,j}(x'),\wdh{a}_p(x'))=0$$ for all $k$ and $l$. By the inductive hypothesis, there exist convergent power series $\ovl{y}_{i,j}(x')\in\k\{x'\}$ and $\ovl{a}_p(x')\in\k\{x'\}$ for all $i$, $j$ and $p$, such that 
$$G_l(x',\ovl{y}_{i,j}(x),\ovl{a}_p(x'))=0 \text{ and }F_{k,l}(x',\ovl{y}_{i,j}(x'),\ovl{a}_p(x'))=0$$
 for all $k$ and $l$, and $\ovl{y}_{i,j}(x')-\wdh{y}_{i,j}(x')$, $\ovl{a}_p(x')-\wdh{a}_p(x')\in(x')^c$ for all $i$, $j$ and $p$ \footnote{Formally in order to apply the induction hypothesis we should have $\wdh{y}_{i,j}(0)=0$ and $\wdh{a}_p(0)=0$ which is not necessarily the case here. We can remove the problem by replacing $\wdh{y}_{i,j}(x')$ and $\wdh{a}_p(x')$ by $\wdh{y}_{i,j}(x')-\wdh{y}_{i,j}(0)$ and $\wdh{a}_p(x')-\wdh{a}_p(0)$, and $G_l(x',y_{i,j},a_p)$ by $G(x',y_{i,j}+\wdh{y}_{i,j}(0),a_p+\wdh{a}_p(0))$ - idem for $F_{k,l}$. We skip the details here.}.\\
\\
Now let us set
$$\ovl{a}(x):=x_{n}^d+\ovl{a}_1(x')x_{n}^{d-1}+\cdots+\ovl{a}_d(x')$$
$$\ovl{y}_i(x):=\ovl{a}(x)\ovl{w}_i(x)+\sum_{j=0}^{d-1}\ovl{y}_{i,j}(x')x_{n}^j$$
for some $\ovl{w}_i(x)\in\k\{x\}$ such that $\ovl{w}_i(x)-\wdh{w}_i(x)\in (x)^c$ for all $i$ (see \eqref{w_i}). 
It is straightforward to check that $\ovl{y}_j(x)-\wdh{y}_j(x)\in(x)^c$ for $1\leq j\leq m$.
If $c>d$, the Taylor formula shows that
$$g(x,\ovl y(x))-g(x,\wdh y(x))\in (x)^c\subset (x)^{d+1}.$$
Thus 
$$g(0,\ldots,0,x_n,\ovl y(0,\ldots,0,x_n))-g(0,\ldots,0,x_n,\wdh y(0,\ldots,0,x_n)) \in (x_n)^{d+1}.$$
Since the order of the power series $g(0,\ldots,0,x_n,\wdh y(0,\ldots,0,x_n))$ is $d$ this implies that the order of $g(0,\ldots,0,x_n,\ovl y(0,\ldots,0,x_n))$ is also  $d$. But $\ovl a(x)$ divides $g(x,\ovl y(x))$ and it is a Weierstrass polynomial of degree $d$. So the Weierstrass Division Theorem implies that $g(x,\ovl y(x))$ equals $\ovl a(x)$ times a unit. Since $f(x,\ovl y(x))\in (\ovl a(x))$ by \eqref{Artin2} we have
$$f(x,\ovl{y}(x))=0 \text{ mod. }g(x,\ovl{y}(x)).$$

\end{proof}

\begin{proof}[Proof of Proposition \ref{TougeronIFT}] 
We may assume that $\d$ is the first $h\times h$ minor of the Jacobian matrix. If we add the equations $f_{h+1}:=y_{h+1}-\wdt{y}_{h+1}(x)=0$, \ldots,  $f_m:=y_m-\wdt{y}_{m}(x)=0$, we may assume that $m=h$ and $\d$ is the determinant of the Jacobian matrix $J(x,y):=\frac{\partial(f_1,\ldots,f_h)}{\partial(y)}$. We have
$$f\left(x,y(x)+\d(x,y(x))z\right)=f(x,y(x))+\d(x,y)J(x,y(x))z+\d(x,y(x))^2H(x,y(x),z)$$
where $z:=(z_1,\ldots,z_m)$ and $H(x,y(x),z)\in\k\{x,y(x),z\}^{m}$ is of order at least 2 in $z$. Let us denote by $J'(x,y(x))$ the adjoint matrix of $J(x,y(x))$. Let $\e(x)$ be in $(x)^c\k\{x\}^r$ such that $f(x,y(x))=\d^2(x,y(x))\e(x)$. Then we have
\begin{equation*}\begin{split}f(x,y(x)+\d(x,y(x))z)=& \\
=\d(x,y(x))J(x,y(x))&\left[J'(x,y(x))\e(x)+z+J'(x,y(x))H(x,y(x),z)\right].\end{split}\end{equation*}
We define
$$g(x,z):=J'(x,y(x))\e(x)+z+J'(x,y(x))H(x,y(x),z).$$
Then $g(0,0)=0$ and  the  matrix $\frac{\partial g(x,z)}{\partial z}(0,0)$ is the identity matrix. Thus, by the Implicit Function Theorem, there exists a unique $z(x)\in\k\{x\}^m$ such that 
$$f(x,y(x)+\d(x,y(x))z(x))=0.$$ This proves the proposition.
 \end{proof}

\begin{remark}\label{rmkAr68}
We make the following remarks about the proof of Theorem \ref{Ar68}:
\begin{enumerate}
\item[i)] In the case  $n=1$ i.e. $x$ is a single variable, set $e:=\ord(\d(x,\wdh{y}(x)))$. If $\ovl{y}(x)\in\k\{x\}^m$ satisfies $\wdh{y}(x)-\ovl{y}(x)\in (x)^{2e+c}$,  then we have $$\ord(f(x,\ovl{y}(x)))\geq 2e+c$$ and $$\d(x,\ovl{y}(x)) =\d(x,\wdh{y}(x)) \text{ mod. }(x)^{2e+c},$$ thus $\ord(\d(x,\ovl{y}(x)))=\ord(\d(x,\wdh{y}(x)))=e$. Hence we have automatically 
$$f(x,\ovl{y}(x))\in(\d(x,\ovl{y}(x)))^2(x)^c$$
  since $\k\{ x\}$ is a discrete valuation ring (i.e. if $\ord(a(x))\leq \ord(b(x))$ then $a(x)$ divides $b(x)$ in $\k\{ x\}$).\\ Thus Lemma \ref{lemma1} is not necessary in this case and the proof is more simple. This fact will be general: approximation results will be easier to obtain, and sometimes stronger,  in discrete valuation rings than in more general rings.
  
\item[ii)] We did not really use the fact that $\k$ is a field of characteristic zero, we just need $\k$ to be a perfect field in order to use the \index{Jacobian Criterion} Jacobian Criterion. But the use of the Jacobian Criterion is more delicate for imperfect fields. This also will be general: approximation results will be more difficult to prove in positive characteristic. For instance M. Andr\'e proved Theorem \ref{Ar68} in the case where $\k$ is a complete valued field of  positive characteristic and replaced the use of the Jacobian Criterion by  the \index{AndrŽ homology} homology of commutative algebras \cite{An}. In fact it is proven that Theorem \ref{Ar68} is satisfied for a field $\k$ if and only if the completion of $\k$ is separable over $\k$ \cite{Schemmel}.

\item[iii)] For $n\geq 2$, the proof of Theorem \cite{Ar68} uses  induction on $n$. In order to do it we use the Weierstrass Preparation Theorem. But to apply the Weierstrass Preparation Theorem we need to make a linear change of coordinates in $\k\{x\}$, in order to transform $g(x,\wdh{y}(x))$ into a power series $h(x)$ such that $h(0,\ldots,0,x_n)\neq 0$. Because of this change of coordinates the proof does not adapt to prove similar results in the case of constraints: for instance if  $\wdh{y}_1(x)$ depends only on $x_1$ and $\wdh{y}_2(x)$ depends only on $x_2$, can we find a convergent solution such that $\wdt{y}_1(x)$ depends only on $x_1$, and $\wdt{y}_2(x)$ depends only on  $x_2$?\\
Moreover, even if we were able to make a linear change of coordinates without modifying the constraints, the use of the Tougeron  Implicit Function Theorem may remove the constraints. We will discuss these problems in Section \ref{constraint}.

\end{enumerate}
\end{remark}

\begin{remark}
The Tougeron Implicit Function Theorem has strong applications for finding normal forms of power series, i.e. to finding local coordinates $x_1',\ldots, x'_n$ such that a given power series is a polynomial in some of these new coordinates (see \cite{To0} or \cite{Ku}). Some generalizations of this theorem have been proven in \cite{BK}.
\end{remark}

 \begin{corollary}\label{Ar68cor}
 Let $\k$ be a valued field of characteristic zero and let $I$ be an ideal of $\k\{x\}$. Let $A$ denote the local ring $\frac{\k\{x\}}{I}$, $\m_A$ its maximal ideal and $\wdh A$ its completion.\\
  Let $f(y)\in \k\{x,y\}^r$ and
 $\wdh{y}\in\wdh A ^m$ be a solution of $f=0$ in $A$ such that  $\wdh{y}\in \m_A\wdh A$.  Then there exists a solution $\wdt y$ of $f=0$ in $A$  such that $\wdt{y}\in \m_A$  and  $\wdt{y}-\wdh{y}\in\m_A^c\wdh A$.
  \end{corollary}
\begin{proof}
 Let $a_1,\ldots,a_s\in\k\{x\}$ be generators of $I$. Let us choose $\wdh{w}(x)\in\k\lb x\rb^m$ such that $\wdh{w}_j(x)=\wdh{y}_j$ mod. $I$ for $1\leq j\leq m$. Since $f_i(\wdh{y})=0$ in $A$  there exist $\wdh{z}_{i,k}(x)\in\k\lb x\rb$, $1\leq i\leq r$ and $1\leq k\leq s$, such that
$$f_i(x,\wdh{w}(x))+a_1 \wdh{z}_{i,1}(x)+\cdots+a_s\wdh{z}_{i,s}(x)=0\ \ \forall i.$$
By Theorem \ref{Ar68} there exist $\wdt{w}_j(x)$, $\wdt{z}_{i,k}(x)\in\k\{x\}$ such that
$$f_i(x,\wdt{w}(x))+a_1 \wdt{z}_{i,1}(x)+\cdots+a_s\wdt{z}_{i,s}(x)=0\ \ \forall i$$
and $\wdh{w}_j(x)-\wdt{w}_j(x)\in (x)^c$ for $1\leq j\leq m$.
Then the images of the $\wdt{w}_j(x)$  in $\frac{\k\{x\}}{I}$ satisfy the conclusion of the corollary.
 \end{proof}


 \subsubsection{P\l oski's  Theorem}
 For his PhD thesis, a few years after M. Artin  result, A. P\l oski strengthened Theorem \ref{Ar68} by a careful analysis of the proof and a smart modification of it. His result yields an analytic parametrization of a piece of the set of solutions of $f=0$ such that the formal solution $\wdh{y}(x)$ is a formal point of this parametrization.
 
  \begin{theorem}[P\l oski's Theorem]\label{Pl}\cite{Pl,Pl15}\index{P\l oski's Theorem}
  Let $\k$ be a valued field of characteristic zero and let $f(x,y)$ be a vector of convergent power series  in two sets of variables $x$ and $y$.
 Let $\wdh{y}(x)$ be a formal power series solution
with $\wdh{y}(0)=0$,
 $$f(x,\wdh{y}(x))=0.$$ Then there is  a convergent power series solution $y(x,z)\in\k\{x,z\}^m$ with $y(0,0)=0$, where $z=(z_1,\ldots,z_s)$ are new variables, 
 $$f(x,y(x,z))=0,$$
 and a vector of formal power series $\wdh{z}(x)\in\k\llbracket x\rrbracket  ^s$ with $\wdh{z}(0)=0$ such that 
 $$\wdh{y}(x)=y(x,\wdh{z}(x)).$$
 \end{theorem}
 
\begin{remark}\label{Pl->app} This result obviously implies Theorem \ref{Ar68} because we can choose convergent power series $\wdt{z}_1(x),\ldots,\wdt{z}_s(x)\in\k\{x\}$ such that $\wdt{z}_j(x)-\wdh{z}_j(x)\in (x)^c$ for $1\leq j\leq s$. Then, by denoting $\wdt{y}(x):=y(x,\wdt{z}(x))$, we get the conclusion of Theorem \ref{Ar68}.\\
\end{remark}

\begin{example}\label{flat_ploski}
Let $T$ be a $p\times m$ matrix whose entries are in $\k\{x\}$ and let $b\in \k\{x\}^p$ be a vector of convergent power series. Let $\wdh y(x)$ be a formal power series vector solution of the following system of linear equations:
\begin{equation}\label{lineq}Ty=b.\end{equation}
By \index{faithful flatness} the faithful flatness of $\k\lb x\rb$ over $\k\{x\}$ (see Example \ref{faithfully_flat} of the introduction) there exists a convergent power series vector solution of \eqref{lineq} denoted by $y^0(x)$. Let $M$ be the (finite) $\k\{x\}$-submodule of $\k\{x\}^m$ of convergent power series solutions of the associated homogeneous linear system:
$$Ty=0.$$
Then by the flatness  of $\k\lb x\rb$ over $\k\{x\}$ (see Example \ref{flat_ana} of the introduction) the set of formal power series solutions is the set of linear combinations of elements of $M$ with coefficients in $\k\lb x\rb$. Thus if $m_1(x),\ldots,m_s(x)$ are generators of $M$ there exist formal power series $\wdh z_1(x),\ldots,\wdh z_s(x)$ such that
$$\wdh y(x)-y^0(x)=\wdh z_1(x)m_1(x)+\cdots+\wdh z_s(x) m_s(x).$$
We define
$$y(x,z):=y^0(x)+\sum_{i=1}^sm_i(x)z_i$$
and Theorem \ref{Pl} is proven in the case of systems of linear equations.\\
\end{example}
 
  \begin{proof}[Sketch   of the proof of Theorem \ref{Pl}] 
  The proof is very similar to the proof of Theorem \ref{Ar68}. It is also an induction on $n$. The beginning of the proof is the same, so we can assume that $r=h$ and  we need to prove an analogue of Lemma \ref{lemma1} with parameters for $g=\d^2$ where $\d$ is the first $h\times h$ minor of the jacobian matrix $\frac{\partial f}{\partial y}$. But in order to prove it  we need to make a slight but crucial modification in the proof. First we  make a linear change of variables and assume that $\d(x,\wdh{y}(x))$ is regular with respect to $x_n$, i.e.
  $$\d(x,\wdh{y}(x))=(x_{n}^d+\wdh{a}_1(x')x_{n}^{d-1}+\cdots+\wdh{a}_d(x'))\times \text{unit}.$$  
  We set
  $$\wdh{a}(x):=x_{n}^d+\wdh{a}_1(x')x_{n}^{d-1}+\cdots+\wdh{a}_d(x').$$
  (in the proof of Theorem \ref{Ar68}, $\wdh{a}(x)$ denotes the square of  this Weierstrass polynomial!)\\
  Then the idea of P\l oski is to perform the Weierstrass division of  $\wdh{y}_i(x)$ by $\wdh{a}(x)$ for $1\leq i\leq h$ and by $\wdh{a}(x)^2$ for $h<i\leq m$:
  \begin{equation}\label{z_i1}\wdh{y}_i(x)=\wdh{a}(x)\wdh{z}_i(x)+\sum_{j=0}^{d-1}\wdh{y}_{i,j}(x')x_{n}^j,\ \ 1\leq i\leq h,\end{equation}
    \begin{equation}\label{z_i2}\wdh{y}_i(x)=\wdh{a}(x)^2\wdh{z}_i(x)+\sum_{j=0}^{2d-1}\wdh{y}_{i,j}(x')x_{n}^j, \ \ Êh<i\leq m.\end{equation}
Let us define
$$\wdh{y}^*_i(x):=\sum_{j=0}^{d-1}\wdh{y}_{i,j}(x')x_{n}^j,\ \ 1\leq i\leq h,$$

$$\wdh{y}^*_i(x):=\sum_{j=0}^{2d-1}\wdh{y}_{i,j}(x')x_{n}^j,\ \ h< i\leq m.$$
Let $M(x,y)$ denote the adjoint matrix of $\frac{\partial(f_1,\ldots,f_h)}{\partial(y_1,\ldots,y_h)}$:
 $$M(x,y) \frac{\partial(f_1,\ldots,f_h)}{\partial(y_1,\ldots,y_h)}=    \frac{\partial(f_1,\ldots,f_h)}{\partial(y_1,\ldots,y_h)}M(x,y)=\d(x,y)I_h$$
  where $I_h$ is the identity matrix of size $h\times h$.  Then we define
  $$g(x,y):=M(x,y)f(x,y)=(g_1(x,y),\ldots,g_h(x,y))$$ where $g$ and $f$ are considered as column vectors. We have
   \begin{equation*}
  \begin{split}0=f(x,\wdh{y}(x))=f\big(x,\wdh{y}_1^*(x)+\wdh{a}(x)\wdh{z}_{1}(x),\ldots,&\wdh{y}_h^*(x)+\wdh{a}(x)\wdh{z}_{h}(x),\\
    \wdh{y}_{h+1}^*(x)+\wdh{a}(x)^2\wdh{z}_{h+1}(x)&,\ldots,\wdh{y}_m^*(x)+\wdh{a}(x)^2\wdh{z}_{m}(x)\big)=\end{split}
  \end{equation*}
   \begin{equation*}
  \begin{split}=f(x,\wdh{y}^*(x))+\wdh{a}(x) \frac{\partial(f_1,\ldots,f_h)}{\partial(y_{1},\ldots,y_h)}(x,\wdh{y}^*(x))&\left(\begin{array}{c}\wdh{z}_{1}(x)\\
  \vdots\\
  \wdh{z}_h(x)\end{array}\right)+\\
  +\wdh{a}(x)^2 \frac{\partial(f_1,\ldots,f_h)}{\partial(y_{h+1},\ldots,y_m)}(x,\wdh{y}^*(x))&\left(\begin{array}{c}\wdh{z}_{h+1}(x)\\
  \vdots\\
  \wdh{z}_{m}(x)\end{array}\right)+\wdh{a}(x)^2Q(x)   \end{split}
  \end{equation*}
  for some $Q(x)\in \k\lb x\rb^h$.
  Hence $g_k(x,\wdh{y}^*(x))\in (\wdh{a}(x)^2)$ since $\d$ is the determinant of $\frac{\partial(f_1,\ldots,f_h)}{\partial(y_{1},\ldots,y_h)}$. As in the proof of Theorem \ref{Ar68} we have $\d(x,\wdh{y}^*(x))\in(\wdh{a}(x))$.\\
  We assume that Theorem  \ref{Pl} is proven for $n-1$ variables. Thus we can imitate  the proof of Lemma \ref{lemma1} to show that there exist  convergent power series  $\ovl{y}_{i,j}(x',t)$, $\ovl{a}_p(x',t)\in\k\{x,t\}$, $t=(t_1,\ldots,t_s)$, such that $\wdh{y}_{i,j}(x')=\ovl{y}_{i,j}(x',\wdh{t}(x'))$ and $\wdh{a}_p(x')=\ovl{a}_p(x',\wdh{t}(x'))$ for some $\wdh{t}(x')\in\k\lb x'\rb^s$ and
  $$g\left(x,\ovl{y}^*(x,t)\right)\in (\ovl{a}(x,t)^2)$$  
  $$f(x,\ovl{y}^*(x,t))\in (g\left(x,\ovl{y}^*(x,t)\right))$$
  with 
  $$\ovl{a}(x,t):=x_{n}^d+\ovl{a}_1(x',t)x_{n}^{d-1}+\cdots+\ovl{a}_d(x',t),$$
   $$\ovl{y}_i^*(x,t):=\sum_{j=0}^{d-1}\ovl{y}_{i,j}(x',t)x_{n}^j\text{ for }1\leq i\leq h,$$
   $$\ovl{y}_i^*(x,t):=\sum_{j=0}^{2d-1}\ovl{y}_{i,j}(x',t)x_{n}^j\text{ for }h< i\leq m.$$
   Moreover $\ovl a(x,t)$ is the Weierstrass polynomial of $\d(x,\ovl{y}^*(x,t))$.\\
  
   Let $z:=(z_1,\ldots,z_h)$ and $z':=(z'_{h+1},\ldots,z'_{m})$ be two vectors of new variables. 
  Let us define
$$\ovl{y}_i(x,t,z_i):=\ovl{a}(x,t)z_i+\sum_{j=0}^{d-1}\ovl{y}_{i,j}(x',t)x_{n}^j\ \text{ for } 1\leq i\leq h,$$
$$y_i(x,t,z_i'):=\ovl a(x,t)^2 z'_{i}+\sum_{j=0}^{2d-1}\ovl{y}_{i,j}(x',t)x_{n}^j,\ \ h<i\leq m.$$
  Then we use the following proposition similar to Proposition \ref{TougeronIFT} whose proof is given below:
  
    \begin{proposition}\label{PloskiIFT}\cite{Pl2}
   
 With the above notation and assumptions there exist convergent power series 
 $$\ovl{z}_i(x,t,z')\in\k\{x,t,z'\},\ \ 1\leq i\leq h,$$ such that 
 
  $$f(x,\ovl y_1(x,t,\ovl{z}_1(x,t,z')),\ldots,\ovl y_h(x,t,\ovl{z}_h(x,t,z')), y_{h+1}(x,t,z'),\ldots,y_m(x,t,z'))=0.$$
  Moreover there exists a vector formal power series $\wdh z'(x)$ such that 
  $$\ovl y_i(x,\wdh t(x'),\ovl z_i (x,\wdh t(x'),\wdh z'(x)))=\wdh y_i(x) \ \text{ for } 1\leq i\leq h,$$
  $$y_i(x,\wdh t(x'),\wdh z'_i(x))=\wdh y_i(x)\text{ for }h< i\leq m.$$\\
  \end{proposition}
  
  Thus we apply this proposition and we define
  $$y_i(x,t,z'):=\ovl y_i(x,t,\ovl z_i(x,t,z'))=\ovl{a}(x,t)\ovl{z}_i(x,t,z')+\sum_{j=0}^{d-1}\ovl{y}_{i,j}(x',t)x_{n}^j\ \text{ for } 1\leq i\leq h$$
so that we have
  $$f(x,y(x,t,z'))=0,$$
  $$\ovl y_i(x,\wdh t(x'),\ovl z_i (x,\wdh t(x'),\wdh z'(x)))=\wdh y_i(x) \ \text{ for } 1\leq i\leq h,$$
  $$y_i(x,\wdh t(x'),\wdh z_i'(x))=\wdh y_i(x)\text{ for }h< i\leq m.$$
  This achieves the proof of Theorem \ref{Pl} with $z=(t,z')$ and 
  $$y(t,z')=(\ovl y_1(x,t,\ovl z_1(x,t,z')),\ldots, \ovl y_h(x,t,\ovl z_h(x,t,z')), y_{h+1}(x,t,z'),\ldots, y_m(x,t,z')).$$   
  
  \end{proof}
    
  \begin{proof}[Proof of Proposition \ref{PloskiIFT}]
  We prove first the existence of the convergent power series $\ovl z_i(x,t,z')$. We have
  \begin{equation*}
  \begin{split}F(x,t,z',z):=f\big(x,\ovl{y}_1^*(x,t)+\ovl a(x,t)z_{1},&\ldots,\ovl{y}_h^*(x,t)+\ovl a(x,t)z_{h},\\
    \ovl{y}_{h+1}^*(x,t)+\ovl a(x,t)^2&z'_{h+1},\ldots,\ovl{y}_m^*(x,t)+\ovl a(x,t)^2z'_{m}\big)=\end{split}
  \end{equation*}
   \begin{equation*}
  \begin{split}=f(x,\ovl{y}^*(x,t))+\ovl a(x&,t)^2 \frac{\partial(f_1,\ldots,f_h)}{\partial(y_{h+1},\ldots,y_m)}(x,\ovl{y}^*(x,t))\left(\begin{array}{c}z'_{h+1}\\
  \vdots\\
  z'_{m}\end{array}\right)+\\
 +\ovl a(x,t) &\frac{\partial(f_1,\ldots,f_h)}{\partial(y_{1},\ldots,y_h)}(x,\ovl{y}^*(x,t))\left(\begin{array}{c}z_{1}\\
  \vdots\\
  z_h\end{array}\right)+\ovl a(x,t)^2Q(x,t,z',z)   \end{split}
  \end{equation*}
where the entries of the vector $Q(x,t,z',z)$ are in $(x,t,z',z)^2$.\\
 Since $\ovl a(x,t)$ is equal to $\d(x,\ovl{y}^*(x,t))$ times a unit, by multiplying on the left  this equality by $M(x,\ovl{y}^*(x,t))$ we obtain that 
 $$M(x,\ovl{y}^*(x,t))F(x,t,z',z)=\d^2(x,\ovl{y}^*(x,t))G(x,t,z',z),$$ 
 where the entries of the vector $G(x,t,z',z)$ are convergent  series and $G(0,0,0,0)=0$.
By differentiation this equality yields
$$M(x,\ovl{y}^*(x,t))\frac{\partial(F_1,\ldots,F_h)}{\partial(z_{1},\ldots,z_h)}(x,t,z',z)=\d^2(x,\ovl{y}^*(x,t))\frac{\partial(G_1,\ldots,G_h)}{\partial(z_{1},\ldots,z_h)}(x,t,z',z).$$
It is easy to check that
\begin{equation*}\begin{split}\text{det}\left(\frac{\partial(F_1,\ldots,F_h)}{\partial(z_{1},\ldots,z_h)}\right)(x,0,0&,0)=\\
=\text{det}\left(\frac{\partial(f_1,\ldots,f_h)}{\partial(y_{1},\ldots,y_h)}\right)&(x,0,0,0)\ovl a(x,0)^h=\d(x,\ovl{y}^*(x,0))^{h+1}\times\text{ unit}.\end{split}\end{equation*}
But det$(M(x,\ovl{y}^*(x,t)))=\d(x,\ovl{y}^*(x,t))^{h-1}$ thus $\text{det}\left(\frac{\partial(G_1,\ldots,G_h)}{\partial(z_{1},\ldots,z_h)}\right)(x,0,0,0)$ is a unit. 
Hence $\text{det}\left(\frac{\partial(G_1,\ldots,G_h)}{\partial(z_{1},\ldots,z_h)}\right)(0,0,0,0)\neq 0$.
So  the Implicit Function Theorem yields unique convergent power series $\ovl z_i(x,t,z')\in\k\{x,t,z'\}$, $ 1\leq i\leq h$, vanishing at 0 such that 
$$G(x,t,z',\ovl z(x,t,z'))=0.$$ This shows $F(x,t,z',\ovl z(x,t,z'))=0$. \\
\\
In order to prove the existence of the formal power series $\wdh z'(x)$ we make the same computation where $t$ is replaced by $\wdh t(x')$. Thus by the Implicit Function Theorem there exist unique power series $\wdt z_i(x,z')\in\k\lb x,z'\rb$, for $1\leq i\leq h$, vanishing at 0 such that
$$G(x,\wdh t(x'),z',\wdt z(x,z'))=0$$
i.e.
\begin{equation*}\begin{split}f(x,\ovl y_1(x,\wdh t(x'),\wdt z_1(x,z')),\ldots,\ovl y_h(x,\wdh t(x'),\wdt z_h(x,z'))&,\\
y_{h+1}(x,\wdh t(x'),&z'),\ldots,y_m(x,\wdh t(x'),z'))=0.\end{split}\end{equation*}
Thus by uniqueness we have 
$$\ovl z(x,\wdh t(x'),z')=\wdt z (x,z').$$
Moreover, again by the Implicit  Function Theorem, any vector of formal power series $\wdh z(x)$ vanishing at the origin is a solution of the equation
\begin{equation}\label{G=0}G(x,\wdh t(x'),z',z)=0\end{equation}
if and only if  there exists a vector of formal power series $\wdh z'(x)$ such that
$$\wdh z(x)=\wdt z(x,\wdh z'(x)).$$
In particular, because the vector $\wdh z(x)$ defined by \eqref{z_i1} and \eqref{z_i2} is a solution of \eqref{G=0}, there exists a vector of formal power series $\wdh z'(x)$ such that 
 $$\wdh{y}_i(x)=     \ovl{a}(x,\wdh t(x'))z_i(x,\wdh z'(x))+\sum_{j=0}^{d-1}\ovl{y}_{i,j}(x',\wdh t(x'))x_{n}^j\ \text{ for } 1\leq i\leq h,$$
    $$\wdh{y}_i(x)=\ovl a(x,\wdh t(x'))^2\wdh z'_{i}(x)+\sum_{j=0}^{2d-1}\ovl{y}_{i,j}(x',\wdh t(x'))x_{n}^j, \ \ Êh<i\leq m.$$ 
  
  \end{proof}

 \begin{remark}\label{Pl_quotient}
 Let us remark that this result remains true if we replace $\k\{x\}$ by a quotient $\frac{\k\{x\}}{I}$ as in Corollary \ref{Ar68cor}.
 \end{remark}

\begin{remark}\label{morphism}
 Let $I$ be the ideal generated by $f_1,\ldots,f_r$. The formal solution $\wdh{y}(x)$ of $f=0$ induces a $\k\{x\}$-morphism $\k\{x,y\}\lgw \k\llbracket x\rrbracket$ defined by the substitution of $\wdh{y}(x)$ for $y$. Then $I$ is included in the kernel of this morphism thus, by the universal property of the quotient ring, this morphism induces a $\k\{x\}$-morphism 
 $\psi : \frac{\k\{x,y\}}{I}\lgw \k\llbracket x\rrbracket$. On the other hand, any $\k\{x\}$-morphism $\psi :  \frac{\k\{x,y\}}{I}\lgw \k\llbracket x\rrbracket$ is clearly defined by substituting for $y$ a vector of formal power series $\wdh{y}(x)$ such that $f(x,\wdh{y}(x))=0$.\\
 \\
 Thus we can reformulate Theorem \ref{Pl} as follows: Let $\psi : \frac{\k\{x,y\}}{I}\lgw \k\llbracket x\rrbracket$ be the $\k\{x\}$-morphism defined by the formal power series solution $\wdh{y}(x)$. Then there exist an analytic smooth $\k\{x\}$-algebra $D:=\k\{x,z\}$ and $\k\{x\}$-morphisms $C\lgw D$ (defined via the convergent power series solution $y(x,z)$ of $f=0$) and $D\lgw \k\llbracket x\rrbracket$ (defined by substituting  $\wdh{z}(x)$ for $z$) such that the following diagram commutes:
 
 $$\xymatrix{\k\{x\} \ar[r]^{\phi} \ar[d] & \k\llbracket x\rrbracket\\
 \frac{\k\{x,y\}}{I} \ar[ru]^{\psi}  \ar@{.>}[r]& D:=\k\{x,z\} \ar@{.>}[u]}$$  
 We will use and generalize this formulation later (see Theorem \ref{Popescu}).
 \end{remark}


 \subsection{Artin Approximation and the Weierstrass Division Theorem}\label{AAATWDT}
 The proof of Theorem \ref{Ar68} uses essentially only two results: the \index{Weierstrass Division Property} Weierstrass Division Theorem and the Implicit Function Theorem. In particular it is straightforward to check that the proof of Theorem \ref{Ar68} remains true if we replace $\k\{x,y\}$ by $\k\langle x,y\rangle$, the ring of algebraic power series in $x$ and $y$, since this ring satisfies the Weierstrass Division Theorem (cf. \cite{Laf}, see Section \ref{weierstrass}) and the Implicit Function Theorem (cf. Lemma \ref{henselization_A[x]}; in fact in general the Weierstrass Division Theorem implies the Implicit Function Theorem, cf. Lemma \ref{WPPIFT}). We state here this very important variant of Theorem \ref{Ar68} (which is in fact valid in any characteristic - see also Remark \ref{rmkAr68} ii):
 
  \begin{theorem}[Algebraic Artin Approximation Theorem]\label{Ar69}\cite{Ar69}\index{algebraic Artin Approximation Theorem}
 Let $\k$ be a field and let $f(x,y)\in \k[x,y]^p$ (resp. $\k\langle x,y\rangle^p$) be a vector of polynomials (resp. algebraic power series)  with coefficients in $\k$.  Assume given a formal power series solution $\wdh{y}(x)\in  \k\lb x\rb^m$ (resp. vanishing at 0),
  $$f(x,\wdh{y}(x))=0.$$
  Then there exists, for any $c\in\N$, an algebraic power series  \index{algebraic power series} solution $\wdt{y}(x)\in \k\langle x\rangle^m$ (resp. vanishing at 0),
  $$f(x,\wdt{y}(x))=0$$
  which coincides with $\wdh{y}(x)$ up to degree $c$,
  $$\wdt{y}(x)\equiv \wdh{y}(x) \text{ modulo } (x)^c.$$
  
  \end{theorem}
  
 In fact in \cite{Ar69} M. Artin gives a more general version of this statement valid for polynomial equations over a field or an excellent discrete valuation ring $R$, and proves that the formal solutions of such equations can be approximated by solutions in the Henselization of the ring of polynomials over $R$, in particular in a localization of a  finite extension of the ring of polynomials over $R$. In the case $R=\k$ is a field the Henselization of $\k[x]_{(x)}$ is the ring of algebraic power series $\k\langle x\rangle$ (see Lemma \ref{henselization_A[x]}).
 The proof of the result of M. Artin, when $R$ is an excellent discrete valuation ring, uses N\'eron $\p$-desingularization \cite{Ne} (see Section \ref{N\'eron} for a statement of N\'eron $\p$-desingularization). This result is very important since it enables to reduce some algebraic problems over complete local rings to local rings which are localizations of finitely generated rings over a field or a discrete valuation ring. \\
 For instance this idea, first used by C. Peskine and L. Szpiro, was exploited by M. Hochster to reduce problems over complete local rings in characteristic zero to the same problems in positive characteristic. The idea is the following: let us assume that some statement $(T)$ is true  in positive characteristic (where we can use the Frobenius map to prove it for instance) and let us assume that there exists an example showing that $(T)$ is not true in characteristic zero. In some cases we can use the Artin Approximation Theorem to show the existence of a counterexample to  $(T)$ in the Henselization at a prime ideal  of a finitely generated algebra over a field of characteristic zero. Since the Henselization is the direct limit of \'etale extensions, we can show the existence of a counterexample to $(T)$ in a  local ring $A$ which is the localization of a finitely generated algebra over a characteristic zero  field  $\k$. Thus $A$ is defined by a finite number of data and we may lift this counterexample to a ring which is the localization of a finitely generated ring over $\Q$, and even over $\Z[\frac{1}{p_1},\ldots,\frac{1}{p_s}]$ where the $p_i$  are prime integers. Finally we may show that this counterexample remains a counterexample to  $(T)$ over $\Z/p\Z$ for all but finitely many primes $p$ by reducing the problem modulo $p$ (in fact for $p\neq p_i$ for $1\leq i\leq s$). This is a contradiction which completes the proof.\\
 This idea was used to prove important results about Intersection Conjectures (in \cite{P-S} for the first time)   and  Homological Conjectures  \cite{H-R,H} (see \cite{Sch} 8.6 for more details).\\
  
 J. Denef and L. Lipshitz   axiomatized the properties a ring needs to satisfy in order to adapt the proof of the main theorem of \cite{Ar69} due M. Artin. They called such families of rings \emph{Weierstrass Systems}. There are two reasons for introducing such rings: the first one is the proof of Theorem \ref{nested_ana} (i.e. the 1-variable Nested Approximation) and the second one is their use in proofs of Strong Artin Approximation results via ultraproducts (see Remark \ref{tobedone}). Previously H. Kurke, G. Pfister, D. Popescu, M. Roczen and T. Mostowski (cf. \cite{KPPRM}) introduced the notion of \emph{Weierstrass category} which is very similar (see \cite{KP} for a connection between these two notions).\\
 Before giving the definition of a Weierstrass System we need two definitions:
 
 \begin{definition}
 Let $(A,\m_A)$ be a local ring. The \index{completion}  \emph{completion} of $A$, denoted by $\wdh A$, is  the limit  $\underset{{\longleftarrow}}{\lim}\frac{A}{\m_A^n}$. In the case where $A$ is one of the following rings: $\k[x]_{(x)}$, $\k\langle x\rangle$, $\k\{x\}$ where $\k$ is a field, then $\wdh A=\k\lb x\rb$. When $A$ is a field then $\wdh A=A$ since $\m_A=(0)$.
 
 \end{definition}
 
 \begin{definition}
 A \index{discrete valuation ring} \emph{discrete valuation ring} is a Noetherian local domain whose maximal ideal is  principal and different from $(0)$.\\
  The main examples of complete discrete valuation rings are the ring of power series $\k\lb x\rb$ with $x$ a single variable and $\k$ a field (in this case its maximal ideal is $\p=(x)$), and the ring of $p$-adic integers $\Z_p$ (in this case its maximal ideal is $\p=(p)$).\\
   The rings of algebraic power series $\k\langle x\rangle$ or convergent power series $\k\{x\}$ in one variable $x$ are Henselian discrete valuation rings.\\
   The rings $\Z_{(p)}$, with $p$ prime, and $\k[x]_{(x)}$, with $x$ one variable, are discrete valuation rings but they are not Henselian.
 \end{definition}
  
\begin{definition}\label{W-sys}\cite{D-L} \index{Weierstrass System} Let $\k$ be a field or a discrete valuation ring of maximal ideal $\p$. By a \emph{Weierstrass System of local $\k$-algebras}, or a \emph{W-system} over $\k$, we mean a family of  $\k$-algebras $\k\llceil x_1,\ldots,\,x_n\rrceil$, $n\in\N$ such that:
\begin{enumerate}
\item[i)] For $n=0$, the $\k$-algebra is $\k$,\\
For any $n\geq 1$, $\k[x_1,\ldots,x_n]_{(\p,x_1,\ldots,x_n)}\subset \k\llceil x_1,\ldots,x_n\rrceil \subset \wdh \k\llbracket x_1,\ldots,x_n\rrbracket  $\\
and $\k\llceil x_1,\ldots,\,x_{n+m}\rrceil \cap \wdh  \k\llbracket x_1,\ldots,\,x_n\rrbracket  =\k\llceil x_1,\ldots,\,x_n\rrceil$ for $m\in\N$.
For any permutation $\s$ of $\{1,\ldots,\,n\}$ 
$$f\in\k\llceil x_1,\ldots,x_n\rrceil \Longrightarrow f(x_{\s(1)},\ldots,\,x_{\s(n)})\in\k\llceil x_1,\ldots, x_n\rrceil.$$
\item[ii)]  Any element of $\k\llceil x\rrceil$, $x=(x_1,\ldots,x_n)$, which is a unit in $\wdh \k\llbracket x\rrbracket  $, is a unit in $\k\llceil x\rrceil$.
\item[iii)] If $f\in\k\llceil x\rrceil $ and $\p$ divides $f$ in $\wdh \k\lb x\rb$ then $\p$ divides $f$ in $\k\llceil x\rrceil$. Here $\wdh\k$ denotes the completion of $\k$ when $\k$ is a discrete valuation ring, i.e. $\wdh \k=\underset{{\longleftarrow}}{\lim}\frac{\k}{\p^n}$. When $\k$ is a field $\wdh\k=\k$.
\item[iv)] Let $f\in(\p, x)\k\llceil x\rrceil$ such that $f\neq 0$. Suppose that $f\in (\p,x_1,\ldots,x_{n-1},x_n^s)$ but $f\notin (\p,x_1,\ldots,x_{n-1}, x_n^{s-1})$. Then for any $g\in\k\llceil x\rrceil$ there exist a unique $q\in\k\llceil x\rrceil$ and a unique $r\in\k\llceil x_1,\ldots,\,x_{n-1}\rrceil[x_n]$ with $\deg_{x_n}r<d$ such that $g=qf+r$.
\item[v)] (if char$(\k)>0$) If $\ovl{y}\in(\p,x_1,\ldots,\,x_n)\wdh \k\llbracket x_1,\ldots,\,x_n\rrbracket ^m $ and $f\in\k\llceil y_1,\ldots,\,y_m\rrceil$ such that $f\neq 0$ and $f(\ovl{y})=0$,  there exists $g\in\k\llceil y\rrceil$ irreducible in $\k\llceil y\rrceil$ such that $g(\ovl{y})=0$ and such that there does not exist any unit $u(y)\in\k\llceil y\rrceil$ with $u(y)g(y)=\sum_{\a\in\N^n}a_{\a}y^{p\a}$ ($a_{\a}\in\k$).
\item[vi)] (if char$(\k/\p)\neq 0$) Let $(\k/\p)\llceil x\rrceil$ be the image  of $\k\llceil x\rrceil$ under the projection $\wdh \k\lb x\rb\lgw (\k/\p)\lb x\rb$. Then $(\k/\p)\llceil x\rrceil$ satisfies v).
\end{enumerate}
\end{definition}

\begin{proposition}\cite{D-L}\label{prop_W} Let us consider a $W$-system $\k\llceil x\rrceil$. 
\begin{enumerate} \item[i)] For every $n\geq$, $\k\llceil x_1,\ldots,x_n\rrceil$ is a Noetherian Henselian regular  local ring. In particular $\k$ is a Henselian local ring.
\item[ii)] If $f\in\k\llceil x_1,\ldots,x_n,y_1,\ldots,y_m\rrceil$ and $g\in(\p,x)\k\llceil x_1,\ldots,x_n\rrceil^m$,  $f(x,g(x))\in\k\llceil x \rrceil$.
\item[iii)] If $f\in\k\llceil x\rrceil$, then $\frac{\partial f}{\partial x_i}\in\k\llceil x\rrceil$.
\item[iv)] If $\k\llceil x_1,\ldots,x_n\rrceil$ is a family of rings satisfying i)-iv) of Definition \ref{W-sys} and if all these rings are excellent,  they satisfy v) and vi) of Definition \ref{W-sys}.

\end{enumerate}
\end{proposition}
\begin{proof}
All these assertions are proven in Remark 1.3 \cite{D-L}, except iv).\\  
Proof of iv): let us assume that char$(\k)=p>0$ and let $\ovl{y}\in (\p,x)\wdh{\k}\lb x\rb^m$. We denote by $I$ the kernel of the $\k\llceil x\rrceil$-morphism $\k\llceil x, y\rrceil\lgw \wdh{\k}\lb x\rb$ defined by the substitution of $\ovl{y}$ for $y$ and let us assume that $I\cap \k\llceil y\rrceil\neq (0)$. Since  $\k\llceil x \rrceil$ is excellent, the morphism $\k\llceil x\rrceil \lgw \wdh{\k}\lb x\rb$ is regular (see Example \ref{ex_reg}). Thus $\Frac(\wdh{\k}\lb x\rb)$ is a separable extension of  $\Frac(\k\llceil x\rrceil)$ (see Example \ref{ex_reg}), but $\Frac\left(\frac{\k\llceil x, y\rrceil}{I}\right)$ is a subfield of $\Frac(\wdh{\k}\lb x\rb)$, hence $\Frac(\k\llceil x\rrceil)\lgw\Frac\left(\frac{\k\llceil x, y\rrceil}{I}\right)$ is separable. This implies that the field extension $\Frac(\k)\lgw \Frac\left(\frac{\k\llceil y\rrceil}{I\cap \k\llceil y\rrceil}\right)$ is a separable field extension. But if for every irreducible $g\in I\cap \k\llceil y\rrceil$  there  existed a unit $u(y)\in\k\llceil y\rrceil$ with $u(y)g(y)=\sum_{\a\in\N^n}a_{\a}y^{p\a}$,  the extension $\Frac(\k)\lgw \Frac\left(\frac{\k\llceil y\rrceil}{I\cap \k\llceil y\rrceil}\right)$ would be purely inseparable. This proves that Property v) of Definition \ref{W-sys} is satisfied.\\ 
The proof that Property vi) of Definition \ref{W-sys} is satisfied is identical.
 \end{proof}

\begin{example}
  We give here a few examples of Weierstrass systems:
  \begin{enumerate}
\item[i)] If $\k$ is a field or a complete discrete valuation ring, the family $\k\llbracket x_1,\ldots,\,x_n\rrbracket$ is a W-system over $\k$ (using Proposition \ref{prop_W} iv) since complete local rings are excellent rings).
\item[ii)] Let $\k\langle x_1,\ldots,\,x_n\rangle$ be the Henselization of the localization of $\k[x_1,\ldots,\,x_n]$ at the maximal ideal $(x_1,\ldots,\,x_n)$ where $\k$ is a field or an excellent  discrete valuation ring. Then, for $n\geq 0$, the family $\k\langle x_1,\ldots,\,x_n\rangle$ is a W-system over $\k$ (using Proposition \ref{prop_W} iv) since the Henselization of an excellent local ring is again  excellent - see Proposition \ref{henz_exc}).
\item[iii)] The family $\k\{x_1,\ldots,\,x_n\}$ (the ring of convergent power series in $n$ variables over a valued field $\k$) is a W-system over $\k$.
\item[iv)] The family of Gevrey power series in $n$ variables over a valued field $\k$ is a W-system  \cite{Br}.

\end{enumerate}
\end{example}
Then we have the following Approximation result (the case of $\k\langle x\rangle$ where $\k$ is a field or a discrete valuation ring is proven in \cite{Ar69}, the general case is proven in \cite{D-L} - see also \cite{Robba} for a particular case):

\begin{theorem}\label{W}\cite{Ar69,D-L}
Let $\k\llceil x\rrceil$ be a W-system over $\k$, where $\k$ is a field or a discrete valuation ring with maximal ideal $\p$. Let  $f\in\k\llceil x,y\rrceil^r$ and  $\wdh{y}\in (\p,x)\wdh{\k}\lb x\rb^m$ satisfy
$$f(x,\wdh{y})=0.$$ Then, for any $c\in\N$, there exists a  power series solution $\wdt{y}\in(\p,x)\k\llceil x\rrceil^m$,
 $$f(x,\wdt{y})=0 \text{ such that } \wdt{y}-\wdh{y}\in(\p,x)^c.$$
\end{theorem}

\noindent Moreover let us mention that Theorem \ref{Pl} extends also  to Weierstrass systems (see \cite{Ro12}).

\begin{remark}\label{mouze'}
Let $(m_k)_k$ be a logarithmically convex sequence of positive real numbers, i.e.
\begin{equation}\label{logcon} m_0=1\ \text{ and } \ m_km_{k+2}\geq m_{k+1}^2\ \forall k\in\N,\end{equation}
 and $\k=\R$ or $\C$. The set $\k\lb x\rb(m_k)$ is the subset of $\k\lb x\rb$ defined as follows:
\begin{equation}\label{mouze}\k\lb x\rb(m_k)= \left\{\sum_{\a\in\N^n}f_{\a}x^{\a}\in\k\lb x\rb\ / \ \exists C>0,\ \forall \a,\ \sup_{\a\in\N^n}\frac{|f_{\a}|}{C^{|\a|}m_{|\a|}}<\infty\right\}.\end{equation}
By Leibniz's rule and \eqref{logcon}, $\k\lb x\rb(m_k)$ is a subring of $\k\lb x\rb$. 
This ring does not satisfy the Weierstrass division Theorem but it satisfies Theorem \ref{W} and Theorem \ref{Pl}  (see \cite{Mouze}). To be more precise if $f$ and $g\in\k\lb x\rb(m_k)$ and $f$ is $x_n$-regular of order $d$, then  for the Weierstrass division of $g$ by $f$:
$$g=fq+r$$
the series $q$ and $r$ are not in $\k\lb x\rb(m_k)$ in general. Nevertheless if $d=\ord(f)$ then $q$ and $r\in\k\lb x\rb(m_k)$ \cite{CC}. But in the proof of the Artin Approximation Theorem one needs to divide by a well chosen minor $\d(x)$ that is made $x_n$-regular $d$ by a linear change of coordinates and $d$ can be chosen such that $d=\ord(\d(x))$ (see Remark \ref{change_x_n-regular}). So the original proof of Artin adapts also to this case.
\end{remark}


\subsection{The General N\'eron Desingularization Theorem}\label{N\'eron}
During the 70s  and the 80s  one of the main objectives concerning  the Artin Approximation Problem was to find  necessary and sufficient conditions for a local ring $(A,\m_A)$ to have the Artin Approximation Property, i.e. such that the set of solutions in $A^m$ of any system of algebraic equations $(\mathcal{S})$ in $m$ variables with coefficients in $A$ is dense for the Krull topology  in the set of solutions  of $(\mathcal{S})$ in $\wdh{A}^m$.  \\
Let us recall that the Krull topology on $A$ is the topology induced by the following norm: $|a|:=e^{-\ord(a)}$ for all $a\in A\backslash\{0\}$.
The problem was to find a way of proving approximation  results without using the Weierstrass Division Theorem which does not hold for every Henselian local ring (see Example \ref{ex_nested}).\\

\begin{remark}
The most important case is when the ring $A$ is Noetherian. So in the following we only consider this case. But there are also examples of non-Noetherian rings $A$ that satisfy analogues of  the Artin Approximation Property, see  \cite{Sch1,M-B'}.\index{Approximation in non-Noetherian rings} See also \cite{To76} for the case  of $\mathcal C^\infty$ real function germs.\index{$\mathcal C^\infty$ function germs}

\end{remark}

 \begin{remark}
 Let $P(y)\in A[y]$ satisfy $P(0)\in\m_A $ and $\frac{\partial P}{\partial y}(0)\notin \m_A$. Then, by the Implicit Function Theorem for complete local rings (see Example \ref{ex_henselian} and Theorem \ref{IFT}), $P(y)$ has a unique root in $\wdh{A}$ equal to 0 modulo $\m_A$. Thus if we want to be able to approximate roots of $P(y)$ in $\wdh{A}$  by roots of $P(y)$  in $A$, a necessary condition is that the root of $P(y)$ constructed by the Implicit Function Theorem is in $A$. Thus it is clear that if a local ring $A$ has the Artin Approximation Property then $A$ has to satisfy the Implicit Function Theorem, in other words $A$ is necessarily Henselian (see Appendix \ref{etale_app} for a definition of  a Henselian ring).
  \end{remark}
In fact M. Artin conjectured that a sufficient condition would be that $A$ is an excellent Henselian local ring (see \cite[Conjecture (1.3)]{Ar70} or \cite{Ar82} where the result is proven when $A$ is the ring of convergent power series). 
The idea emerges soon that in order to prove this conjecture one should generalize P\l oski's   Theorem \ref{Pl} and a theorem of desingularization of A. N\'eron \cite{Ne} (see \cite[Question 3]{Ra72}). This generalization is the following (for the definitions and properties of a regular morphism and of an excellent local ring cf. Appendix \ref{regular_app} - for those concerning smooth and Žtale morphisms cf. Appendix \ref{etale_app}):

  \index{general N\'eron Desingularization}
 
  \begin{theorem}[General N\'eron Desingularization]\label{Popescu}\cite{Po85, Po86}
   Let be given $\phi : A\lgw B$ a \index{regular morphism} regular morphism of  Noetherian rings, $C$ a finitely generated $A$-algebra and $\psi : C\lgw B$ a morphism of $A$-algebras. Then $\psi$ factors through a finitely generated $A$-algebra $D$ which is smooth over $A$:
 $$\xymatrix{A \ar[r]^{\phi} \ar[d] & B\\
 C \ar[ru]^{\psi}  \ar@{.>}[r]& D \ar@{.>}[u]}$$
 \end{theorem}
 
 Historically the first version of this theorem has been proven by A. N\'eron \cite{Ne} under the assumption that $A$ and $B$ are discrete valuation rings.  Then several authors gave  proofs of particular cases (see for instance \cite{Po80,Ar82,Brown',Ar-De,Ar-Ro,Rot1} - in this last paper the result is proven in the equicharacteristic zero case) until D. Popescu \cite{Po85,Po86} proved the general case. Then several authors provided simplified proofs or strengthened   this result \cite{Og, Sp, Sw, ST}. This result is certainly the most difficult to prove among all the results presented in this paper. We will just give a slight hint of the proof of this result here because there exist very nice and complete presentations of the proof elsewhere (see \cite{Sw} or \cite{ST}  for the general case, \cite{Qu} or \cite{Po00} for  the equicharacteristic zero case). \\
 Before explaining the relation with the Artin Approximation Theorem let us give one more definition. 
 Let $(A,I)$ be the data of a ring $A$ and an ideal $I$ of $A$. There exists a notion of \emph{Henselian pair} for such a couple $(A,I)$ which coincides with the notion of Henselian local ring when $A$ is a local ring and $I$ is its maximal ideal. One definition is the following: a couple $(A,I)$ is a Henselian pair \index{Henselian pair} if Hensel's Lemma (with the notation of Proposition \ref{Henselian}) is satisfied for $\m_A$ replaced by the ideal $I$. The reader may consult \cite[Part XI]{Ra}  for details. In what follows the reader may think about a Henselian pair $(A,I)$ only as a Henselian local ring $A$ whose maximal ideal is $I$. \\
Because $A\lgw \wdh{A}$ is regular when $A$ is an \index{excellent ring} excellent ring (see \ref{ex_reg}), $I$ is an ideal of $A$ and $\wdh{A}:=\underset{{\longleftarrow}}{\lim}\frac{A}{I^n}$ is  the $I$-adic completion of $A$, we get the following result:
 \index{general Artin Approximation}
 
 \begin{theorem}[General Artin Approximation]\label{Pop_app}
 Let $(A,I)$ be an excellent  Henselian pair and $\wdh A$ be the $I$-adic completion of $A$. Let $f(y)\in A[y]^r$ and $\wdh{y}\in\wdh{A}^m$  satisfy $f(\wdh{y})=0$. Then, for any $c\in\N$, there exists $\wdt{y}\in A^m$  such that $\wdt{y}-\wdh{y}\in I^c\wdh{A}$, and $f(\wdt{y})=0$.
 \end{theorem}
 
 \begin{proof}
The proof goes as follows: let us set  $C:=\frac{A[y]}{J}$ where $J$ is the ideal generated by $f_1,\ldots,f_r$. The formal solution $\wdh{y}\in\wdh{A}$ defines a $A$-morphism $\wdh{\phi} : C\lgw \wdh{A}$ (see Remark \ref{morphism}). By Theorem \ref{Popescu}, since $A\lgw \wdh{A}$ is regular (Example \ref{ex_excellent}), there exists a smooth $A$-algebra $D$ factorizing this morphism. After a change of variables we may assume that $\wdh y\in \m_{\wdh A}$ so the morphism $C\lgw \wdh A$ extends to a morphism $C_{\m_A+(y)}\lgw \wdh A$ and this latter morphism factors through $D_{\m}$ where $\m$ is the inverse image of $\m_{\wdh A}$.
The morphism $A\lgw D_{\m}$  decomposes  as $A\lgw A[z]_{\m_A+(z)}\lgw D_{\m}$ where $z=(z_1,\ldots,z_s)$ and $A[z]_{(z)}\lgw D_{\m}$ is a local \'etale morphism  \cite[Theorem 3.1 III.3]{Iv}. Let us choose $\wdt{z}\in A^s$ such that $\wdt{z}-\wdh{z}\in \m_A^c\wdh{A}^s$ ($\wdh{z}$ is the image of $z$ in $\wdh{A}^s$). This defines a morphism $A[z]_{(z)}\lgw A$. Then $A\lgw \frac{D_{\m}}{(z_1-\wdt{z}_1,\ldots,z_s-\wdt{z}_s)}$ is local \'etale and admits a section in $\frac{A}{\m_A^c}$. Since $A$ is Henselian, this section lifts to a section in $A$ by Proposition \ref{lift}. This section composed with $A[z]_{(z)}\lgw A$ defines a $A$-morphism $D_{\m}\lgw A$, and this latter morphism composed with $C\lgw D_{\m}$ yields a morphism $\wdt{\phi}: C\lgw A$ such that $\wdt{\phi}(z_i)-\wdh{\phi}(z_i)\in\m_A^c\wdh{A}$ for $1\leq i\leq m$.
  \end{proof}
 
 \begin{remark}
 Let $(A,I)$ be a Henselian pair and let $J$ be an ideal of $A$. By applying this result to  the Henselian pair $\left(\frac{B}{J},\frac{IB}{J}\right)$ we can prove the following result (using the notation of Theorem \ref{Pop_app}): if $f(\wdh{y})\in J\wdh{A}$ then there exists $\wdt{y}\in A^m$ such that $f(\wdt{y})\in J$ and $\wdt{y}-\wdh{y}\in I^c\wdh{A}$.
 \end{remark}
 
 In fact the General N\'eron Desingularization Theorem is a result of desingularization which generalizes  Theorem \ref{Pl} to any excellent Henselian local ring as shown in Corollary \ref{Ploski2} given below. In particular it provides a parametrization of a piece of the set $f=0$ locally at a given formal solution. Corollary \ref{Ploski2} does not appear  in the literature but it is useful to understand Theorem \ref{Popescu} when $B$ is the completion of a local domain. Before giving this statement let us state the following lemma which was first proven by M. Nagata with the extra assumption of normality \cite[ 44.1]{Na}:
 \begin{lemma}\label{henselization_A[x]}
 If $A$ is an excellent local domain  we denote by $A^h$ its Henselization. Then  $A^h$ is exactly the algebraic closure of $A$ in its completion $\wdh{A}$. In particular, for an excellent Henselian local domain $A$  (a field for instance)   the ring $A\langle x\rangle$ of elements of $\wdh A[[x]]$ algebraic over $A[x]$, i.e. the ring of algebraic power series with coefficients in $A$, is the Henselization of the local ring $A[x]_{\m_A+(x)}$. Thus $A\langle x\rangle$ satisfies the Implicit Function Theorem (see Theorem \ref{IFT}).
 \end{lemma}
 
 Apparently it is not known if this lemma remains true for  excellent local rings which are not integral domains.
 
 \begin{proof}
 Indeed $A\lgw A^h$ is a filtered limit of algebraic extensions, thus $A^h$ is a subring of the ring of algebraic elements of $\wdh{A}$  over $A$.\\
 On the other hand if $f\in \wdh{A}$ is algebraic over $A$, then $f$ satisfies an equation 
 $$a_0f^d+a_1f^{d-1}+\cdots+a_d=0$$
 where $a_i\in A$ for all $i$. Thus for $c$ large enough there exists $\wdt{f}\in A^h$ such that $\wdt{f}$ satisfies the same polynomial equation and $\wdt{f}-f\in\m_A^c$ (by Theorem \ref{Pop_app} and Theorem \ref{henz_exc}). Because $\cap_c\m_A^c=(0)$ and a polynomial equation has a finite number of roots (because $\wdh A$ is a domain -  see Proposition \ref{int_domain} given in the next chapter),  we have  $\wdt{f}=f$ for $c$ large enough and $f\in A^h$.
  \end{proof}

Then we have the following result that also implies Theorem \ref{Pop_app} in the same way as Theorem \ref{Pl} implies Theorem \ref{Ar68} (see Remark \ref{Pl->app}):
 \begin{corollary}\label{Ploski2}
 Let  $A$ be an excellent Henselian local  domain and $f(y)\in A[y]^p$ where $y=(y_1,\ldots,y_m)$. Let $\wdh y\in\wdh A^m$ be a solution of $f(y)=0$. Then there exist an integer $s$, a vector $y(z)\in A\langle z\rangle$ with $z=(z_1,\ldots,z_s)$ and a vector $\wdh z\in\wdh A^s$ such that
 $$f(y(z))=0, $$
 $$\wdh y= y(\wdh z).$$ 
 \end{corollary}
 
 \begin{proof}
Let us define $C=A[y]/(f)$. The  formal solution $\wdh{y}\in \wdh A^m$ of the equations $f=0$ defines a $A$-morphism $\psi : C\lgw \wdh A$  such that the following diagram commutes:
  $$\xymatrix{A \ar[r]^{\phi} \ar[d] & \wdh A\\
 C \ar[ru]^{\psi}  &  }$$
 Let $D$ be a smooth finitely generated $A$-algebra given by Theorem \ref{Popescu}. The $A$-algebra $D$ has the form
 $$D=A[z_1,\ldots,z_t]/(g_1,\ldots,g_r)$$
 for some polynomials $g_i\in A[z_1,\ldots,z_t]$ and new variables $z=(z_1,\ldots,z_t)$. For every $j$ let $a_j\in A$ such that the image of $z_j$ in $\wdh A$ is equal to $a_j$ modulo $\m_A$. By replacing $z_j$ by $z_j-a_j$ for every $j$ we can assume that $\psi$ factors through $D_{\m_A+(z_1,\ldots,z_t)}$.\\
 Since $A$ is an excellent local domain, $A[z]_{\m_A+(z)}$ is also an excellent local domain and its Henselization is equal to its algebraic closure in its completion $\wdh A[[z]]$ (see Example \ref{henselization_A[x]} below). Thus the Henselization $D^h$ of $D_{\m_A+(z)}$ is equal to 
 $$D^h=A\langle z_1,\ldots, z_t\rangle/(g_1,\ldots,g_r).$$
 But $D^h$ being smooth over $A$ means that the jacobian matrix $\left(\frac{\partial g_i}{\partial z_j}\right)$ has maximal rank modulo $\m_A+(z)$. Thus by Hensel's Lemma $D^h$ is isomorphic to $A\langle z_1,\ldots,z_s\rangle$ for some integer $s\leq t$. Since $\wdh A$ is Henselian, by the universal property of the Henselization $\psi$ factors through $D^h$, i.e. $\psi$ factors through $A\langle z_1,\ldots,z_s\rangle$:
 
 $$\xymatrix{A \ar[r]^{\phi} \ar[d] & \wdh A\\
 C \ar[ru]^{\psi}  \ar@{.>}[r]^{\s}& A\langle z\rangle \ar@{.>}[u]^{\t}}$$
where $z=(z_1,\ldots,z_s)$. The morphism $\t$ is completely determined by the images $\wdh z_i\in \wdh A$ of the $z_i$ and the morphism $\s$ is uniquely determined by the images $y_i(z)\in A\langle z\rangle$ of the $y_i$ that are solution of $f=0$.
  \end{proof}
 
 \begin{example}
 Let $A=\C\{ x_1,\ldots,x_n\}$ be the ring of convergent power series in $n$ variables over $\C$. Let $C=\frac{A[y_1,y_2]}{(f)}$ where $f=y_1^2-y_2^3$ and let $(\wdh y_1,\wdh y_2)\in \wdh A ^2$ be a solution of $f=0$. Since $\wdh A=\C\lb x_1,\ldots,x_n\rb$ is a unique factorization domain and $\wdh{y_1}^2=\wdh y_2^3$, $\wdh{y}_2$ divides $\wdh y_1$. Let us define $\wdh z=\frac{\wdh y_1}{\wdh y_2}$. Then we obtain $(\wdh y_1,\wdh y_2)=(\wdh z^3,\wdh z^2)$. \\
 Conversely any vector of the form $(\wdh z^3,\wdh z^2)$, for a power series $\wdh z\in \wdh A$, is a solution of $f=0$. In this example the previous corollary is satisfied with $s=1$ and $y(z)=(z^3,z^2)$. Here we remark that $y(z)$ does not depend on the given formal solution $(\wdh y_1,\wdh y_2)$ which is not true in general.
 
  \end{example}

 \begin{remark}
 In \cite{Rot}, C. Rotthaus proved the converse of Theorem \ref{Pop_app} in the local case: if $A$ is a Noetherian local ring that satisfies Theorem \ref{Pop_app}, then $A$ is excellent. In particular this shows that Weierstrass systems are excellent local rings. Previously this problem had been studied in \cite{C-P} and \cite{Brown}.
 \end{remark}

 \begin{remark}\label{flat}
 Let $A$ be a Noetherian  ring and $I$ be an ideal of $A$. If we assume that $f_1(y),\ldots,f_r(y)\in A[y]$ are linear homogeneous with respect to $y$, then Theorem \ref{Pop_app} may be proven easily  in this case since  $A\lgw \wdh{A}$ is  \index{flatness} flat (see Examples \ref{flat_ana} and \ref{flat_ploski}). The proof of this flatness result uses the Artin-Rees Lemma (see  \cite[Theorems 8.7 and 8.8]{Ma}).
 \end{remark}

 \begin{example}\label{ex_nested}
 The strength of Theorem \ref{Pop_app}  is that it applies to rings that do not satisfy the Weierstrass Preparation Theorem and for which the proof of Theorem \ref{Ar68} or Theorem \ref{W} does not apply. For example Theorem \ref{Pop_app} applies to the local  ring  $B=A\langle x_{1},\ldots,x_n\rangle$ where $A$ is an excellent Henselian local ring (the main example is $B=\k\lb t\rb\langle x\rangle$ where $t$ and $x$ are multivariables). Indeed, this ring is the Henselization of $A[x_{1},\ldots,x_n]_{\m_A+(x_1,\ldots,x_n)}$. Thus $B$ is an excellent local ring by Example \ref{ex_excellent} and Proposition \ref{henz_exc}. \\
  This case was the main motivation of D. Popescu for proving Theorem \ref{Popescu} (see also  \cite{Ar70}) because it implies a nested Artin Approximation result (see Theorem \ref{nested_alg}). \\
  Particular cases of this application had been studied before: see \cite{P-P81} for a direct proof that $V\lb x_1\rb\langle x_2\rangle$ satisfies Theorem \ref{Pop_app}, when $V$ is a complete discrete valuation ring, and \cite{BDL} for the ring $\k\lb x_1,x_2\rb\langle x_3,x_4,x_5\rangle$.
 \end{example}
 
 \begin{remark}\index{Bass-Quillen Conjecture}
 Let us mention that Theorem \ref{Popescu} has other applications than Theorem \ref{Pop_app} even if this latter result is our main motivation for presenting the former theorem. For example one very important application of Theorem \ref{Pop_app} is the proof of the so-called Bass-Quillen Conjecture that asserts that any finitely generated projective $R[y_1,\ldots,y_m]$-module is free when $R$ is a regular local ring (cf. \cite{Sp} for instance).
 \end{remark}
 
 \begin{proof}[Idea of the proof of Theorem \ref{Popescu}] The proof of this theorem is quite involved and would require more machinery than we can present in this paper. The reader interested by the whole proof should consult  \cite{Sw} or \cite{ST} for the general case, or \cite{Qu} or \cite{Po00} for the equicharacteristic zero case.\\
    Let $A$ be a Noetherian ring and  $C$ be a $A$-algebra of finite type, $C=\frac{A[y_1,\ldots,y_m]}{I}$ with $I=(f_1,\ldots,f_r)$. We denote by $\Delta_g$ the ideal of $A[y]$ generated by the $h\times h $ minors of the Jacobian matrix $\left(\frac{\partial g_i}{\partial y_j}\right)_{1\leq i\leq h, 1\leq j\leq m}$ for $g:=(g_1,\ldots,g_h)\subset I$. We define the \index{Jacobian ideal} Jacobian ideal
  $$H_{C/A}:=\sqrt{\sum_{g}\Delta_g((g):I)C}$$
  where the sum runs over all $g:=(g_1,\ldots,g_h)\subset I$ and $h\in\N$. The definition of this ideal may be a bit scary at first sight. What the reader has to know about this ideal is that it    is independent of the presentation of $C$ and its support is the singular locus of $C$ over $A$:
  
  \begin{lemma}\label{smooth_loc}
  For any prime $\p\in\Spec(C)$, $C_{\p}$ is smooth over $A$ if and only if  $H_{C/A}\not\subset \p$.
  \end{lemma}
 
   The following property  will be used in the proof of Proposition \ref{smoothing1}:
  \begin{lemma}\label{prop_smooth_loc}
Let $C$ and $C'$ be two $A$-algebras of finite type and let $A\lgw C\lgw C'$ be two morphisms of $A$-algebras. Then
 $$H_{C'/C}\cap \sqrt{H_{C/A}C'}=H_{C'/C}\cap  H_{C'/A}.$$

  \end{lemma}

 The idea of the proof of Theorem \ref{Popescu} is the following: if $H_{C/A}B\neq B$, then we replace $C$ by a $A$-algebra of finite type $C'$ such that  $H_{C/A}B$ is a proper sub-ideal of $H_{C'/A}B$. Using the Noetherian assumption, after a finite number of steps we have $H_{C/A}B=B$. Then we use the following proposition:
 \begin{proposition}\label{smoothing1}
 Using the notation of Theorem \ref{Popescu}, let us assume that we have $H_{C/A}B=B$. Then  $\psi$ factors as in Theorem \ref{Popescu}.
 \end{proposition}
  \begin{proof}[Proof of Proposition \ref{smoothing1}]
  Let $(c_1,\ldots,c_s)$ be a system of generators of $H_{C/A}$. Then 
  $\displaystyle 1=\sum_{i=1}^sb_i\psi(c_i)$
  for some $b_i$  in $B$. Let us define
  $$D:=\frac{C[z_1,\ldots,z_s]}{(1-\sum_{i=1}^sc_iz_i)}.$$
  We construct a morphism of $C$-algebras $D\lgw B$ by sending $z_i$  onto $b_i$, $1\leq i\leq s$. It is easy to check that $D_{c_i}$ is a smooth $C$-algebra for any $i$, thus $c_i\in H_{D/C}$ by Lemma \ref{smooth_loc}, and $H_{C/A}D\subset H_{D/C}$. By Lemma \ref{prop_smooth_loc} used for $C'=D$, since $1\in H_{C/A}D$, we see that $1\in H_{D/A}$. By Lemma \ref{smooth_loc}, this proves that $D$ is a smooth $A$-algebra.
   \end{proof}
 
  Now to increase the size of $H_{C/A}B$ we use the following proposition:
  \begin{proposition}\label{key_prop}
  Using the notation of Theorem \ref{Popescu}, let $\p$ be a minimal prime ideal of $H_{C/A}B$. Then there exists a factorization of $\psi : C\lgw D\lgw B$ such that  $D$ is finitely generated over $A$ and $\sqrt{H_{C/A}B}\subsetneq\sqrt{H_{D/A}B}\not\subset \p$.
  
  \end{proposition}
  
  The proof of Proposition \ref{key_prop} is done by a decreasing induction on the  height of $\p$. Thus there are two things to prove: first the case where ht$(\p)=0$ , then the reduction ht$(\p)=k+1$ to the case ht$(\p)=k$. This last case is quite technical, even in the equicharacteristic zero case (i.e. when $A$ contains $\Q$, see \cite{Qu} for a good presentation of this case). In the case where  $A$ does not contain $\Q$ there appear more problems due to the existence of inseparable extensions of residue fields. In this case the  Andr\'e homology is the right tool to handle these problems (see \cite{Sw}). \index{AndrŽ homology}

   \end{proof}


   \section {Strong Artin Approximation}
We review here results about the Strong Approximation Property. There are  clearly two different cases: the case where the base ring is a discrete valuation ring (where life is easy!) and the second case is the general case (where life is less easy).


\subsection{Greenberg's  Theorem: the case of a discrete valuation ring}\label{sec_Green}
Let $V$ be a Henselian discrete valuation ring, $\m_V$ its maximal ideal and $\K$ be its field of fractions. Let us denote by $\wdh{V}$ the $\m_V$-adic completion of $V$ and by $\wdh{\K}$ its field of fractions. If char$(\K)>0$, let us assume that $\K\lgw \wdh{\K}$ is a separable field extension (in this case this is equivalent to $V$ being excellent, see Example \ref{ex_reg} iii) and Example \ref{ex_excellent} iv)).\index{Greenberg's  Theorem}

\begin{theorem}[Greenberg's  Theorem]\label{Greenberg}\cite{Gr} With the above notation and hypotheses, if $f(y)\in V[y]^r$,  there exist $a$, $b\geq 0$   such that
$$\forall c\in\N \ \forall \ovl{y}\in V^m \text{ such that } f(\ovl{y})\in\m_V^{ac+b}$$
$$\exists \wdt{y}\in V^m \text{ such that } f(\wdt{y})=0 \text{ and } \wdt{y}-\ovl{y}\in \m_V^c.$$
\end{theorem}

\begin{proof}[Sketch of proof]
We will give the proof in the case char$(\K)=0$. As for the classical Artin Approximation Theorem this statement depends only on the ideal generated by the components of $f(y)$. The result is proven by induction on the height of the ideal generated by $f_1(y),\ldots,f_r(y)$. Let us denote by $I$ this ideal. We will denote by $\nu$ the $\m_V$-adic order on $V$:
$$\nu(v)=\max\{n\in\N \backslash v\in\m_V^n\}\ \ \forall v\in V, v\neq 0$$
and $\nu(0)=+\infty$. This  is a valuation by assumption.\\
 Let $e$ be an integer such that $\sqrt{I}^e\subset I$. Then $f(\ovl{y})\in  \m_V^{ec}$ for all $f\in I$ implies that $f(\ovl{y})\in\m_V^c$ for all $f\in \sqrt{I}$ since $V$ is a valuation ring. So if the theorem is proven for $\sqrt{I}$ with the constants $a$ and $b$, it is proven for $I$ with the constants $ea$ and $eb$.\\
 Moreover if $\sqrt{I}=\P_1\cap \cdots\cap  \P_s$ is the prime decomposition of $\sqrt{I}$, then $f(\ovl{y})\in  \m_V^{sc}$ for all $f\in \sqrt{I}$ implies that $f(\ovl{y})\in\m_V^c$ for all $f\in \P_{i_0}$ for some $i_0$. So if the theorem is proven for $\P_{i_0}$ with the constants $a$ and $b$, it is proven for $\sqrt I$ with the constants $sa$ and $sb$. This allows us to replace  $I$ by one of its associated  primes, namely $\P_{i_0}$, so we may assume that $I$ is a prime ideal of $V[y]$.\\
 Let $h$ be the height of $I$. If $h=m+1$,  $I$ is a maximal ideal of $V[y]$ and so it contains some non-zero element of $V$ denoted by $v$. Then there is no $\ovl{y}\in V^m$ such that $f(\ovl{y})\in \m_V^{\nu(v)+1}$ for all $f\in I$. Thus the theorem is true for $a=0$ and $b=\nu(v)+1$ (see Remark \ref{none} below).\\
 Let us assume that the theorem is proven for ideals of height $h+1$ and let $I$ be a prime ideal of height $h\leq m$. As in the proof of Theorem \ref{Ar68}, we may assume that $r=h$ and that a $h\times h$ minor of the Jacobian matrix of $f$, denoted by $\d$, is not in $I$. Let us define $J:=I+(\d)$. Since $I$ is prime  we have  ht$(J)=h+1$, so by the inductive hypothesis there exist $a$, $b\geq 0$ such that
$$\forall c\in\N \ \forall \ovl{y}\in V^m \text{ such that } f(\ovl{y})\in\m_V^{ac+b}\ \ \ \forall f\in J$$
$$\exists \wdt{y}\in V^m \text{ such that } f(\wdt{y})=0 \ \ \forall f\in J \text{ and } \wdt{y}_j-\ovl{y}_j\in \m_V^c,\ 1\leq j\leq m.$$
Then let $c\in\N$ and  $\ovl{y}\in V^m$ satisfy $f(\ovl{y})\in\m_V^{(2a+1)c+2b}$ for all $f\in I$. If $\d(\ovl{y})\in\m_V^{ac+b}$, then $f(\ovl{y})\in\m_V^{ac+b}$ for all $f\in J$ and the result is proven by the inductive hypothesis.\\
If $\d(\ovl{y})\notin\m_V^{ac+b}$, then $f_i(\ovl{y})\in(\d(\ovl{y}))^2\m_V^c$ for $1\leq i\leq r$. Then the result comes from the following result.
  \end{proof}
 
 \begin{proposition}[Tougeron  Implicit Function Theorem]\label{TougeronIFT2}\index{Tougeron  Implicit Function Theorem}
Let $A$ denote a Henselian local ring and  $f(y)\in A[y]^h$, $y=(y_1,\ldots,y_m)$, $m\geq h$.  Let $\d(x,y)$ be  a $h\times h$ minor of the Jacobian matrix $\frac{\partial(f_1,\ldots,f_h)}{\partial(y_1,\ldots,y_m)}$. Let us assume that there exists $\ovl{y}\in A^m$ such that
$$f_i(\ovl{y})\in (\d(\ovl{y}))^2\m_A^c  \ \text{ for all }\ 1\leq i\leq h$$
and for some $c\in\N$. Then there exists $\wdt{y}\in A^m$ such that 
$$f_i(\wdt{y})=0 \text{ for all } 1\leq i\leq h\ \text{ and }\  \wdt{y}-\ovl{y}\in (\d(\ovl{y}))m_A^c.$$
\end{proposition}

\begin{proof}
The proof is completely similar to the proof of Theorem \ref{TougeronIFT}.
 \end{proof}

In fact we can prove the following result whose proof is identical to the proof of Theorem \ref{Greenberg} and extends Theorem \ref{Greenberg}  to more general equations than polynomial ones:

\begin{theorem}\label{Greenberg_formal}\cite{Sc1,Sc}
Let $V$ be a complete discrete valuation ring and  $f(y,z)\in V\lb y\rb[z]^r$, where $z:=(z_1,\ldots,z_s)$. Then there exist $a$, $b\geq 0$   such that
$$\forall c\in\N \ \forall \ovl{y}\in (\m_VV)^m,\ \forall  \ovl{z}\in V^s \text{ such that } f(\ovl{y},\ovl{z})\in\m_V^{ac+b}$$
$$\exists \wdt{y}\in (\m_VV)^m,\ \exists \wdt{z}\in  V^s \text{ such that } f(\wdt{y},\wdt{z})=0 \text{ and } \wdt{y}-\ovl{y},\ \wdt{z}-\ovl{z}\in \m_V^c.$$
\end{theorem}

\begin{example}
Let $k$ be some positive integer. Then for any $y\in V$ and $c\in\N$ we have
$$y^k\in \m_V^{kc}\Longrightarrow y\in \m_V^c.$$
Thus for  the polynomial $f(y)=y^k$  Theorem \ref{Greenberg} is satisfied by  the constants $a=k$ and $b=0$. \\
But let us remark that if $y^k\in\m_V^{k(c-1)+1}$ we have  $k\nu(y)\geq k(c-1)+1$ thus $\nu(y)\geq c-1+\frac{1}{k}$. But since $\nu(y)$ is an integer we have  $\nu(y)\geq c$ and $y\in\m_V^c$. Therefore we see here that we can also choose $a=k$ and $b=1-k$ which give a smaller bound than the previous one.

\end{example}

\begin{example}
Let us assume that $V=\C\lb t\rb$ where $t$ is a single variable. Here the valuation $\nu$ is just the $t$-adic order. Let $y_1$, $y_2\in V$ and $c\in\N$ such that 
\begin{equation}\label{cusp_Greenberg}y_1^2-y_2^3\in (t)^{3c}.\end{equation}
If $\ord(y_1)\geq \ord(y_2)$, let us denote by $z$ the power series $\frac{y_1}{y_2}$. Then
$$y_1^2-y_2^3=(z^2-y_2)y_2^2\in (t)^{3c}.$$
Thus 
$$z^2-y_2\in (t)^c \text{ or }y_2\in (t)^c.$$
In the first case we set
$$(\wdt y_1,\wdt y_2):=\left(z^3,z^2\right)=\left(\frac{y_1^3}{y_2^3},\frac{y_1^2}{y_2^2}\right),$$
in the second case we set
$$(\wdt y_1,\wdt y_2):=(0,0).$$
In both cases we have $\wdt y_1^2-\wdt y_2^3=0$. In the first case 
$$\wdt y_1-y_1=\left(\frac{y_1^2}{y_2^2}-y_2\right)\frac{y_1}{y_2}\in (t)^c$$
 and in the second case 
 $$\wdt y_1-y_1=-y_1\in (t)^c$$
 since $\ord(y_1)\geq \ord(y_2)\geq c$. We also have $\wdt y_2-y_2\in (t)^c$.\\
 If $\ord(y_1)<\ord(y_2)$ we have $3c\leq \ord(y_1^2)<\ord(y_2^3)$ and we set 
 $$(\wdt y_1,\wdt y_2):=(0,0).$$
Hence
$$\wdt y_1-y_1\text{ and }\wdt y_2-y_2\in (t)^c.$$
Thus  for the polynomial $f(y_1,y_2)=y_1^2-y_2^3$, Theorem \ref{Greenberg} is satisfied by  the constants $a=3$ and $b=0$. 
 \end{example}
 
 \begin{remark}\label{none}
In the case $f(y)$ has no solution in $V$ we can choose $a=0$ and Theorem \ref{Greenberg} asserts that there exists a constant $b$ such that $f(y)$ has no solution in $\frac{V}{\m_V^b}$.

\end{remark}

\begin{remark}\index{$C_i$ field}
M. Greenberg proved this result in order to study $C_i$ fields. Let us recall that a $C_i$ field is a field $\k$ such that for every integer $d$ every homogeneous form  of degree  $d$ in more than $d^i$ variables with coefficients in $\k$ has a non-trivial zero. More precisely M. Greenberg proved that for a $C_i$ field $\k$, the field of formal power series $\k(\!(t)\!)$ is $C_{i+1}$. Previous results about $C_i$ fields had been previously  studied, in particular by S. Lang in \cite{Lan} where appeared for the first time a special case of the Artin Approximation Theorem appeared for the first time (see Theorem 11 and its corollary in \cite{Lan}).
\end{remark}

\begin{remark}\label{Loj}
The valuation $\nu$ of $V$ defines an ultrametric norm on $\K$ (as noticed in Remark \ref{rem_topology}): we define it as
$$\left|\frac{y}{z}\right|:=e^{\nu(z)-\nu(y)}, \ \ \forall y,z\in V\backslash\{0\}.$$
The norm is ultrametric means that a much stronger version of the triangle inequality holds:
$$\forall y,z \in \K \ \ |y+z|\leq \max\{|y|,|z|\}.$$
This norm defines  a distance on $V^m$, for any $m\in\N^*$, denoted by $d(.,.)$ and defined by 
$$d(y,z):=\max_{1\leq k\leq m}\left|y_k-z_k\right|.$$
Then it is well known that Theorem \ref{Greenberg} can be reformulated as a \L ojasiewicz Inequality (see \cite{T2} or \cite{Ro} for example):
$$\exists a\geq 1,\ C>0 \text{ s.t. } |f(\ovl{y})|\geq Cd(f^{-1}(0),\ovl{y})^a \ \ \forall \ovl{y}\in V^m.$$\\
This kind of  \index{\L ojasiewicz Inequality} \L ojasiewicz Inequality is well known for complex or real analytic functions and Theorem \ref{Greenberg} can be seen as a generalization of this \L ojasiewicz Inequality for algebraic or analytic functions defined over $V$.
If $V=\k\lb t\rb$ where $\k$ is  a field, there are very few known results  about the geometry of algebraic varieties defined over $V$. It is a general problem to extend classical results of differential or analytic geometry over $\R$ or $\C$ to this setting. See for instance  \cite{B-H} (extension of  the Rank Theorem), \cite{Reg} 
(extension of the Curve Selection Lemma), \cite{HiT} (concerning the extension of local metric properties of analytic sets or functions) for some results in this direction. 
\end{remark}
\noindent \index{Greenberg's Function}
\begin{definition}[Greenberg's Function] For any $c\in\N$ let us denote by $\b(c)$ the smallest integer such that:\\
 for all $\ovl{y}\in V^m$ with $f(\ovl{y})\in (x)^{\b(c)}$,  there exists $\wdt{y}\in V^m$ with $f(\wdt{y})=0$ and  $\wdt{y}-y\in (x)^c$.\\
  Greenberg's  Theorem asserts that such a function $\b\ : \N\lgw \N$ exists and  is bounded by a linear function. We call this function $\b$ the \emph{Greenberg's function} of $f$. 
 \end{definition}
 We can remark that the Greenberg's function is an invariant of the integral closure of the ideal generated by $f_1,\ldots,f_r$:\index{integral closure}
 
  \begin{lemma}\label{int_closure}
  Let us consider $f(y)\in V[y]^r$ and  $g(y)\in V[y]^q$. Let us denote by $\b_f$ and $\b_g$ their Greenberg's functions. Let $I$ (resp. $J$) be the ideal of $V[y]$ generated by $f_1(y),\ldots,f_r(y)$ (resp. $g_1(y),\ldots,g_q(y)$). If $\ovl{I}=\ovl{J}$ then $\b_f=\b_g$. 
  \end{lemma}
  \begin{proof}
  Let $\I$ be an ideal of $V$ and  $\ovl{y}\in V^m$. We remark that 
  $$f_1(\ovl{y}),\ldots, f_r(\ovl{y})\in \I\Longleftrightarrow g(\ovl{y})\in\I \ \ \forall g\in I.$$
  Then by replacing $\I$ by $(0)$ and $\m_V^c$, for all $c\in\N$, we see that $\b_f$ depends only on $I$ (see also Remark \ref{artin_ideal}).\\
  Now, for any $c\in \N$, we have:
  $$g(\ovl{y})\in \m_V^c\ \ \forall g\in I\Longleftrightarrow \nu(g(\ovl{y}))\geq c\ \  \forall g\in I$$
  $$\ \qquad\qquad\qquad\qquad\Longleftrightarrow \nu(g(\ovl{y}))\geq c\ \  \forall g\in \ovl{I}$$
 $$\  \qquad\qquad\qquad\qquad\Longleftrightarrow g(\ovl{y})\in \m_V^c \ \ \forall g\in \ovl{I}.$$
    Indeed  if $g\in \ovl I$ then we have 
    $$g^d+a_1 g^{d-1}+\cdots+a_{d-1}g+a_d=0$$
    for some $d\geq 1$ and $a_i\in I^i$. If $\nu(a_i(\ovl y))\geq ic$  for every $i$ we have that $\nu(g(\ovl y))\geq c$. This proves the  implication  $$\nu(g(\ovl{y}))\geq c\ \  \forall g\in I\Longrightarrow \nu(g(\ovl{y}))\geq c\ \  \forall g\in \ovl{I}.$$
    Thus $\b_f$ depends only on $\ovl{I}$.
   \end{proof}

 In general, it is a difficult problem to compute the Greenberg's function of an ideal $I$. It is even a difficult problem to bound this function in general. If we analyze carefully the proof of Greenberg's  Theorem, using classical effective results in commutative algebra, we can prove the following result:
 
 \begin{theorem}\label{Ro10}\cite{Ro10}
 Let $\k$ be a characteristic zero field and  $V:=\k\lb t\rb$ where $t$ is a single variable. Then there exists a function
 $$\N^2\lgw \N$$
 $$(m,d)\lgm a(m,d)$$
 which is a polynomial function in  $d$ whose degree is exponential in $m$, such that for any vector  $f(y)\in \k[t, y]^r$ of polynomials of total degree $\leq d$ with $y=(y_1,\ldots,y_m)$, the Greenberg's function of $f$ is bounded by $c\lgm a(m,d)(c+1)$. 
 \end{theorem}
 
Moreover let us remark that, in the proof of Theorem \ref{Greenberg}, we proved a particular case of the following inequality:
 $$\b_I(c)\leq 2\b_J(c)+c,\ \ \forall c\in\N$$
 where $J$ is the \index{Jacobian ideal} Jacobian ideal of $I$ (for a precise definition of the Jacobian ideal in general and a general proof of this inequality let see \cite{Elk}).	 The coefficient 2 comes from the use of Tougeron  Implicit Function Theorem. We can sharpen this bound in the following particular case:\\
 
 \begin{theorem}\cite{Hi93}\label{Hi1}
 Let $\k$ be an algebraically closed  field of characteristic zero and  $V:=\k\lb t\rb$ where $t$ is a single variable.  Let $f(y)\in V\lb y\rb$ be one power series. Let us denote by $J$ the ideal of $V\lb y\rb$ generated by $f(y)$, $\frac{\partial f}{\partial t}(y)$, $\frac{\partial f}{\partial y_1}(y),\ldots,\frac{\partial f}{\partial y_m}(y)$, and let us denote by $\b_f$ the Greenberg's function of $(f)$ and by $\b_J$ the Greenberg's function of $J$. Then
 $$\b_f(c)\leq \b_J(c)+c\ \ \forall c\in\N.$$

  \end{theorem}
  
 This bound may be used to find sharp bounds of some Greenberg's functions (see Remark \ref{Greenberg_ana_set}). \\
  On the other hand we can describe  the behaviour of $\b$ in the following case:
 
\begin{theorem}\label{De84}\cite{De84}\cite{De-Lo}
Let $V$ be $\Z_p$ or a Henselian discrete valuation ring whose residue field is an algebraically closed field of characteristic zero. Let us denote by $\m_V$ the maximal ideal of $V$. Let $\b$ denote  the  Artin function of $f(y)\in V[y]^r$.  Then there exists a finite partition of $\N$ in congruence classes such that on each such  class the function $c\lgm \b(c)$ is linear for $c$ large enough.
\end{theorem}

\begin{proof}[Hints on the proof in the case the  residue field has characteristic zero]
 Let us consider the following \index{first order language} first order language with three sorts:
\begin{enumerate}
\item[1)] the field $(\K:=\Frac(V),+,\times,0,1)$
\item[2)] the group $(\Z,+,<,\equiv_d (\forall d\in\N^*),0)$ ($\equiv_d$ is the relation $a\equiv_d b$ if and only if $a-b$ is divisible by $d$ for $a$, $b\in \Z$)
\item[3)] the residue field $(\k:=\Frac\left(\frac{V}{\m_V}\right),+,\times,0,1)$
\end{enumerate}
with both following functions:
\begin{enumerate}
\item[a)] $\nu : \K\lgw \Z^*$
\item[b)] $ac : \K\lgw \k$ ("angular component")
\end{enumerate}
The function $\nu$ is the valuation of the valuation ring $V$. The function $ac$ may be characterized by axioms, but here let us just give an example: let us assume that $V=\k\llbracket t\rrbracket$. Then $ac$  is defined by $ac(0)=0$ and $ac\left(\sum_{n=n_0}^{\infty}a_nt^n\right)=a_{n_0}$ if $a_{n_0}\neq 0$.\\
The second sort $(\Z,+,<,\equiv_d,0)$ admits elimination of quantifiers (\cite{Pr}) and the elimination of quantifiers of $(\k,+,\times,0,1)$ is a classical result of Chevalley since $\k$ is algebraically closed. J. Pas proved that  the three sorted language admits elimination of quantifiers \cite{Pa89}. This means that any subset of $\K^{n_1}\times\Z^{n_2}\times\k^{n_3}$ defined by a first order formula in this three sorts language (i.e. a logical formula involving 0, 1, +, $\times$ (but not $a\times b$ where $a$ and $b$ are integers), (, ), =, <, $\wedge$, $\vee$, $\neg$, $\forall$, $\exists$, $\nu$, $ac$, and variables for elements of $\K$, $\Z$ and $\k$ may be defined by a formula involving the same symbols except $\forall$, $\exists$.\\
Then we notice that $\b$ is defined by the following formula:
$$\left[\forall c\in\N\  \forall \ovl{y}\in \K^m\left( \nu(f(\ovl{y}))\geq \b(c)\right)\wedge\left(\nu(\ovl{y})\geq 0\right)\, \exists \wdt{y}\in \K^m\, \left(f(\wdt{y})=0 \wedge \nu(\wdt{y}-\ovl{y})\geq c\right)\right]$$
\begin{equation*}\begin{split}\wedge\left[\forall c\in\N\ \exists \ovl{y}\in \K^m\left( \nu(f(\ovl{y}))\geq \b(c)+1\right)\wedge\left(\nu(\ovl{y})\geq 0\right)\right.&\\
 \neg\exists \wdt{y}\in \K^m\  &\left.\left( f(\wdt{y})=0 \wedge \nu(\wdt{y}-\ovl{y})\geq c\right)\right]\end{split}\end{equation*}
Applying the latter elimination of quantifiers result we see that $\b(c)$ may be defined without $\forall$ and $\exists$. Thus $\b(c)$ is defined by a formula using $+$, <, $\equiv_d$ (for a finite set of integers $d$). This proves the result.\\
The case where $V=\Z_p$ requires more work since the residue field of $\Z_p$ is not algebraically closed, but the idea is the same.
  \end{proof}
 
 \begin{remark}\label{Greenberg_ana_set}
  When $V=\C\{ t\}$, $t$ being a single variable, it is tempting to link together the Greenberg's function of a system of equations with coefficients in $\C$, or even in $V$, and some geometric invariants of the germ of the complex set defined by this system of equations. This has been done in several cases:
  \begin{enumerate}
  \item[i)] In \cite{El}, a bound (involving the multiplicity and the Milnor number) of the Greenberg's function is given when the system of equations defines the germ of a curve in $(\C^m,0)$. 
  \item[ii)] Using Theorem \ref{Hi1} M. Hickel  gives the following bound of the Greenberg's function $\b$ of the germ of a complex hypersurface with an isolated singularity (cf. \cite{Hi93}): $\b(c)\leq \lfloor \lambda c\rfloor +c$ for all $c\in\N$, and this bound is sharp for plane curves. Here $\lambda$ denotes the \L ojasiewicz exponent of the germ, i.e.
  \begin{equation*}\begin{split}\lambda:=\inf\left\{\theta\in \R\ /\ \exists C>0 \ \exists U \text{ neighborhood of } 0 \right.&\text{ in }\C^m,\\
  |f(z)|+\left|\frac{\partial f}{\partial z_1}(z)\right|+\cdots+\left|\frac{\partial f}{\partial z_m}(z)\right|&\left.\geq C|z|^{\theta}\ \forall z\in U\right\}.\end{split}\end{equation*}
  \item[iii)] \cite{Hi04} gives the complete computation of the Greenberg's function of one branch of plane curve and proves that it is a topological invariant. This computation has been done for two branches in \cite{Sa}. Some particular cases depending on the Newton polygon of the plane curve singularity are computed in \cite{W2}.
  
  \item[iv)] In the case where $V$ is the ring of $p$-adic integers and the variety defined by $f(y)=0$ is non-degenerate with respect to  its Newton polyhedron, D. Bollaerts \cite{Bol} gives a bound on the infimum of numbers $a$ such that Theorem \ref{Greenberg} is satisfied for some constant $b$. This bound is defined in terms of  the Newton polyhedra of the components of $f$.
  
  \end{enumerate}
    \end{remark}
 
 Finally we mention the following recent result that extends Theorem \ref{Greenberg} to non-Noetherian valuation rings and whose proof is based on \index{ultraproduct} ultraproducts methods used in \cite{BDLvdD} to prove Theorem \ref{Greenberg} (see Section \ref{ultraproducts}):
 
 \begin{theorem}\cite{M-B}
 Let $V$ be a Henselian valuation ring and $\nu: V\lgw \G$ its associated valuation. Let us denote by $\wdh{V}$ its $m_V$-adic completion, $\K:=\Frac(V)$ and $\wdh{\K}:=\Frac(\wdh{V})$. Let us assume that $\K\lgw \wdh{\K}$ is a separable field extension. Then for any $f(y)\in V[y]^r$ there exist $a\in\N$, $b\in \G^+ $ such that
 
$$\forall c\in\G\  \forall \ovl{y}\in V^m\left( \nu(f(\ovl{y}))\geq ac+b\right)\Longrightarrow\ \exists \wdt{y}\in V^m\ \left(f(\wdt{y})=0 \wedge \nu(\wdt{y}-\ovl{y})\geq c\right).$$

  \end{theorem}


\subsection{Strong Artin Approximation: the general case}
\index{strong Artin approximation Theorem}
In the general case (when $V$ is not a valuation ring) there still exists an approximation function $\b$ analogous to the Greenberg's function. The analogue of Greenberg's Theorem in the general case is the following:
\begin{theorem}[Strong Artin Approximation Theorem]\cite{P-P,Po86}\label{SAP_formal} Let $A$ be a complete local ring whose maximal ideal is denoted by $\m_A$. Let  $f(y,z)\in A\lb y\rb[z]^r$, with $z:=(z_1,\ldots,z_s)$. Then there exists a function $\b : \N\lgw \N$ such that the following holds:\\
 For any $c\in\N$ and any $\ovl{y}\in (\m_A.A)^m$, $\ovl{z}\in A^s$ such that $f(\ovl{y},\ovl{z})\in \m_A^{\b(c)}$, there exist $\wdt{y}\in(\m_A.A)^m$ and $\wdt{z}\in A^s$  such that $f(\wdt{y},\wdt{z})=0$ and $\wdt{y}-\ovl{y}$, $\wdt{z}-\ovl{z}\in\m_A^c$.
 \end{theorem}

\begin{remark}
This theorem can be extended to the case where $A$ is an excellent Henselian local ring by using Theorem \ref{Pop_app}.\\
Let us also mention that there  exists a version of this theorem for analytic equations \cite{W1} or Weierstrass systems \cite{D-L}.
\end{remark}

In the case of polynomial equations over a field the approximation function $\b$ may be chosen to depend only on the degree of the equations and the number of variables:

\begin{theorem}\label{theo1}\cite{Ar69,BDLvdD}
For all  $n,m,d\in\N$, there exists a function  $\b_{n,m,d}:\N\lgw \N$ such that the following holds:\\
Let $\k$ be a field and set $x:=(x_1,\ldots,x_n)$ and $y:=(y_1,\ldots,y_m)$. Then for all $f(x,y)\in\k[x,y]^r$ of total degree $\leq d$, for all $c\in\N$, for all $\ovl{y}(x)\in\k\llbracket x\rrbracket^m$ such that  
 $$f(x,\ovl{y}(x))\in (x)^{\b_{n,m,d}(c)},$$  there exists $\wdt{y}(x)\in\k\llbracket x\rrbracket^m$ such that $f(\wdt{y}(x))=0$ and $\wdt{y}(x)-\ovl{y}(x)\in (x)^c$.\\
\end{theorem}

\begin{remark}\label{Lascar}
By following the proof of Theorem \ref{theo1} given in \cite{Ar69}, D. Lascar proved that there exists a recursive function $\b$ that satisfies the conclusion of Theorem \ref{theo1} \cite{La}. But the proof of Theorem \ref{theo1} uses a double induction on the height of the ideal (like in Theorem \ref{Greenberg}) and on $n$ (like in Theorem \ref{Ar68}). In particular, in order to apply the Jacobian Criterion, we need to work with prime ideals (at least radical ideals), and replace the original ideal $I$ generated by $f_1,\ldots,f_r$ by one of its associated primes and then make a reduction to the case of $n-1$ variables. But the bounds on the degree of the generators of such an associated prime may be very large compared to the degree of the generators of $I$. This is essentially the reason why the proof of this theorem does not give much more information about the quality of $\b$ than Lascar's  result.
\end{remark}

\begin{example}\cite{Spi}
Set $f(x_1,x_2,y_1,y_2):=x_1y_1^2-(x_1+x_2)y_2^2$.
Let $$\sqrt{1+t}=1+\sum_{n\geq 1}a_nt^n\in\Q\llbracket t\rrbracket$$ be the unique power series  such that $(\sqrt{1+t})^2=1+t$ and whose value at the origin is 1.
For every $c\in\N$ we  set $y_2^{(c)}(x):=x_1^c$ and $y_1^{(c)}(x):=x_1^c+\sum_{n=1}^ca_nx_1^{c-n}x_2^n$. Then 
$$f(x_1,x_2,y_1^{(c)}(x),y_2^{(c)}(x))\in (x_2)^c.$$
On the other side the equation  $f(x_1,x_2,y_1(x),y_2(x))=0$ has no other solution $(y_1(x)$, $y_2(x))\in\Q\llbracket x\rrbracket^2$ but $(0,0)$. This proves that
 Theorem \ref{SAP_formal} is not valid for general Henselian pairs because $(\Q\llbracket x_1,x_2\rrbracket, (x_2))$ is a Henselian pair.\\
 \\
 Let us notice that L. Moret-Bailly proved that if a pair $(A,I)$ satisfies Theorem \ref{SAP_formal}, then $A$ has to be an excellent Henselian local ring \cite{M-B'}. On the other hand A. More proved that a pair $(A,I)$, where $A$ is an equicharacteristic excellent regular Henselian local ring, satisfies Theorem \ref{SAP_formal} if and only if $I$ is $\m$-primary \cite{Mo}.\\
 It is still an open question to know under which conditions on $I$ the pair $(A,I)$ satisfies Theorem \ref{SAP_formal} when $A$ is a general excellent Henselian local ring.
 \end{example}

   \begin{remark}\label{no_sol}
   As for Theorem \ref{Greenberg}, Theorem \ref{SAP_formal} implies that, if $f(y)$ has no solution in $A$,  there exists a constant $c$ such that $f(y)$ has no solution in $\frac{A}{\m_A^c}$.
   \end{remark}

 \begin{definition}
 Let $f$ be as in Theorem \ref{SAP_formal}. The least function $\b$ that satisfies Theorem \ref{SAP_formal} is called the Artin function \index{Artin function} of $f$.
 \end{definition}
 
 \begin{remark}
 When $f$ is a vector of polynomials of $A[y]$ for some complete local ring $A$, then we can consider the Artin function of $f$ seen as a vector of formal power series in $y$ (i.e. we restrict to approximate solutions vanishing at 0) or we can consider the Artin function of $f$ seen as a vector of polynomials (i.e. we consider every approximate solution, not only the ones vanishing at 0). The two may not be equal in general even if the first one is bounded by the second one (exercise!). We hope that there will be no ambiguity in the rest of the text.
 \end{remark}
 
 \begin{remark}\label{int_closure2}\index{integral closure}
 As before, the Artin function of $f$ depends only on the integral closure of the ideal $I$ generated by $f_1,\ldots,f_r$ (see Lemma \ref{int_closure}). 
  \end{remark}
  
  \begin{remark}\label{smooth} (See also Remark \ref{smooth'} just below)
  Let $f(y)\in A[ y]^r$  and  $\ovl{y}\in (\m_A)^m$ satisfy $f(\ovl{y})\in \m_A^c$ and let us assume that $A\lgw B:=\frac{A[ y]_{\m_A+(y)}}{(f(y))}$ is a smooth morphism. This morphism is local thus it factors as $A\lgw C:=A[z]_{\m_A+(z)}\lgw B$ such that $C\lgw B$ is \'etale (see Definition \ref{standard_etale}) and $z:=(z_1,\ldots,z_s)$. We remark that $\ovl{y}$ defines a morphism of $A$-algebras $\phi : B\lgw \frac{A}{\m_A^c}$.  Let us choose any $\wdt{z}\in A^s$ such that $\ovl{z}_i-\wdt{z}_i\in\m_A^c$ for all $1\leq i\leq s$ ($\ovl{z}_i$ denotes the image of $z_i$ in $\frac{A}{\m_A^c}$). Then $A\lgw \frac{B}{(z_1-\wdt{z}_1,\ldots,z_s-\wdt{z})}$ is \'etale and admits a section in $\frac{A}{\m_A^c}$. By Proposition \ref{lift} this section lifts to a section in $A$. Thus we have a section $B\lgw A$ equal to $\phi$ modulo $\m_A^c$.\\
   This proves that $\b(c)=c$ when $A\lgw \frac{A[ y]_{\m_A+(y)}}{(f(y))}$ is smooth. \\
   On the other hand we can prove that if $\b$ is the identity function then $A\lgw \frac{A[ y]_{\m_A+(y)}}{(f(y))}$ is smooth \cite{Hi93}. This shows that the Artin function of $f$ may be seen as a measure of the non-smoothness of the morphism $A\lgw \frac{A[y]_{\m_A+(y)}}{(f)}$.

  \end{remark}
  
    \begin{remark}\label{smooth'}
 For the convenience of some readers we can express the previous remark in the setting of convergent power series equations. The proof is the same but the language is a bit different:\\
  Let $f(x,y)\in \k\{x,y\}^m$ be a vector of convergent power series in two sets of variables $x$ and $y$ where $y=(y_1,\ldots,y_m)$. Let us assume that $f(0,0)=0$ and
  $$\frac{\partial (f_1,\ldots,f_m)}{\partial( y_1,\ldots,y_m)}(0,0) \text{ is invertible.}$$
  Let $\ovl y(x)\in\k\lb x\rb^m$ be a vector of formal power series vanishing at the origin such that
  $$f(x,\ovl y(x))\in (x)^c$$
  for some integer $c$. We can write
  $$\ovl y(x)=y^0(x)+y^1(x)$$
  where $y^0(x)$ is a vector of polynomials of degree $ <c$ and $y^1(x)$ is a vector of formal power series whose components have order equal at least to $c$. We set 
  $$g(x,z):=f(x,y^0(x)+z)$$
  for new variables $z=(z_1,\ldots,z_m)$.
Then $g(0,0)=0$ and
$$\frac{\partial (g_1,\ldots,g_m)}{\partial( z_1,\ldots,z_m)}(0,0)=\frac{\partial (f_1,\ldots,f_m)}{\partial( y_1,\ldots,y_m)}(0,0) \text{ is invertible.}$$
By the Implicit Function Theorem for convergent power series there exists a unique vector of convergent power series $z(x)$ vanishing at the origin such that 
$$g(x,z(x))=0.$$

 Since
 $$g(x,z(x))=g(x,0)+\frac{\partial (g_1,\ldots,g_m)}{\partial( z_1,\ldots,z_m)}(0,0)\cdot z(x)+\e(x)$$
 where the components of $\e(x)$ are linear combinations of products of the components of $z(x)$, we have
 $$\ord(z_i(x))=\ord(g_i(x,0))\ \ \forall i.$$
 Moreover
 $$g(x,0)=f(x,y^0(x))=f(x,y(x)) \text{ modulo } (x)^c,$$
thus $z(x)\in (x)^c$. Thus $\wdt y(x):=y^0(x)+z(x)$ is a solution of $f(x,y)=0$ with
$$\wdt y(x)-\ovl y(x)\in (x)^c.$$
This shows that the Artin function of $f(x,y)$ is the identity function. 
 \end{remark}
 
 \subsection{Ultraproducts and  Strong Approximation type results}\label{ultraproducts}
 Historically M. Artin proved Theorem \ref{theo1} in \cite{Ar69} by a   modification of the proof of Theorem \ref{Ar68}, i.e. by  induction on $n$ using the Weierstrass Division Theorem. Roughly speaking it is a concatenation of his proof of Theorem \ref{Ar68} and of the proof of Greenberg's Theorem \ref{Greenberg}. Then several authors  provided proofs of  generalizations of his result using the same kind of proof but this was not always easy, in particular when the base field is not  a characteristic zero field (for example there is a gap in the inseparable case of \cite{P-P}). On the other hand, in 1970 A. Robinson gave a new proof of Greenberg's Theorem \cite{Rob} based on the use of ultraproducts. Then  ultraproducts methods have been successfully used to give more direct proofs of this kind of Strong Approximation type results (see \cite{BDLvdD} and \cite{D-L}; see also \cite{Po79} for the general case), even if the authors of these works seemed unaware of the work of A. Robinson. The general principle is the following:  ultraproducts  transform  approximate solutions into exact solutions of a given system of polynomial equations defined over a complete local ring $A$. So they are a tool to reduce Strong Artin Approximation Problems to Artin Approximation Problems.  But these new exact solutions are not living anymore in the given base ring $A$ but in bigger rings that also satisfy Theorem \ref{Pop_app}. In the case where the equations are not polynomial but analytic or formal, this reduction based on ultraproducts transforms the given equations into equations belonging to a different Weierstrass System (see Definition \ref{W-sys} and Theorem \ref{W}) which is a first justification to the introduction the Weierstrass Systems. We will present here the main ideas.\\
 \\
  Let us start with some terminology.
A \emph{filter} $D$ (over $\N$) is a non-empty subset of $\mathcal{P}(\N)$, the set of subsets of $\N$, that satisfies the following properties:
$$\text{a) }\emptyset\notin D,\ \ \ \text{b) } \E,\,\F\in D\Longrightarrow \E\cap \F\in D,\ \ \ \text{c) } \E\in D,\, \E\subset\F\Longrightarrow \F\in D.$$
A filter  $D$ is \emph{principal} if  $D=\{\F\ /\ \E\subset \F\}$ for some non empty subset $\E$ of $\N$. An \textit{ultrafilter} is a filter which is maximal for the inclusion. It is easy to check that a filter $D$ is an ultrafilter if and only if for any subset $\E$ of $\N$,  $D$ contains $\E$ or its complement  $\N-\E$.
In the same way an ultrafilter is non-principal if and only if it contains the filter $E:=\{\E\subset\N\ /\ \N-\E \text{ is finite}\}$. Zorn's  Lemma yields the existence of non-principal ultrafilters.\\
\\
Let $A$ be a Noetherian ring. Let $D$ be a  non-principal ultrafilter. We define the \emph{ultrapower} (or \emph{ultraproduct}) of $A$ as follows: \index{ultraproduct}
$$A^*:=\frac{\left\{(a_i)_{i\in\N}\in\prod_iA\right\}}{\left((a_i)\sim (b_i) \text{ iff } \{i\, /\, a_i=b_i\}\in D\right)}.$$
The ring structure of $A$ induces a ring structure on $A^*$
and the map $A\lgw A^*$ that sends $a$ onto the class of $(a)_{i\in\N}$ is a ring morphism.\\
\\
We have the following fundamental result that shows that several properties of $A$ are also satisfied by $A^*$:
\begin{theorem}[\L o\'s Theorem]\label{CKL}\cite{C-K} Let $L$ be a first order language, let $A$ be a structure for  $L$ and let  $D$ be an ultrafilter over $\N$. Then for any $(a_i)_{i\in\N}\in A^*$ and for any first order formula $\phi(x)$, $\phi((a_i))$ is true in $A^*$ if and only if $$\{i\in\N\ /\ \phi(a_i)\text{ is true in }A\}\in D.$$
\end{theorem}
Roughly speaking this statement means that any logical sentence involving the special elements and the operations of the language $L$ (for instance 0, 1,+  and $\times$ for the language of commutative rings) along with (, ), =, $\wedge$, $\vee$, $\neg$, $\forall$, $\exists$,  and variables for the elements of the structure is true in $A$ if and only if it is true in $A^*$.\\
In particular we can deduce the following properties:\\
The ultrapower $A^*$ is equipped with a structure of a commutative ring. If $A$ is a field then $A^*$ is a field. If $A$ is an algebraically closed field then $A^*$ is an algebraically closed field. If $A^*$ is a local ring with maximal ideal $\m_A$ then $A^*$ is a local ring with maximal ideal $\m_A^*$ defined by $(a_i)_i\in \m_A^*$ if and only if $\{i\,/\, a_i\in\m_A\}\in D$. If $A$ is a local Henselian ring, then $A^*$ is a local Henselian ring.  In fact all these properties are elementary and can be checked directly by hand without the help of Theorem \ref{CKL}. Elementary proofs of these results can be found in  \cite{BDLvdD}.\\
Nevertheless if $A$ is Noetherian, then $A^*$ is not  Noetherian in general, since Noetherianity is a condition on ideals of $A$ and not on elements of $A$. For example, if $A$ is a Noetherian local ring, then $\m^*_{\infty}:=\bigcap _{n\geq 0} {\m_A^*}^n\neq (0)$ in general. But we have the following lemma:

\begin{lemma}\cite{Po00}
Let $(A,\m_A)$ be a Noetherian complete local ring. Let us denote $A_1:=\frac{A^*}{\m_{\infty}^*}$. Then $A_1$ is a Noetherian complete local ring of the same dimension as $A$ and the composition $A\lgw A^*\lgw A_1$ is flat. 
\end{lemma}

In fact, since $A$ is excellent and $\m_AA_1$ is the maximal ideal of $A_1$, it is not difficult to prove that $A\lgw A_1$ is even a regular morphism. Details can be found in \cite{Po00}.\\
\\
Let us sketch the idea of the use of ultraproducts to prove the existence of an approximation function in the case of Theorem \ref{SAP_formal}:
\begin{proof}[Sketch of the proof of Theorem \ref{SAP_formal}]
 Let us assume that some system of algebraic equations over an excellent Henselian local ring $A$, denoted by $f=0$, does not satisfy Theorem \ref{SAP_formal}. Using Theorem \ref{Pop_app}, we may assume that $A$ is complete. Thus it means that there exists an integer $c_0\in\N$ and, for every $c\in\N$, there exists $\ovl{y}^{(c)}\in A^m$ such that $f(\ovl{y}^{(c)})\in\m_A^c$ and there is no  $\wdt{y}^{(c)}\in A^m$ solution of $f=0$  with $\wdt{y}^{(c)}-\ovl{y}^{(c)}\in \m_A^{c_0}$.\\
Let us denote by $\ovl{y}$ the image of $(\ovl{y}^{(c)})_c$ in $(A^*)^m$. Since $f(y)\in A[y]^r$, we may assume that $f(y)\in A^*[y]^r$ using the morphism $A\lgw A^*$. Then $f(\ovl{y})\in \m_{\infty}^*$. Thus $f(\ovl{y})=0$ in $A_1$. Let us choose $c>c_0$. Since $A\lgw A_1$ is regular, $A$ is Henselian and excellent (because $A$ is complete), we can copy the proof of Theorem \ref{Pop_app} to show that for any $c\in\N$ there exists $\wdt{y}\in A^m$ such that $f(\wdt{y})=0$ and $\wdt{y}-\ovl{y}\in \m_A^cA_1$. Thus $\wdt{y}-\ovl{y}\in \m_A^cA^*$. Hence the set $\{i\in\N\ /\ \wdt{y}-\ovl{y}^{(i)}\in\m_A^cA^*\}\in D$ is non-empty. This is a contradiction. \end{proof}

\begin{remark}\label{tobedone}
If, instead of working with polynomial equations over a general excellent Henselian local ring, we work with a more explicit subring of $\k\lb x,y\rb$ satisfying the Implicit Function Theorem and the Weierstrass Division Theorem (like the rings of algebraic or convergent power series) the use of ultraproducts enables us to reduce the problem of the existence of an approximation function to a problem of approximation of formal solutions of a system of equations by solutions in a Weierstrass System (see \cite{D-L}). This is also true in the case of constraints.
\end{remark}
We can also prove easily the following proposition with the help of ultraproducts (see also Theorem \ref{ex_Artin} of Example \ref{ex_art} in the introduction):

\begin{proposition}\cite{BDLvdD}\label{prop_mod}
Let $f(x,y)\in\C[x,y]^r$. For any $1\leq i\leq m$ let $J_i$ be a subset of $\{1,\ldots,n\}$. \\Let us assume that for every $c\in\N$ 
there exist $\ovl{y}_i^{(c)}(x)\in\C[x_j, j\in J_i]$, $1\leq i\leq m$,  such that $$f(x,\ovl{y}^{(c)}(x))\in (x)^c.$$ Then there exist $\wdt{y}_i(x)\in\C\llbracket x_j, j\in J_i\rrbracket$, $1\leq i\leq m$, such that $f(x,\wdt{y}(x))=0$.
\end{proposition}

\begin{proof}
Let us denote by $\ovl{y}\in \C[x]^*$ the image of $(\ovl{y}^{(c)})_c$. Then $f(x,\ovl{y})=0$ modulo $(x)^*_{\infty}$. It is not very difficult to check that $\frac{\C[x]^*}{(x)^*_{\infty}}\simeq \C^*\llbracket x\rrbracket$ as $\C^*[x]$-algebras. Moreover $\C^*\simeq \C$ as $\k$-algebras (where $\k$ is the subfield of $\C$ generated by the coefficients of $f$). Indeed both  are field of transcendence degree over $\Q$ equal to the cardinality of the continuum, so their transcendence degree over $\k$ is also the cardinality of the continuum. Since both are algebraically closed they are isomorphic over $\k$. Then the image of $\ovl{y}$ by the isomorphism yields the desired solution in $\C\llbracket x\rrbracket$.
 \end{proof}
Let us remark that the proof of this result remains valid if we replace $\C$ by any uncountable algebraically closed field $\K$.  If we replace $\C$ by $\Q$, this result is no more true in general (see Example \ref{ex_Q}).

\begin{remark}
Several authors proved "uniform" Strong Artin approximation results, i.e. they proved the existence of a function $\b$ satisfying Theorem \ref{SAP_formal} for a parametrized family of  equations $(f_{\la}(y,z))_{\la\in\Lambda}$ which satisfy tameness properties that we do not describe here (essentially this condition is that the coefficients of $f_{\la}(y,z)$ depend analytically on the parameter $\la$). The main example is Theorem \ref{theo1} that asserts that the Artin functions of polynomials in $n+m$ variables of degree less than $d$ are uniformly bounded.  There are also two types of proof for these kind of "uniform" Strong Artin approximation results: the ones using ultraproducts (see Theorems 8.2  and 8.4 of \cite{D-L} where uniform Strong Artin approximation results are proven for families of polynomials whose coefficients depend analytically on some parameters) and the ones using the scheme of proof due to Artin (see \cite{EKT} where more or less the same results as those of  \cite{BDLvdD} and \cite{D-L} are proven).
\end{remark}


\subsection{Effective examples of Artin functions }\label{effective}
In general the proofs of Strong Artin Approximation results do not give much information about  the Artin functions. Indeed there are two kinds of proofs: the proofs based on ultraproducts methods use a proof by contradiction and are not effective, and the proofs based on the classical argument of Greenberg and Artin are not direct and require too many steps (see also Remark \ref{Lascar}). In fact this latter kind of proof gives uniform versions the Strong Artin Approximation Theorem (as Theorem \ref{theo1}) which is a more general result. Thus this kind of proof is not optimal to bound effectively  a given Artin function. The problem of finding  estimates of Artin functions was first raised  in \cite{Ar70} and only a very few general results are known (the only ones in the case of Greenberg's Theorem  are Theorems  \ref{Hi1}, \ref{De84} and Remark \ref{Greenberg_ana_set},  and Remark \ref{Lascar} in the general case). We present here a list of examples of equations for which we can bound the Artin function\index{Artin function}.


\subsubsection{The Artin-Rees Lemma}
\index{Artin-Rees Lemma}
The following result has been known for long by the specialists without appearing in the literature and has been communicated to the author by M. Hickel:
\begin{theorem}\cite{Ro3}\label{linear}
Let $f(y)\in A[y]^r$ be a  vector of  linear homogeneous polynomials  with coefficients in a Noetherian ring $A$. Let $I$ be an ideal of $A$. Then there exists a constant $c_0\geq 0$ such that:
$$\forall c\in\N \ \forall \ovl{y}\in A^m \text{ such that  } f(\ovl{y})\in I^{c+c_0}$$
$$\exists  \wdt{y}\in A^m \text{ such that  } f(\wdt{y})=0 \text{ and } \wdt{y}-\ovl{y}\in I^c.$$
\end{theorem}

This theorem asserts that the Artin function of $f$ is bounded by the function $c\lgm c+c_0$.
Moreover let us remark that this theorem is valid for any Noetherian ring and any ideal $I$ of $A$. This can be compared with the fact that, for linear equations, Theorem \ref{Pop_app} is true for any Noetherian ring $A$ and that the Henselian condition is unnecessary in this situation (see Remark \ref{flat}).

\begin{proof}
For convenience, let us assume that there is only one linear polynomial: $$f(y)=a_1y_1+\cdots+a_my_m.$$ Let us denote by $\I$ the ideal of $A$ generated by $a_1,\ldots,a_m$. The Artin-Rees Lemma implies that there exists $c_0>0$ such that $\I\cap  I^{c+c_0}\subset \I.I^c$ for any $c\geq 0$.\\
If $\ovl{y}\in A^m$ is such that $f(\ovl{y})\in I^{c+c_0}$ and since $f(\ovl{y})\in \I$,  there exists $\e\in(\I^{c}A)^m$ such that $f(\ovl{y})=f(\e)$. If we define $\wdt{y}_i:=\ovl{y}_i-\e_i$, for $1\leq i\leq m$, we have the result.
 \end{proof}

We have the following result whose proof is similar:
\begin{proposition}\label{quotient}
Let $(A,\m_A)$ be an  excellent Henselian local ring, $I$ an ideal of $A$ generated by $a_1,\ldots,a_q$ and $f(y)\in A[y]^r$. Set $$F_i(y,z):=f_i(y)+a_1z_{i,1}+\cdots+a_qz_{i,q}\in A[y,z],\ 1\leq i\leq r$$ where the $z_{i,k}$  are new variables and let $F(y,z)$ be the vector whose coordinates are the $F_i(y,z)$. Let us denote by $\b$ the Artin function of $f(y)$ seen as a vector of polynomials of $\frac{A}{I}[y]$ and $\g$ the Artin function of $F(y,z)\in A[y,z]^r$. Then there exists a constant $c_0$ such that:
$$\b(c)\leq \g(c)\leq \b(c+c_0),\ \ \forall c\in\N.$$
\end{proposition}

\begin{proof}
Let $\ovl{y}\in \frac{A}{I}^m$ satisfy $f(\ovl{y})\in\m_A^{\g(c)}\frac{A}{I}^r$. Then there exists $\ovl{z}\in A^{qr}$ such that $F(\ovl{y},\ovl{z})\in\m_A^{\g(c)}$ (we denote again by $\ovl{y}$ a lifting of $\ovl{y}$ in $A^m$).  Thus there exist $\wdt{y}\in A^m$  and $\wdt{z}\in A^{qr}$ such that $F(\wdt{y},\wdt{z})=0$ and $\wdt{y}-\ovl{y}$, $\wdt{z}-\ovl{z}\in \m_A^{c}$. Thus $f(\wdt{y})=0$ in $\frac{A}{I}^r$.\\
On the other hand let $c_0$ be a constant such that $I\cap \m_A^{c+c_0}\subset I.\m_A^c$ for all $c\in\N$ (such constant exists by  Artin-Rees Lemma). Let $\ovl{y}\in A^m$, $\ovl{z}\in A^{qr}$ satisfy $F(\ovl{y},\ovl{z})\in\m_A^{\b(c+c_0)}$. Then $f(\ovl{y})\in \m_A^{\b(c+c_0)}+I$. Thus there exists $\wdt{y}\in A^m$ such that $f(\wdt{y})\in I$ and $\wdt{y}-\ovl{y}\in \m_A^{c+c_0}$. Thus $F(\wdt{y},\ovl{z})\in \m_A^{c+c_0}\cap  I$. Then we conclude by following the proof of Theorem \ref{linear}.
 \end{proof}

\begin{remark}\label{reduction_series}
By Theorem \ref{Pop_app}, in order to  study the behaviour of the Artin function of some ideal we may assume that $A$ is a complete local ring. Let us assume that $A$ is an equicharacteristic local ring.  Then $A$ is the quotient of a power series ring over a field by Cohen Structure Theorem \cite{Ma}. Thus Proposition \ref{quotient} enables us to reduce the problem to the case $A=\k\lb x_1,\ldots,x_n\rb$ where $\k$ is a field.
\end{remark}


\subsubsection{Izumi's  Theorem and Diophantine Approximation}

Let $(A,\m_A)$ be a Noetherian local ring. We denote by $\nu$ the $\m_A$-adic order on $A$, i.e.
$$\nu(x):=\max\{n\in\N\ /\ x\in\m_A^n\}\  \ Ê \text{ for any }x\neq 0.$$
We always have $\nu(x)+\nu(y)\leq \nu(xy)$ for all $x$, $y\in A$. But we do not have the equality in general. For instance, if $A:=\frac{\C\lb x,y\rb}{(x^2-y^3)}$ then $\nu(x)=\nu(y)=1$ but $\nu(x^2)=\nu(y^3)=3$. Nevertheless we have the following theorem: \index{Izumi's  Theorem}

\begin{theorem}[Izumi's  Theorem]\label{Iz}\cite{Iz,Re}
Let $A$ be a local Noetherian ring whose maximal ideal is denoted by $\m_A$. Let us assume that $A$ is analytically irreducible, i.e. $\wdh{A}$ is irreducible. Then there exist $b\geq 1$ and $d\geq 0$ such that 
$$\forall x,y\in A,\ \ \nu(xy)\leq b(\nu(x)+\nu(y))+d.$$
\end{theorem}

This result implies easily the following corollary using Proposition \ref{quotient}:

\begin{corollary}\cite{Iz95,Ro3}
Let us consider the polynomial $$f(y):=y_1y_2+a_3y_3+\cdots+a_my_m,$$ with $a_3,\ldots,a_m\in A$ where $(A,\m_A)$ is a Noetherian local ring such that $\frac{A}{(a_3,\ldots,a_m)}$ is analytically irreducible. Then there exist $b'\geq 1$ and $d'\geq 0$ such that the Artin function $\b$ of  $f$ satisfies $\b(c)\leq b'c+d'$ for all $c\in\N$. \end{corollary}

\begin{proof}
By Proposition \ref{quotient} we have to prove that the Artin function  of $y_1y_2\in A[y]$ is bounded by a linear function if $A$ is analytically irreducible. Thus let $\ovl{y}_1$, $\ovl{y}_2\in A$ satisfy $\ovl{y}_1\ovl{y}_2\in\m_A^{2bc+d}$ where $b$ and $d$ satisfy Theorem \ref{Iz}. This means that 
$$2bc+d\leq \nu(\ovl{y}_1\ovl{y}_2)\leq b(\nu(\ovl{y}_1)+\nu(\ovl{y}_2))+d.$$
Thus $\nu(\ovl{y}_1)\geq c$ or $\nu(\ovl{y}_2)\geq c$. In the first case we define $\wdt{y}_1=0$ and $\wdt{y}_2=\ovl{y}_2$, and in the second case we define $\wdt{y}_1=\ovl{y}_1$ and $\wdt{y}_2=0$. Then $\wdt{y}_1\wdt{y}_2=0$ and $\wdt{y}_1-\ovl{y}_1$, $\wdt{y}_2-\ovl{y}_2\in\m_A^c$.
 \end{proof}

\begin{proof}[Idea of the proof of Theorem \ref{Iz} in the complex analytic case:]
Let $\ovl \nu$ denote the local reduced order of $A$:
$$\forall x\in A,\   \ovl\nu(x)=\lim_n\frac{\nu(x^n)}{n}.$$
By a theorem of D. Rees \cite{Re2} there exists a constant $d\geq 0$ such that
$$\forall x\in A,\ \nu(x)\leq\ovl\nu(x)\leq \nu(x)+d.$$
According to the theory of Rees valuations, there exist discrete valuations $\nu_1,\ldots,\nu_k$ such that $\ovl\nu(x)=\min\{\nu_1(x),\ldots,\nu_k(x)\}$ (they are called the Rees valuations of $\m_A$ - see \cite{H-S}). The valuation rings associated to $\nu_1,\ldots,\nu_k$ are the valuation rings associated to the irreducible components of the exceptional divisor of the normalized blowup of $\m_A$.\\
Since $\nu_i(xy)=\nu_i(x)+\nu_i(y)$ for any $i$, in order to prove the theorem we have to show that there exists a constant $a\geq 1$ such that 
$$\forall x\in A,\forall i, j,\ \ \nu_i(x)\leq a\nu_j(x).$$
 Indeed let $x$, $y\in A$ and assume that $\ovl\nu(x)=\nu_{i_1}(x)$, $\ovl\nu(y)=\nu_{i_2}(y)$ and $\min_i\{\nu_i(x)+\nu_i(y)\}= \nu_{i_0}(x)+\nu_{i_0}(y)$. Then we would that
$$\nu(xy)\leq \ovl\nu(xy)=\min_i\{\nu_i(xy)\}=\min_i\{\nu_i(x)+\nu_i(y)\}= \nu_{i_0}(x)+\nu_{i_0}(y)$$
$$\leq a\nu_{i_1}(x)+a\nu_{i_2}(y)=a(\ovl\nu(x)+\ovl\nu(y))\leq a(\nu(x)+\nu(y))+2ad.$$

 If $A$ is a complex analytic local ring, following S. Izumi's  proof, we may reduce the problem to the case $\dim(A)=2$ by using a Bertini type theorem.
 Then we consider a resolution of singularities of Spec$(A)$ (denoted by $\pi$) that factors through the normalized blow-up of $\m_A$. In this case let us denote by $E_1,\ldots,E_s$ the irreducible components of the exceptional divisor of $\pi$ and  set $e_{i,j}:=E_i.E_j$ for all $1\leq i,j\leq s$. Since $\pi$ factors through the normalized blow-up of $\m_A$, the Rees valuations $\nu_i$ are valuations associated to some of the $E_i$, let us say to $E_1,\ldots,E_k$. By extension we denote by $\nu_i$ the valuation associated to $E_i$ for any $i$.\\
 Let $x$ be an element of $A$. This element defines the germ of an analytic hypersurface whose total transform $T_x$ may be written $T_x=S_x+\sum_{j=1}^sm_jE_j$ where $S_x$ is the strict transform of $\{x=0\}$ and $m_i=\nu_i(x)$, $1\leq i\leq s$. Then we have
$$0=T_x.E_i=S_x.E_i+\sum_{j=1}^sm_je_{i,j}.$$
Since $S_x.E_i\geq 0$ for any $i$,  the vector $(m_1,\ldots,m_s)$ is contained in the closed convex  cone $C$ defined by $ m_i\geq 0$, $1\leq i\leq s$, and $\sum_{j=1}^se_{i,j} m_j\leq 0$, $1\leq i\leq s$. This cone $C$ is called the Lipman cone of $\Spec(A)$ and it is well known that it has a minimal element $\wdt m$ \cite{Ar0} (i.e. $\forall m\in C$, $\wdt m_i\leq m_i$ for all $1\leq i\leq s$). Thus to prove the theorem, it is enough to prove that $C$ is included in $ \{m\ /\ m_i>0,\ 1\leq i\leq s\}$, i.e. every component of $\wdt m$ is positive. Let assume that it is not the case. Then, after renumbering the $E_i$, we may assume that $( m_1,\ldots, m_l,0,\ldots,0)\in C$ where $ m_i>0$, $1\leq i\leq l<s$. Since $e_{i,j}\geq 0$ for all $i\neq j$, $\sum_{j=1}^se_{i,j} m_j=0$ for $l< i\leq s$ implies that $e_{i,j}=0$ for all $l<i\leq s$ and $1\leq j\leq l$. This contradicts the fact that the exceptional divisor of $\pi$ is connected (since $A$ is an integral domain).
 \end{proof}

Let us mention that Izumi's  Theorem is the key ingredient in proving  the following \index{Diophantine Approximation} analogue of Liouville's theorem on diophantine approximation: 

\begin{corollary}\label{diop}\cite{Ro4,Hi08,I-I,HII}
Let $(A,\m_A)$ be an excellent  Henselian   local domain. Let us denote respectively by $\K$ and $\wdh{\K}$ the fraction fields of $A$ and $\wdh{A}$. Let $z\in \wdh{\K}\backslash\K$ be algebraic over $\K$. Then
$$\exists a\geq 1, C\geq 0, \forall x\in A\  \forall y\in A\backslash\{0\},\  \ \left|z-\frac{x}{y}\right|\geq C|y|^a$$
where $|u|:=e^{-\nu(u)}$ and $\nu$ is the usual $\m_A$-adic valuation.

\end{corollary}

This result is equivalent to the following:

\begin{corollary}\label{homog}\cite{Ro4,Hi08,I-I,HII}
Let $(A,\m_A)$ be an excellent Henselian local domain and let $f_1(y_1,y_2),\ldots,f_r(y_1,y_2)\in A[y_1,y_2]$ be homogeneous polynomials. Then the Artin function of $(f_1,\ldots,f_r)$ is bounded by a linear function.

\end{corollary}


\subsubsection{Reduction to one quadratic equation and examples}  
In general Artin functions are not bounded by linear functions as in Theorem \ref{Greenberg}. Here is such an example:
\begin{example}\cite{Ro2}
Set $f(y_1,y_2,y_3):=y_1^2-y_2^2y_3\in\k\llbracket x_1,x_2\rrbracket[y_1,y_2,y_3]$ where $\k$ is a field of characteristic zero. Let us denote by $h(T):=\sum_{i=1}^{\infty}a_iT^i\in\Q\llbracket T\rrbracket$ the power series such that $(1+h(T))^2=1+T$. Let us define for every integer $c$:
$$y_1^{(c)}:=x_1^{2c+2}\left(1+\sum_{i=1}^{c+1}a_i\frac{x_2^{ci}}{x_1^{2i}}\right)=x_1^{2c+2}+\sum_{i=1}^{c+1}a_ix_1^{2(c-i+1)}x_2^{ci},$$
$$y_2^{(c)}:=x_1^{2c+1},$$
$$y_3^{(c)}:=x_1^2+x_2^c.$$
Then in the ring $\k(\frac{x_2}{x_1})\llbracket x_1\rrbracket$  we have 
$$f(y_1^{(c)},y_2^{(c)},y_3^{(c)})=\left(\left(\frac{y_1^{(c)}}{y_2^{(c)}}\right)^2-y_3^{(c)}\right){y_2^{(c)}}^2=\left(\left(\frac{y_1^{(c)}}{y_2^{(c)}}\right)^2-x_1^2\left(1+\frac{x_2^c}{x_1^2}\right)\right){y_2^{(c)}}^2$$
$$\hspace{2cm}=\left(\frac{y_1^{(c)}}{y_2^{(c)}}-x_1\left(1+h\left(\frac{x_2^c}{x_1^2}\right)\right)\right)\left(\frac{y_1^{(c)}}{y_2^{(c)}}+x_1\left(1+h\left(\frac{x_2^c}{x_1^2}\right)\right)\right){y_2^{(c)}}^2.$$
Thus we see that $f(y_1^{(c)},y_2^{(c)},y_3^{(c)})\in (x)^{c^2+4c}$ for all $c\geq 2$. On the other hand for any $(\wdt{y}_1$, $\wdt{y}_2$, $\wdt{y}_3)\in\k\llbracket x_1,x_2\rrbracket^3$  solution of $f=0$ we have the following two cases:
\begin{enumerate}
\item[1)] Either $\wdt{y}_3$ is a square in $\k\llbracket x_1,x_2\rrbracket$. But $\sup_{z\in\k\llbracket x\rrbracket}(\ord(y_3^{(c)}-z^2))=c$.
\item[2)] Either $\wdt{y}_3$ is not a square, hence $\wdt{y}_1=\wdt{y}_2=0$ since $\wdt y_1^2-\wdt y_2^2\wdt y_3=0$. But we have  $\ord(y_1^{(c)})-1=\ord(y_2^{(c)})=2c+1$.
\end{enumerate}
Hence in any case we have
$$\sup_{(\wdt y_1,\wdt y_2,\wdt y_3)}(\min\{\ord(y_1^{(c)}-\wdt{y}_1),\ord(y_2^{(c)}-\wdt{y}_2),\ord(y_3^{(c)}-\wdt{y}_3)\})\leq 2c+1$$
where $(\wdt y_1,\wdt y_2,\wdt y_3)$ runs over all the solutions of $f=0$.
This proves that the Artin function $f$ is bounded from below by a polynomial function of degree 2. Thus Theorem \ref{Greenberg} does not extend to $\k\llbracket x_1,\ldots,x_n\rrbracket$ if $n\geq 2$.\\

\end{example}
In \cite{Ro3} another example is given: the Artin function of the polynomial  $y_1y_2-y_3y_4\in\k\lb x_1,x_2,x_3\rb[y_1,y_2,y_3,y_4]$ is bounded from below by a polynomial function of degree 2. Both examples are the only known examples of Artin functions which are not bounded by a linear function.\\
We can remark that both examples are given by binomial equations (and the key fact to study these examples is that the ring of formal power series is a UFD). In the binomial case we can find upper bounds of  the Artin functions as follows:

\begin{theorem}\cite{Ro10,Ro}
 Let $\k$ be an algebraically closed field of characteristic zero. Let  $I$ be an ideal of $\k\lb x_1,x_2\rrbracket[y]$. If $I$ is generated by binomials of $\k[y]$ or if $\Spec(\k\lb x_1,x_2\rb[y]/I)$ has an isolated singularity then the Artin function of $I$ is bounded  by a function which is doubly exponential, i.e. a function of the form $c\lgm a^{a^c}$ for some constant $a>1$. \\
  \end{theorem}
  Moreover the Artin function of $I$ is bounded by a linear function if the approximate solutions are not too close to the singular locus of $I$ \cite{Ro}. We do not know if this doubly exponential bound is sharp since there is no known example of Artin function whose growth is greater than a polynomial function of degree 2.

In general, in order to investigate bounds on the growth of Artin functions, we can reduce the problem as follows, using a trick from \cite{Ro12}. Let us recall that we may assume that $A=\k\lb x_1,\ldots,x_n\rb$ where $\k$ is a field (see Remark \ref{reduction_series}).\\

 \begin{lemma}\cite{Be77}\label{trick_B} Let $A=\k\lb x_1,\ldots,x_n\rb$ where $\k$ is a field. For any $f(y)\in A[y]^r$ or $A\lb y\rb^r$ the  Artin function of $f$ is bounded by  the Artin function of 
 $$g(y):= f_1(y)^2+x_1\left(f_2(y)^2+x_1(f_3(y)^2+\cdots)^2\right)^2.$$
 \end{lemma}
 \begin{proof}
  Indeed,  if $\b$ is the Artin function of $g$ and if $f(\wdh{y})\in \m_A^{\b(c)}$ then $g(\wdh{y})\in\m_A^{\b(c)}$. Thus there exists $\wdt{y}\in A^m$ such that $g(\wdt{y})=0$ and $\wdt{y}_i-\wdh{y}_i\in\m_A^c$. But since $x_1$ is not a square in $A$ we have that  $g(\wdt{y})=0$ if and only if $f(\wdt{y})=0$. This proves the lemma. \end{proof}
  This allows us to assume that $r=1$ and we define $f(y):=f_1(y)$. If $f(y)$ is not irreducible, then we may write $f=h_1\ldots h_s$, where $h_i\in A\lb y\rb$ is irreducible for $1\leq i\leq s$, and the Artin function of $f$ is bounded by the sum of the Artin functions of the $h_i$. Hence we may assume that $f(y)$ is irreducible.\\
  We have the following lemma:
  
  \begin{lemma}
  For any $f(y)\in A\llbracket y\rrbracket$, where $A$ is a complete local ring, the Artin function of $f(y)$ is bounded by the Artin function of the polynomial
  $$P(u,x,z):=f(y)u+x_1z_1+\cdots+x_mz_m\in B[x,z,u]$$ where $B:=A\llbracket y\rrbracket$.
  
  \end{lemma}
  
  \begin{proof}
 Let us assume that $f(\ovl{y})\in \m_A^{\b(c)}$ where $\b$ is the Artin function of $P$ and $\ovl y_i\in\m_A$ for every $i$.\\
  Then, by Lemma \ref{Taylor}, there exist $\ovl{z}_{i}(y)\in A\llbracket y\rrbracket$, $1\leq i\leq m$, such that
$$f(y)+\sum_{i=1}^m(y_i-\ovl{y}_i)\ovl{z}_i(y)\in (\m_A+(y))^{\b(c)}.$$
Thus there exist $u(y)$, $f_i(y),\,z_i(y)\in A\llbracket y\rrbracket$, $1\leq i\leq m$, such that
$$u(y)-1,\,z_i(y)-\ovl{z}_i(y),\,x_i(y)-(y_i-\ovl{y}_i)\in (\m_A+(y))^c, 1\leq i\leq n$$
$$\text{ and }f(y)u(y)+\sum_{i=1}^mx_i(y)z_i(y)=0.$$
In particular $u(y)$ is invertible in $A\llbracket y\rrbracket$ if $c>0$. Let us assume that  $c\geq 2$. In this case the determinant of the matrix of the partial derivatives of $(x_i(y),\ 1\leq i\leq m)$ with respect to $y_1,\ldots,\,y_m$ is equal to 1 modulo $\m_A+(y)$, since $\ovl y_i\in\m_A$ for every $i$. By  Hensel's Lemma there exist $y_{j,c}\in \m_A$ such that
$x_i(y_{1,c},\ldots,\,y_{m,c})=0\text{ for } 1\leq i\leq m.$
Hence, since $u(y_{i,c})$ is invertible,
$f(y_{1,c},\ldots,\,y_{m,c})=0$
 and $y_{i,c}-\ovl{y}_i\in \m_A^c,\ 1\leq i\leq m.$
    \end{proof}
  
Thus, by Proposition \ref{quotient}, in order to study the general growth of Artin functions, it is enough to study the Artin function of the polynomial 
$$y_1y_2+y_3y_4+\cdots+y_{2m+1}y_{2m}\in A[y]$$ where $A$ is a complete local ring.

 \section {Examples of Applications}
 Artin approximation Theorems have numerous applications in commutative algebra, local analytic geometry, algebraic geometry, analysis. Most of these applications require extra material that is too long to be presented here 
 so we choose to present only basic applications to commutative algebra of Theorem \ref{Pop_app} and Theorem \ref{SAP_formal}. 
    
   \begin{proposition}\label{int_domain}
 Let $A$ be an excellent Henselian local ring. Then $A$ is reduced (resp. is an integral domain, resp. an integrally closed domain) if and only if $\wdh{A}$ is reduced (resp. is an integral domain, resp. an integrally closed domain).
 \end{proposition}

\begin{proof}\label{red}
First of all it is clear that $A$ is reduced (resp. is an integral domain) if $\wdh{A}$ is reduced (resp. is an integral domain). Thus for these two properties we only need to prove the converse.\\
If $\wdh{A}$ is not reduced, then there exists $\wdh{y}\in\wdh{A}$, $\wdh{y}\neq 0$, such that $\wdh{y}^k=0$ for some positive integer $k$. Thus we apply Theorem \ref{Pop_app} to the polynomial $y^k$ with $c\geq \ord(\wdh{y})+1$ in order to find $\wdt{y}\in A$ such that $\wdt{y}^k=0$ and $\wdt{y}\neq 0$. So $A$ is not reduced.\\
In order to prove that $\wdh{A}$ is an integral domain if $A$ is an integral domain, we apply the same procedure to the polynomial $y_1y_2$.\\
Now if $\wdh A$ is an integrally closed domain and $f/g\in\Frac(A)$ is integral over $A$, then $f/g$ is integral over $\wdh A$  and so  is in $\wdh A$. Thus there is $\wdh h\in\wdh A$ such that $f=g\wdh h$. By Theorem \ref{Pop_app} applied to the equation $f-gy=0$ we see that there is $h\in A$ such that $f=gh$ so $f/g\in A$. This shows that $A$ is integrally closed.\\
If $A$ is an integrally closed domain, then $A$ is an integral domain. Let  $P(z):=z^d+\wdh{a}_1z^{d-1}+\cdots+\wdh{a}_d\in \wdh{A}[z]$, $\wdh{f}$, $\wdh{g}\in \wdh{A}$, $\wdh{g}\neq 0$, satisfy $P\left(\frac{\wdh{f}}{\wdh{g}}\right)=0$, i.e. $\wdh{f}^d+\wdh{a}_1\wdh{f}^{d-1}\wdh{g}+\cdots+\wdh{a}_d\wdh{g}^d=0$. By Theorem \ref{Pop_app}, for any $c\in \N$, there exist $\wdt{a}_{i,c}$, $\wdt{f}_c$, $\wdt{g}_c\in A$ such that $\wdt{f}_c^d+\wdt{a}_{1,c}\wdt{f}_c^{d-1}\wdt{g}_c+\cdots+\wdt{a}_{d,c}\wdt{g}_c^d=0$ and $\wdt{f}_c-\wdh{f}$, $\wdt{g}_c-\wdh{g}\in \m_A^c\wdh A$. Then for $c>c_0$, where $c_0=\ord(\wdh g)$, we have $\wdt{g}_c\neq 0$. Since $A$ is an integrally closed domain,  $\wdt{f}_c\in (\wdt{g}_c)$ for $c>c_0$. Thus $\wdh{f}\in (\wdh{g})+\m^c$ for every $c$ large enough. By Nakayama  Lemma this implies that $\wdh{f}\in (\wdh{g})$ and $\wdh{A}$ is integrally closed.
 \end{proof}
 
 \begin{proposition}\label{UFD}\cite{KPPRM,Po86}
Let $A$ be an excellent Henselian local domain. Then $A$ is a unique factorization domain if and only if $\wdh{A}$ is a unique factorization domain. 
\end{proposition}

\begin{proof}
If $\wdh{A}$ is a unique factorization domain, then any irreducible element of $\wdh{A}$ is prime. Now let $a\in A$ be an irreducible element of $A$. By applying Theorem \ref{Pop_app} to the polynomial $P:=a-y_1y_2$ we see that $a$ remains irreducible in $\wdh A$ (indeed if $a$ is not irreducible in $\wdh A$ then $P$ has a solution $(\wdh y_1,\wdh y_2)$ where the $\wdh y_i$ are in $\m \wdh A$ - so Theorem \ref{Pop_app} shows that $P$ has a solution $(\wdt y_1,\wdt y_2)$ where the $\wdt y_i\in\m$ contradicting the fact that $a$ is irreducible). So $a$  is prime in $\wdh A$, thus it is prime in $A$. Since $A$ is a Noetherian integral domain, this proves that $A$ is a unique factorization domain.\\
Let us assume that $\wdh{A}$ is  not a unique factorization domain (but it is Noetherian since $A$ is Noetherian). Thus there exists an irreducible element $\wdh{x}_1\in\wdh{A}$ that is not prime. This is equivalent to the following assertion:
$$\exists \wdh{x}_2, \wdh{x}_3, \wdh{x}_4\in \wdh{A}\  \text{ such that } \ \wdh{x}_1\wdh{x}_2-\wdh{x}_3\wdh{x}_4=0$$
$$ \not\exists \wdh{z}_1\in\wdh A   \text{ such that  }  \wdh{x}_1\wdh{z}_1-\wdh{x}_3=0$$
$$ \not\exists  \wdh{z}_2\in\wdh{A}  \text{ such that  }  \wdh{x}_2\wdh{z}_2-\wdh{x}_4=0$$
$$\text{and }\not\exists \wdh{y}_1, \wdh{y}_2\in\m_A\wdh{A}\ \text{ such that  } \wdh{y}_1\wdh{y}_2-\wdh{x}_1=0.$$
Let us denote by $\b$ the Artin function of 
$$f(y,z):=(\wdh{x}_1z_1-\wdh{x}_3)( \wdh{x}_2z_2-\wdh{x}_4)( y_1y_2-\wdh{x}_1)\in \wdh{A}\lb y\rb[z].$$ Since $f(y,z)$ has no solution in $(\m_A\wdh{A})^2\times\wdh{A}^2$, by Remark \ref{no_sol} $\b$ is a constant, and $f(y,z)$ has no solution in $(\m_A\wdh{A})^2\times\wdh{A}^2$ modulo $\m_A^{\b}$.\\
On the other hand by Theorem \ref{Pop_app} applied to $x_1x_2-x_3x_4$,  there exists $\wdt{x}_i\in A$, $1\leq i\leq 4$, such that $\wdt{x}_1\wdt{x}_2-\wdt{x}_3\wdt{x}_4=0$ and $\wdt{x}_i-\wdh{x}_i\in\m_A^{\b+1}$, $1\leq i\leq 4$. Hence
$$g(y,z):=(\wdt{x}_1z_1-\wdt{x}_3)( \wdt{x}_2z_2-\wdt{x}_4)( y_1y_2-\wdt{x}_1)\in \wdh{A}\lb y\rb[z]$$ has no solution in $(\m_A\wdh{A})^2\times\wdh{A}^2$ modulo $\m_A^{\b}$, hence has no solution in $(\m_AA)^2\times A^2$. This means that $\wdt{x}_1$ is an irreducible element of $A$ but it is not prime. Hence $A$ is not a unique factorization domain.
 \end{proof}

\begin{proposition}\label{primary}
Let $A$ be an excellent Henselian local ring. Let $Q$ be a primary ideal of $A$. Then $Q\wdh{A}$ is a primary ideal of $\wdh{A}$.
\end{proposition}

\begin{proof}
Let $\wdh{f}\in \wdh{A}$ and  $\wdh{g}\in \wdh{A}\backslash\sqrt{Q\wdh{A}}$ satisfy $\wdh{f}\wdh{g}\in Q\wdh{A}$. By Theorem \ref{Pop_app}, for any integer $c\in\N$, there exist $\wdt{f}_c$, $\wdt{g}_c\in A$ such that $\wdt{f}_c\wdt{g}_c\in Q$ and $\wdt{f}_c-\wdh{f}$, $\wdt{g}_c-\wdh{g}\in \m_A^c$. For all $c$ large enough, $\wdt{g}_c\notin \sqrt{Q}$. Since $Q$ is a primary ideal, this proves that $\wdt{f}_c\in Q$ for $c$ large enough, hence $\wdh{f}\in Q\wdh{A}$.

 \end{proof}

\begin{corollary}\label{primary2}
Let $A$ be an excellent Henselian local ring. Let $I$ be an ideal of $A$ and let $I=Q_1\cap \cdots\cap  Q_s$ be a primary decomposition of $I$ in $A$. Then 
$$Q_1\wdh{A}\cap \cdots \cap Q_s\wdh{A}$$ is a primary decomposition of $I\wdh{A}$.
\end{corollary}

\begin{proof}
Since $\displaystyle I=\cap_{i=1}^sQ_i$, $\displaystyle I\wdh{A}=\cap _{i=1}^s(Q_i\wdh{A})$ by faithful flatness (or by Theorem \ref{Pop_app} for linear equations). We conclude with the help of Proposition \ref{primary}.
 \end{proof}
 
 \begin{corollary}\label{ufd_fact}
 Let $A$ be an excellent Henselian unique factorization local domain and let $f\in A$. If $\wdh g\in \wdh A$ divides $f$ in $\wdh A$ then there exists a unit $\wdh u\in\wdh A$ such that $\wdh u\wdh g\in A$.
 \end{corollary}
 
 \begin{proof}
 By Proposition \ref{UFD} the ring $\wdh A$ is a unique factorization domain. Since $\wdh g$ is a product of irreducible divisors of $f$ in $\wdh A$ it is enough to prove the corollary when $\wdh g$ is an irreducible divisor of $f$ in $\wdh A$. Moreover the ideal $\sqrt{fA}$ is a radical ideal of $A$, thus it generates a radical ideal of $\wdh A$ by Proposition \ref{int_domain}. So we can replace $f$ by a generator of $\sqrt{fA}$ and assume that $f$ is square free.\\ 
So the ideal generated by $\wdh g$ is a primary component of $f\wdh A$ and by Corollary \ref{primary2} it is generated by a primary component of $fA$. This means that $\wdh g\wdh A=P\wdh A$ where $P$ is a primary component of $fA$ (in fact a prime ideal associated to $fA$ since $f$ is square free). But the height of $P$ is one and $P$ is prime, so $P$ is a principal ideal since $A$ is a unique factorization domain. Let $g\in A$ be a generator of $P$. Then $\wdh g \wdh A=g\wdh A$ so there exists a unit $\wdh u\in\wdh A$ such that $\wdh u\wdh g=g\in A$.
 \end{proof}

The following result is a generalization of this corollary to integrally closed domains:

\begin{corollary}\label{Iz92}\cite{Iz92}
Let $A$ be an excellent Henselian  integrally closed local domain. If $\wdh{f}\in \wdh{A}$ and if there exists $\wdh{g}\in \wdh{A}$ such that $\wdh{f}\wdh{g}\in A\backslash\{0\}$, then there exists a unit $\wdh{u}\in \wdh{A}$ such that $\wdh{u}\wdh{f}\in A$. 

\end{corollary}
\begin{proof}
Let $(\wdh{f}\wdh{g})A=Q_1\cap  \cdots\cap  Q_s$ be a primary decomposition of the principal ideal of $A$ generated by $\wdh{f}\wdh{g}$. Since $A$ is an integrally closed domain, it is a Krull ring and $Q_i=\p_i^{(n_i)}$ for some prime ideal $\p_i$, $1\leq i\leq s$, where $\p^{(n)}$ denote the $n$-th symbolic power of $\p$ (see \cite[p.88]{Ma}). In fact $n_i:=\nu_{\p_i}(\wdh{f}\wdh{g})$ where $\nu_{\p_i}$ is the $\p_i$-adic valuation of the valuation ring $A_{\p_i}$. By Corollary \ref{primary2}, $\p_1^{(n_1)}\wdh{A}\cap \cdots\cap \p_s^{(n_s)}\wdh{A}$ is a primary decomposition of $(\wdh{f}\wdh{g})\wdh{A}$. Since $\nu_{\p_i}$ are valuations we have 
$$\wdh{f}\wdh{A}=\p_1^{(k_1)}\wdh{A}\cap \cdots\cap \p_s^{(k_s)}\wdh{A}=\left(\p_1^{(k_1)}\cap \cdots\cap \p_s^{(k_s)}\right)\wdh{A}$$
for some non negative integers $k_1,\ldots, k_s$.
Let $h_1,\ldots,h_r\in A$ be generators of the ideal $\p_1^{(k_1)}\cap \cdots\cap \p_s^{(k_s)}$. Then $\displaystyle\wdh{f}=\sum_{i=1}^r\wdh{a}_ih_i$ and $h_i=\wdh{b}_i\wdh{f}$ for some $\wdh{a}_i$, 
$\wdh{b}_i\in A$, $1\leq i\leq r$. Thus $\displaystyle\sum_{i=1}^r\wdh{a}_i\wdh{b}_i=1$, since $\wdh{A}$ is an integral domain. Thus one of the $\wdh{b}_i$  is invertible and we choose $\wdh{u}$ to be this invertible $\wdh{b}_i$.
 \end{proof}

\begin{corollary}\cite{To}
Let $A$ be an excellent Henselian local domain. For  $f(y)\in A[y]^r$  let $I$ be the ideal of $A[y]$ generated by $f_1(y),\ldots,f_r(y)$. Let us assume that ht$(I)=m$.
Let $\wdh{y}\in\wdh{A}^m$ satisfy $f(\wdh{y})=0$. Then $\wdh{y}\in A^m$.
\end{corollary}

\begin{proof}
Set $\p:=(y_1-\wdh{y}_1,\ldots, y_m-\wdh{y}_m)$. It is a prime ideal of $\wdh{A}[y]$ and ht$(\p)=m$. Of course $I\wdh{A}\subset\p$ and ht$(I\wdh{A})=m$ by Corollary \ref{primary2}. Thus $\p$ is of the form $\p'\wdh{A}$ where $\p'$ is minimal prime of $I$. Then $\wdh{y}\in \wdh{A}^m$ is the only common zero of all  the elements of $\p'$. By Theorem \ref{Pop_app}, $\wdh{y}$ can be approximated by a common zero of all the elements of $\p'$ which is in $A^m$. By uniqueness of such a common zero we have that $\wdh{y}\in A^m$.
 \end{proof}

  
 \section {Approximation with constraints}\label{constraint}
 We will now discuss the problem of the Artin Approximation with constraints. By constraints we mean that some of the components of the power series solutions do not depend on all the variables $x_i$ but only on some of them (as in Examples \ref{IMP} or \ref{biholo} of the introduction). In fact there are two different problems: one is the existence of convergent or algebraic solutions with constraints under the assumption that there exist formal solutions with the same constraints - the second one is the existence of formal solutions with constraints under the assumption that there exist approximate solutions with the same constraints. We can describe more precisely these two problems as follows:\\ 
 \\
 \textbf{Problem 1} (Artin Approximation with constraints):\\
  \emph{Let $A$ be an excellent Henselian local subring of $\k\llbracket x_1,\ldots,x_n\rrbracket$ and  $f(y)\in A[y]^r$. Let us assume that we have a formal solution $\wdh{y}\in \wdh{A}^m$ of $f=0$ and assume moreover that 
  $$\wdh{y}_i(x)\in\wdh{A}\cap \k\llbracket x_j, j\in J_i\rrbracket$$ for some subset $J_i\subset \{1,\ldots, n\}$, $1\leq i\leq m$.\\
   Is it possible to approximate $\wdh{y}(x)$ by a solution $\wdt{y}(x)\in A^m$  of $f=0$ such that $$\wdt{y}_i(x)\in A\cap  \k\llbracket x_j, j\in J_i\rrbracket,\  1\leq i\leq m?$$}\\
 \\
 The second problem is the following one:\\
 \\
\textbf{Problem 2} (Strong Artin Approximation with constraints): \\
\emph{Let us consider $f(y)\in\k\llbracket x\rrbracket[y]^r$ and $J_i\subset \{1,\ldots,n\}$, $1\leq i\leq m$. Is there a function $\b:\N\lgw \N$ such that:\\
 for all $c\in\N$ and all $\ovl{y}_i(x)\in\k\llbracket x_j,j\in J_i\rrbracket $, $1\leq i\leq m$, such that 
 $$f(\ovl{y}(x))\in (x)^{\b(c)},$$ there exist $\wdt{y}_i(x)\in\k\llbracket x_j, j\in J_i \rrbracket$ such that $f(\wdt{y}(x))=0$ and $\wdt{y}_i(x)-\ovl{y}_i(x)\in (x)^c$, $1\leq i\leq m$?}\\
 \\
 If such a function $\b$ exists, the smallest function satisfying this property is called the \emph{Artin function} of the system $f=0$.\\
 \\
 Let us remark that we have already given a positive answer to a similar weaker problem (see Proposition \ref{prop_mod}).\\
In general there are counterexamples to both problems stated in such a generality. But for some particular cases these two problems have a positive answer. We present here some positive and negative known results concerning these problems.  We will see that some systems yield a positive answer to one problem but a negative answer to the other one.

 
 \subsection{Examples}\label{examples}
 First of all we give here a list of examples that show that there is no hope, in general, to have a positive answer to Problem 1 without any more specific hypothesis, even if $A$ is the ring of algebraic or convergent power series. These examples are constructed by looking at the  Artin Approximation Problem for equations involving differentials  (Examples \ref{eq_dif_1} and \ref{eq_dif_2}) and operators on germs of functions (Examples \ref{K-G} and \ref{Becker}). To construct these examples the following lemma will be used repeatedly (most of the time when $A$ and $B$ are complete, $B=A\lb y\rb$ and $I=(0)$):
 
 \begin{lemma}\label{Taylor}\cite{Be}
 Let $(A,\m_A)$ be a Noetherian local ring and let $B$ be a Noetherian  local subring of $A\lb y\rb$  such that $\wdh{B}=\wdh{A}\lb y\rb$. Let $I$ be an ideal of $A$ such that $IB\cap A=I$. For any $P(y)\in B$ and  $\wdt{y}\in A^m$ whose components are in $\m_A$,  $P(\wdt{y})\in I$ if and only if there exists ${h}(y)\in B^m$ such that
 $$P(y)+\sum_{i=1}^m(y_i-\wdt{y}_i){h}_i(y)\in IB.$$
 \end{lemma}
 
 \begin{proof}[Proof of Lemma \ref{Taylor}]
 By Taylor expansion, we have:
 $$P(y)-P(\wdt{y})=\sum_{\a\in\N^m\backslash\{0\}}\frac{1}{\a_1!\ldots \a_m!}(y_1-\wdt{y}_1)^{\a_1}.\ldots(y_m-\wdt{y}_m)^{\a_m}\frac{\partial^{\a}P(\wdt{y})}{\partial y^{\a}}.$$
So  if $P(\wdt{y})\in IB$  there exists ${h}(y)\in \wdh A\lb y\rb^m$ such that
 $$P(y)+\sum_{i=1}^m(y_i-\wdt{y}_i){h}_i(y)\in IB.$$
 Since $B\lgw \wdh{B}=\wdh{A}\lb x\rb$ is faithfully  flat  we may assume that ${h}(y)\in B^m$ (See Example \ref{faithfully_flat}). This proves the necessary condition.\\
 On the other hand if $P(y)+\sum_{i=1}^m(y_i-\wdt{y}_i){h}_i(y)\in I B$, by substitution of $y_i$ by $\wdt{y}_i$, we get $P(\wdt{y})\in IB\cap A=I$.  \end{proof}
 
 \begin{example}\label{diff}
 Let us consider  $P(x,y,z)\in\k\lb x,y,z\rb$ where $\k$ is a field and $x$, $y$ and $z$ are single variables and $\wdh{y}\in(x)\k\lb x\rb$. Then $P(x,\wdh{y},\frac{\partial \wdh{y}}{\partial x})=0$ if and only if there is a formal power series $\wdh z$ such that $P(x,\wdh{y},\wdh{z})=0$ and $\wdh{z}-\frac{\partial \wdh{y}}{\partial x}=0$.\\
 We can remark that $\frac{\partial \wdh{y}}{\partial x}(x)$ is the coefficient of $t$ in the Taylor expansion of 
 $$\wdh{y}(x+t)-\wdh{y}(x).$$
So the equation  $\wdh{z}-\frac{\partial \wdh{y}}{\partial x}=0$ is equivalent to the existence of a formal power series $\wdh h(x,t)$ such that
  $$\wdh{y}(x+t)-\wdh{y}(x)=t\wdh z(x)+t^2\wdh h(x,t).$$
 By Lemma \ref{Taylor} this is equivalent to the existence of a formal power series
$ \wdh{k}(x,t,u)\in\k\lb x,t,u\rb$ such that   

$$\wdh{y}(u)-\wdh{y}(x)-t\wdh{z}(x)-t^2\wdh{h}(x,t)+(u-x-t)\wdh{k}(x,t,u)=0.$$
We remark that we may even assume that $\wdh h$ depends on the three variables $x$, $t$ and $u$: even in this case the previous equation implies that  $\wdh{z}=\frac{\partial \wdh{y}}{\partial x}$.\\
Now we introduce a new power series $\wdh g(u)$ such that  $\wdh g(u)=\wdh y(u)$. This equality is equivalent to the existence of  a formal power series $\wdh l(x,u)$ such that
$$\wdh g(u)-\wdh y(x)-(u-x)\wdh l(x,u)=0.$$
Once more we can even assume that $\wdh l$ depends on $x$, $u$ and $t$.\\
 Finally we see that
  \begin{equation*}\begin{split}P\left(x,\wdh{y}(x),\frac{\partial \wdh{y}}{\partial x}(x)\right)=0\Longleftrightarrow &\\
  \exists \wdh{z}(x)\in\k\lb x\rb, \ \wdh{h}(x,t,u),&\ \wdh{k}(x,t,u),\ \wdh{l}(x,t,u)\in\k\lb x,t,u\rb,\  \wdh{g}(u)\in\k\lb u\rb \text{ s.t.}\end{split}\end{equation*}
  $$\left\{\begin{aligned}  & P(x,\wdh{y}(x),\wdh{z}(x))=0\\
&  \wdh{g}(u)-\wdh{y}(x)-t\wdh{z}(x)-t^2\wdh{h}(x,t,u)+(u-x-t)\wdh{k}(x,t,u)=0\\
&  \wdh{g}(u)-\wdh{y}(x)+(u-x)\wdh{l}(x,t,u)=0\\ \end{aligned}\right.$$\\
\\
  
    \end{example}
    
  Assuming that the characteristic of the base field $\k$ is zero, Lemma \ref{Taylor} and Example \ref{diff} enable us to transform any system of equations involving partial differentials and compositions of power series into a system of algebraic equations whose solutions depend only on some of the $x_i$. Indeed we can also perform the same trick as in Example \ref{diff}  to handle higher order derivatives of $g$ since $\frac{1}{n!}\frac{\partial^n g}{\partial x^n}$ is the coefficient of $t^n$ in the Taylor expansion of 
  $$\wdh{y}(x+t)-\wdh{y}(x).$$
  Thus every equation of the form
  $$P(x,y,y',\ldots,y^{(n)})=0$$
  has a formal power series solution $\wdh y(x)$ if and only if there exist formal power series 
  $$ \wdh{z}_1(x),\ldots, \wdh z_n(x)\in\k\lb x\rb, \ \wdh{h}(x,t,u),\ \wdh{k}(x,t,u),\ \wdh{l}(x,t,u)\in\k\lb x,t,u\rb,\  \wdh{g}(u)\in\k\lb u\rb$$
  such that 
  $$\left\{\begin{aligned}  & P(x,\wdh{y}(x),\wdh{z}_1(x),\ldots,\wdh z_n(x))=0\\
& \wdh{g}(u)-\wdh{y}(x)-t\wdh{z}_1(x)-\cdots-t^n\frac{\wdh z_n(x)}{n!}-t^{n+1}\wdh{h}(x,t,u)+(u-x-t)\wdh{k}(x,t,u)=0\\
&  \wdh{g}(u)-\wdh{y}(x)+(u-x)\wdh{l}(x,t,u)=0\\ \end{aligned}\right.$$\\
We can also apply the same trick for differential equations involving power series in several variables (see for instance Example \ref{lastone}).\\
Let us remark that these transformations preserve the approximate solutions, i.e. every approximate solution up to degree $c$ of a system of equations involving partial differentials and compositions of power series provides an approximate solution up to degree $c$ of this system of algebraic equations whose solutions depend only on some of the $x_i$.\\
\\

  Of course there exist plenty of examples of such systems of equations with algebraic or analytic coefficients that do not have algebraic or analytic solutions but only formal solutions. These kinds of examples will enable us to construct counterexamples to Problem 1 as follows:

 \begin{example}\label{eq_dif_1}
 Let us consider the following differential equation: $y'=y$. The solutions of this equation are the convergent but not algebraic power series $ce^x\in\C\{x\}$ where $c$ is a complex number.\\
 On the other hand, by Example \ref{diff},  $\wdh{y}(x)$ is a convergent power series solution  of this equation if and only if there exist $\wdh{y}(x)$, $\wdh z(x)\in\C\{x\}$, $\wdh{g}(u)\in\C\{u\}$ and $\wdh{h}(x,t,u)$, $\wdh{k}(x,t,u)$, $\wdh{l}(x,t,u)\in\C\{ x,t,u\}$ such that:
 $$\left\{\begin{aligned}
 &\wdh z(x) -\wdh y(x) =0\\
& \wdh g(u)-\wdh y(x)-t\wdh z(x)-t^2\wdh h(x,t,u)+(u-x-t)\wdh k(x,t,u)  =0\\
 &\wdh g(u)-\wdh y(x)+(u-x)\wdh l(x,t,u)  =0
 \end{aligned}\right.$$
 Thus the former system of equations has a nonzero convergent solution $$(\wdh{y},\wdh z, \wdh g ,\wdh{h},\wdh{k},\wdh{l})\in\C\{x\}^2\times\C\{u\}\times\C\{x,t,u\}^3,$$ but no nonzero algebraic solution in $\C\langle x\rangle^2\times\C\langle u\rangle\times\C\langle x,t,u\rangle^3$.
  \end{example}

 \begin{example}[Kashiwara-Gabber  Example]\label{K-G}\cite[p. 75]{Hir}\index{Kashiwara-Gabber  Example}
 Let us perform the division of $xy$ by $$g:=(x-y^2)(y-x^2)=xy-x^3-y^3+x^2y^2$$ as formal power series in $\C\{ x,y\}  $ with respect to the monomial $xy$ (see Example \ref{IMP} in the introduction). The remainder of this division can be written as a sum $r(x)+s(y)$ where $r(x)\in(x)\C\{ x\}  $ and $s(y)\in(y)\C\{y\}$ since this remainder has no monomial divisible by $xy$. By symmetry, we get $r(x)=s(x)$, and by substituting $y$ by $x^2$ we get the Mahler equation:
 
 $$r(x^2)+r(x)-x^3=0.$$
 This relation yields the expansion
 $$r(x)=\sum_{i=0}^{\infty}(-1)^ix^{3.2^i}$$
 and shows that the remainder of the division is not algebraic since the gaps in the expansion of an algebraic power series over a characteristic zero field are bounded. This proves that the equation
 $$xy-gQ(x,y)-R(x)-S(y)=0$$
 has a convergent solution $(\wdh{q}(x,y),\wdh{r}(x),\wdh{s}(y))\in\C\{x,y\} \times\C\{x\} \times\C\{y\} $
  but has no algebraic solution $(q(x,y),r(x),s(y))\in\C\langle x,y\rangle\times\C\langle x\rangle\times\C\langle y\rangle$.

 \end{example}

 \begin{example}[Becker  Example]\label{Becker}\cite{Be77}\index{Becker  Example}
By a direct computation we can show that there exists a unique power series $f(x)\in\C\llbracket x\rrbracket  $ such that $f(x+x^2)=2f(x)-x$ and that this power series is not convergent. But, by Lemma \ref{Taylor}, we have:
$$f(x+x^2)-2f(x)+x=0$$
\begin{equation*}\begin{split}\displaystyle\Longleftrightarrow \exists g(y)\in\C\llbracket y\rrbracket  ,h(x,y), k&(x,y)\in\C\llbracket x,y\rrbracket   \text{ s.t. }\\ 
&\left\{\begin{array}{c}F_1:=g(y)-2f(x)+x+(y-x-x^2)h(x,y)=0\\ F_2:=g(y)-f(x)+(x-y)k(x,y)=0\end{array}\right.\end{split}\end{equation*} \\
 Then this system of equations  has solutions in $\C\llbracket x\rrbracket  \times\C\llbracket y\rrbracket  \times\C\llbracket x,y\rrbracket  ^2$ but no solution in $\C\langle x\rangle\times\C\langle y\rangle\times\C\langle x,y\rangle^2$, even  no solution in $\C\{x\}\times\C\{y\}\times\C\{x,y\}^2$.
 
 \end{example}

 \begin{example}\label{eq_dif_2}
 Set $\displaystyle \wdh{y}(x):=\sum_{n\geq 0}n!x^{n+1}\in\C\lb x\rb$. This power series is divergent and we have shown in Example \ref{euler}  that it is the only solution of the equation
 $$x^2y'-y+x=0\text{ (Euler Equation)}.$$
 By Example \ref{diff}, $\wdh{y}(x)$ is a solution of this differential equation if and only if there exist $\wdh{y}(x)$, $\wdh z(x)\in\C\{x\}$, $\wdh{g}(u)\in\C\{u\}$ and $\wdh{h}(x,t,u)$, $\wdh{k}(x,t,u)$, $\wdh{l}(x,t,u)\in\C\{ x,t,u\}$ such that:
  $$\left\{\begin{aligned}
& \wdh x^2z(x) -\wdh y(x)+x =0\\
 &\wdh g(u)-\wdh y(x)-t\wdh z(x)-t^2\wdh h(x,t,u)+(u-x-t)\wdh k(x,t,u)  =0\\
 &\wdh g(u)-\wdh y(x)+(u-x)\wdh l(x,t,u)  =0
 \end{aligned}\right.$$ with $\wdh{y}_1(x_1):=\wdh{y}(x_1)$. Thus this  system has no solution in 
 $$\C\{x\}^2\times\C\{u\}\times\C\{x,t,u\}^3$$ but it has solutions in $$\C\lb x\rb^2\times\C\lb u\rb\times\C\lb x,t,u\rb^3.$$
 \end{example}

 \begin{remark}
 By replacing $f_1(y),\ldots,f_r(y)$ by 
 $$g(y):= f_1(y)^2+x_1(f_2(y)^2+x_1(f_3(y)^2+\cdots)^2)^2$$ in these examples as in the proof of Lemma \ref{trick_B}, we can construct the same kind of examples involving only one equation. Indeed $f_1=f_2=\cdots=f_r=0$ if and only if $g=0$.
 \end{remark}


 
 \subsection{Nested Approximation in the algebraic case}
All the examples of Section \ref{examples} involve components that depend on separate variables. Indeed Example \ref{diff} shows that in general equations involving partial derivatives yield algebraic equations whose solutions have components with separate variables.\\
In the case the variables are nested (i.e. $y_i=y_i(x_1,\ldots,x_{s(i)})$ for some integer $i$, which is equivalent to saying that $J_i$ contains or is contained in $J_j$ for any $i$ and $j$ with the notation of Problems 1 and 2), it is not possible to construct a counterexample based on  differential equations as we did in Section \ref{examples}. In fact in this nested case, for polynomial equations, algebraic power series solutions are dense in the set of formal power series solution. Moreover we will see, still in the nested case, that this is no longer true in the analytic case.\\
  First of all in the algebraic case we have the following \index{nested  approximation theorem} result:

 \begin{theorem}[Nested Approximation Theorem]\label{nested_alg}\cite{KPPRM,Po86}
 
 Let $(A,\m_A)$ be an excellent Henselian local ring and  $f(x,y)\in A\langle x,y\rangle^r$. Let $\wdh{y}(x)$ be a solution of $f=0$ in $(\m_A+(x))\wdh{A}\llbracket x\rrbracket ^m $. Let us assume that $\wdh{y}_i\in\wdh{A}\llbracket x_1,\ldots,x_{s_i}\rrbracket  $, $1\leq i\leq m$, for integers $s_i$, $1\leq s_i\leq n$ (we say that $\wdh y(x)$ satisfies a nested condition).\\ 
 Then for any $c\in\N$  there exists a solution $\wdt{y}(x)\in A\langle x\rangle^m$ such that for all $i$, $\wdt{y}_i(x)\in A\langle x_1,\ldots,x_{s_i}\rangle$ and $\wdt{y}(x)-\wdh{y}(x)\in (\m_A+(x))^c$.
 
 \end{theorem}
 This result has a lot of applications and is one of the most important results about Artin Approximation. Its proof is based on General NŽron desingularization, more precisely it uses the fact that the rings $A\lb x_1,\ldots, x_i\rb\langle x_{i+1},\ldots, x_n\rangle$ satisfy Theorem \ref{Pop_app}  (see Remark \ref{ex_nested}). The proof we present here is based on a characterization of the ring of algebraic power series over a ring of formal power series (see Lemma \ref{lemma_nested}) and 
 is   different from the classical one even if it is based on the key  fact that  the rings $A\lb x_1,\ldots, x_i\rb\langle x_{i+1},\ldots, x_n\rangle$ satisfy Theorem \ref{Pop_app} for any excellent Henselian local ring $B$. \\
 We note that before the work of D. Popescu \cite{Po86} this result was already known in the case where $A=\k$ is a field and the integers $s_i$ equal $1$ or $n$ (see  \cite[Theorem 4.1]{BDLvdD} where the result is attributed to M. Artin) and whose proof is based on the fact that the ring $\k\lb x_1\rb\langle x_2,\ldots,x_n\rangle$ satisfies the Artin Approximation Theorem (see also the comment following Theorem \ref{W}).

 \begin{proof}[Proof of Theorem \ref{nested_alg}]
 For simplicity we assume that $A$ is a complete local domain (this covers already the important case where $A$ is a field).
 We first give the following lemma that may be of independent interest and whose proof is given below:
 
 \begin{lemma}\label{lemma_nested}
 Let $A$ be a complete normal local domain, $u:=(u_1,\ldots,u_n)$, $v:=(v_1,\ldots,v_m)$. Then
 \begin{equation*}\begin{split}A\lb u\rb\langle v\rangle=
 \{f\in A\lb u,v\rb\ /\ &\exists s\in \N, g\in A\langle v,z_1,\ldots,z_s\rangle,\\ &\wdh{z}_i\in (\m_A+(u))A\lb u\rb,\ f=g(v,\wdh{z}_1,\ldots,\wdh{z}_s)\}.\end{split}\end{equation*}
 \end{lemma}

Using this lemma we can prove Theorem \ref{nested_alg} by induction on $n$. First of all, since $A=\frac{B}{I}$ where $B$ is a complete regular local ring (by Cohen  Structure Theorem), by using the same trick as in the proof of Corollary \ref{Ar68cor} we may replace $A$ by $B$ and assume that $A$ is a complete regular local ring.\\
 Let us assume that Theorem \ref{nested_alg} is true for $n-1$. We set $ x' :=(x_1,\ldots,x_{n-1})$. Then we denote by $y_1,\ldots,y_k$ the unknowns depending only on $ x' $ and by $y_{k+1},\ldots,y_m$ the unknowns depending on $x_n$. Let us consider the following system of equations
 \begin{equation}\label{eq}f( x' ,x_n,y_1( x' ),\ldots,y_k(x'),y_{k+1}( x' ,x_n), y_m(x',x_n))=0.\end{equation}
By Theorem \ref{Pop_app} and Remark \ref{ex_nested}  we may assume that $\wdh{y}_{k+1},\ldots,\wdh{y}_{m}\in\k\lb x'\rb \langle x\rangle $.
Thus by Lemma \ref{lemma_nested} we can  write $\wdh{y}_i=\sum_{j\in\N} h_{i,j}(\wdh{z})x_n^{j}$ with $\sum_{j\in\N} h_{i,j}(z)x_n^{j}\in\k\langle z,x_n\rangle$ and $\wdh{z}=(\wdh{z}_1,\ldots,\wdh{z}_s)\in(x')\k\lb x'\rb^s$. We can write
 \begin{equation*}\begin{split} f\left(x',x_n,y_1,\ldots,y_k, \sum_{j}h_{k+1,j}(z)x_n^{j},\ldots , \sum_{j}h_{m,j}(z)x_n^{j} \right)&=\\
= \sum_{j}G_{j}(x',y_1,\ldots,&y_k,z)x_n^{j}\end{split}\end{equation*}
 where $G_{j}(x',y_1,\ldots,y_k,z)\in\k\langle x',y_1,\ldots,y_k,z\rangle$ for all $j\in\N$. Thus 
 $$(\wdh{y}_1,\ldots,\wdh{y}_k,\wdh{z}_1,\ldots,\wdh{z}_s)\in\k\lb x'\rb^{k+s}$$
  is a solution of the equations $G_{j}=0$ for all $j\in\N$. Since $\k\langle t,y_1,\ldots,y_k,z\rangle$ is Noetherian, this system of equations is equivalent to a finite system 
 $G_{j}=0$ with $j\in E$ where $E$ is a finite subset of $\N$. Thus by the induction hypothesis applied to the system $G_{j}(x',y_1,\ldots,y_k,z)=0$, $j\in E$, there exist $\wdt{y}_1,\ldots,\wdt{y}_k$, $\wdt{z}_1,\ldots,\wdt{z}_s\in\k\langle x'\rangle$, with nested conditions,  such that $\wdt{y}_i-\wdh{y}_i$, $\wdt{z}_l-\wdh{z}_l\in (x')^c$, for $1\leq i\leq k$ and $1\leq l\leq s$, and $G_{j}(x', \wdt{y}_1,\ldots,\wdt{y}_k,\wdt{z})=0$ for all $j\in E$. Hence $G_{j}(x',\wdt{y}_1,\ldots,\wdt{y}_k,\wdt{z})=0$ for all $j\in\N$.\\
 Set $\wdt{y}_i=\sum_{j\in\N}h_{i,j}(\wdt{z})x_n^{j}$ for $k<j\leq m$. Then $\wdt{y}_1,\ldots,\wdt{y}_{m}$ satisfy the conclusion of the theorem.
  
 \end{proof}
 
  \begin{proof}[Proof of Lemma \ref{lemma_nested}]
 Let us define
 \begin{equation*}\begin{split}B:=\{f\in A\lb u,v\rb\ /\ \exists s\in \N, g\in A&\langle v,z_1,\ldots,z_s\rangle,\\
 \wdh{z}_i\in &(\m_A+(u))A\lb u\rb,\ f=g(v,\wdh{z}_1,\ldots,\wdh{z}_s)\}.\end{split}\end{equation*} 
Clearly $B$ is a subring of $A\lb u\rb\langle v\rangle$.\\

Since $A\lb u\rb\langle v\rangle$ is the Henselization of $A\lb u\rb[ v]_{\m_A+(u,v)}$, by Theorem \ref{iversen} there exist $h\in (\m_A+(u,v))A\lb u\rb\langle v\rangle$ and a monic polynomial $P\in A\lb u\rb[v][T]$ in $T$ such that
$$P(h)=0,\ \ \frac{\partial P}{\partial T}(h)\notin \m_A+(u,v)$$
and 
$$f\in A\lb u\rb[v,h]_{(\m_A+(u,v))\cap A\lb u\rb[v,h]}.$$
So there are two polynomials $Q$, $R\in A\lb u\rb[v,h]$, $R\notin (\m_A+(u,v))\cap A\lb u\rb[v,h]$, with
$$f=QR^{-1}.$$
We have 
$$R=\sum_{\g ,j}(r_{\g ,j}+\wdh z_{\g ,j})v^\g  h^j$$
where $r_{\g ,j}\in A$ and $\wdh z_{\g ,j}\in (\m_A+(u))A\lb u\rb$ for every $\g $, $j$. Let us set
$$R':=\sum_{\g ,j}(r_{\g ,j}+z_{\g ,j})v^\g  H^j$$
where the $z_{\g ,j}$ and $H$ denote new variables. The coefficient $r_{0,0}$ is a unit in $A$ since $R\notin \m_A+(u,v)$,  $\wdh z_{\g ,j}\in (\m_A+(u))A\lb u\rb$ for every $\g $ and $j$, and $h\in  (\m_A+(u,v))A\lb u\rb\langle v\rangle$. Then $R'$ is a unit in $A\langle z_{\g ,j},v,H\rangle$ and its inverse is in $A\langle z_{\g ,j},v,H\rangle$. Moreover ${R'}^{-1}(\wdh z_{\g ,j},v,h)=R^{-1}$ by uniqueness  of the unit. Since the composition of algebraic power series over $A$ is an algebraic power series over $A$, this shows that if the lemma is proven for $h$ then it is proven for $f$ by adding the coefficients of $Q$ from $A\lb u\rb$ and  the $\wdh z_{\g ,j}$ as new $\wdh z$.

So we may replace $f$ by $h$ and assume that $f\in (\m_A+(u,v))A\lb u\rb\langle v\rangle$,
$$P(f)=0 \text{ and } \ \frac{\partial P}{\partial T}(f)\notin \m_A+(u,v).$$
Let us write $P=\sum_{\a,i}P_{\a,i}v^\a T^i$ for some $P_{\a,i}\in A\lb u\rb.$\\
Let us write $$P_{\a,i}=p_{\a,i}+\wdh z_{\a,i}$$
where $p_{\a,i}\in A $ and $\wdh z_{\a,i}\in(\m_A+(u))A\lb u\rb$.

 Let $\wdh z$ be the vector whose components are the $\wdh z_{\a,i}$ and set
$$F:=\sum_{\a,i}(p_{\a,i}+z_{\a,i})v^\a T^i$$
for some new variables $z_{\a,i}$. Therefore $F(v,T,\wdh z)=P$. We set $F'=\frac{\partial F}{\partial T}$. We have that
$$F(v,T,0)=F(v,T,\wdh z)=P(f) \text{ modulo } \m_A+(u,v,T),$$
$$F'(v, T,0)=F'(v,T,\wdh z)=\frac{\partial P}{\partial T}(f)\neq 0  \text{ modulo } \m_A+(u,v,T)$$
since the components of $\wdh z$ are in $\m_A+(u)$ and $f\in (\m_A+(u,v))A\lb u\rb\langle v\rangle$.
Thus
$$F(v,0,z)= 0 \text{ modulo } (\m_A+(v,z))A\langle v,z\rangle,$$
$$F'(v,0,z)\neq 0 \text{ modulo } (\m_A+(v,z))A\langle v,z\rangle.$$
Hence, by the Implicit Function Theorem \ref{IFT}, there exists a unique 
$$g\in(\m_A+(v,z))A\langle v,z\rangle$$ such that
$$F(v,g,z)=0.$$
So we have that $F(v,g(\wdh z),\wdh z)=0$. But $P=F(v,T,\wdh z)=0$ has a unique solution $T=f$ in $(\m_A+(u,v))A\lb u\rb\langle v\rangle$ by the Implicit Function Theorem. This proves that $f=g(\wdh z)\in B$.
 \end{proof}

\begin{remark}
There exists a more elementary proof of Theorem \ref{nested_alg} in the case where $f(x,y)$ is linear with respect to $y$, i.e. a proof that does not involve the use of General NŽron desingularization Theorem (see \cite{CJR}). It is based on the fact that, for linear equations with one nest, Theorem \ref{nested_alg} reduces to Theorem \ref{Ar69} (see  \cite[Theorem 12.6]{B-M1}).
\end{remark}
 
 There also exists a nested version of P\l oski's Theorem for algebraic power series (or equivalently a nested version of Corollary \ref{Ploski2}): see \cite[Theorem 11.4 ]{Sp} or  \cite[Theorem 2.1]{BPR}. This "nested P\l oski 's Theorem" is used in \cite{BPR} (see also \cite{Mos} where this idea has first been introduced, and \cite{BKPR,Ro17} for subsequent works based on this idea) to show that any complex or real analytic set germ (resp. analytic function germ) is homeomorphic to an algebraic set germ (resp. algebraic function germ). In fact it is used to construct a topologically trivial deformation of a given analytic set germ whose one of the fibers is a Nash set germ, i.e. an analytic set germ defined by algebraic power series.
 \\
  \\
 Moreover, using ultraproducts methods, we can  deduce the following Strong nested Approximation result:
 \begin{corollary}\label{appr_funct_nested}\cite{BDLvdD}
  Let $\k$ be a field and  $f(x,y)\in \k\langle x,y\rangle^r$. There exists a function $\b : \N\lgw \N$ satisfying the following property:\\
   Let  $c\in\N$ and $\ovl{y}(x)\in \left((x)\k\lb x\rb\right)^m$  satisfy $f(x,\ovl{y}(x))\in (x)^{\b(c)}$. Let us assume that $\ovl{y}_i(x)\in\k\llbracket x_1,\ldots,x_{s_i}\rrbracket  $, $1\leq i\leq m$, for integers $s_i$, $1\leq s_i\leq n$.\\ 
 Then there exists a solution $\wdt{y}(x)\in \left((x)\k\langle x\rangle\right)^m$ of $f=0$ with $\wdt{y}_i(x)\in \k\langle x_1,\ldots,x_{s_i}\rangle$ and $\wdt{y}(x)-\ovl{y}(x)\in (x)^c$.

 \end{corollary}


 \subsection{Nested Approximation in the analytic case}
 In the analytic case, Theorem \ref{nested_alg} is no longer valid, as shown by the following example:
 
 \begin{example}[Gabrielov  Example]\label{Gab_ex}\cite{Ga1}\index{Gabrielov Example}
 Let 
 $$\phi : \C\{x_1,x_2,x_3\}\lgw\C\{y_1,y_2\}$$
  be the morphism of analytic $\C$-algebras defined by 
  $$\phi(x_1)=y_1,\ \phi(x_2)=y_1y_2,\  \phi(x_3)=y_1y_2e^{y_2}.$$

Let $f\in\Ker(\wdh{\phi})$ be written as $f=\sum_{d=0}^{+\infty}f_d$ where $f_d$ is a homogeneous polynomial of degree $d$ for all $d\in\N$. Then $0=\wdh{\phi}(f)=\sum_dy_1^df_d(1,\,y_2,\,y_2e^{y_2})$. Thus $f_d=0$ for all $d\in\N$ since $y_2$ and $y_2e^{y_2}$ are algebraically independent over $\C$. Hence $\Ker(\wdh{\phi})=(0)$ and $\Ker(\phi)=(0)$. This example is due to W. S. Osgood \cite{Os}.\\
\\
\textbf{(1)} We can remark that "$\phi\left(x_3-x_2e^{\frac{x_2}{x_1}}\right)=0$". But $x_3-x_2e^{\frac{x_2}{x_1}}$  is not an element of $\C\{x_1,\,x_2,\,x_3\}$.\\
Let us set 
$$f_n:=\left(x_3-x_2\sum_{i=0}^{n}\frac{1}{i!}\frac{x_2^i}{x_1^i}\right)x_1^n\in\C[x_1,\,x_2,\,x_3],\ \forall n\in\N.$$
Then
$$\phi(f_n)=y_1^{n+1}y_2\sum_{i=n+1}^{+\infty}\frac{y_2^i}{i!},\ \forall n\in\N.$$
Then we see that $(n+1)!\phi(f_n)$ is a convergent power series whose coefficients have module less than 1. Moreover if the coefficient of $y_1^ky_2^l$ in the Taylor expansion of $\phi(f_n)$ is non-zero then $k=n+1$. This means that the supports of $\phi(f_n)$ and $\phi(f_m)$ are disjoint as soon as $n\neq m$. Thus the power series 
$$h:=\sum_n(n+1)!\phi(f_n)$$  is  convergent  since each of its coefficients has module less than 1. But $\wdh{\phi}$ being injective, the unique element whose image is $h$   is necessarily $\wdh{g}:=\sum_n(n+1)!f_n$. But
$$\wdh{g}=\sum_n(n+1)!f_n=\left(\sum_n(n+1)!x_1^n\right)x_3+\wdh{f}(x_1,\,x_2),$$
 $\sum_n(n+1)!x_1^n$  is a divergent power series  and $\wdh{\phi}(\wdh{g}(x))=h(y)\in\C\{y\}$.\\
This shows that 
$$\phi(\C\{x\})\subsetneq\wdh{\phi}(\C\lb x\rb)\cap \C\{y\}.$$
\\
\textbf{(2)}  By Lemma \ref{Taylor} $\wdh{\phi}(\wdh{g}(x))=h(y)$ is equivalent to saying that there exist $\wdh{k}_1(x,y)$, $\wdh{k}_2(x,y)$, $\wdh{k}_3(x,y)\in\C\lb x,y\rb$ such that
 \begin{equation}\label{eq_gab}\wdh{g}(x)+(x_1-y_1)\wdh{k}_1(x,y)+(x_2-y_1y_2)\wdh{k}_2(x,y)+(x_3-y_1y_2e^{y_2})\wdh{k}_3(x,y)-h(y)=0.\end{equation}
 Since $\wdh{g}(x)$ is the unique element whose image under $\wdh{\phi}$ equals $h(y)$,  Equation (\ref{eq_gab}) has no convergent solution $g(x)\in\C\{ x\}$, $k_1(x,y)$, $k_2(x,y)$, $k_3(x,y)\in\C\{ x,y\}$. Thus Theorem \ref{nested_alg} is not true in the analytic setting. \\
 \\
\textbf{(3)} We can modify a little bit the previous example as follows.
 Let us define $\wdh{g}_1(x_1,x_2):=\sum_n(n+1)!x_1^n$ and $\wdh{g}_2(x_1,x_2):=\wdh{f}(x_1,x_2)$. By replacing $y_1$ by $x_1$, $y_2$ by $y$ and $x_3$ by $x_1ye^y$ in Equation (\ref{eq_gab}) we see that the equation
 \begin{equation}\label{eq_gab'}\wdh{g}_1(x_1,x_2)x_1ye^{y}+\wdh{g}_2(x_1,x_2)+(x_2-x_1y)\wdh{k}(x,y)-h(x_1,y)=0.\end{equation}
 has a nested formal solution 
 $$(\wdh g_1,\wdh g_2,\wdh k)\in\C\lb x_1,x_2\rb^2\times \C\lb x_1,x_2,y\rb$$
 but no nested convergent solution in $\C\{ x_1,x_2\}^2\times \C\{ x_1,x_2,y\}$.\\
  \end{example}
 
  Nevertheless there are, at least, three positive results about the nested approximation problem in the analytic category. We present them here.
    
  
  \subsubsection{Grauert  Theorem} \index{Grauert Theorem}
 The first one is due to H. Grauert who proved it in order to construct analytic deformations of a complex analytic  germ in the case where it has an isolated singularity. The approximation result of H. Grauert   may be reformulated as: "if a system of complex analytic equations, considered as a formal nested system, admits an Artin function (as in Problem 2) which is the Identity function, then it has nested analytic solutions". We present here the result.\\
Set $x:=(x_1,\ldots,x_n)$, $t:=(t_1,\ldots,t_l)$, $y=(y_1,\ldots,y_m)$ and $z:=(z_1,\ldots,z_p)$.
Let $f:=(f_1,\ldots,f_r)$ be in $ \C\{t,x,y,z\}^r$. Let $I$ be an ideal of $\C\{t\}$.

\begin{theorem}\label{Grauert_thm}\cite{Gr72}
Let $d_0\in\N$ and  $(\ovl{y}(t),\ovl{z}(t,x))\in\C[t]^{m}\times\C\{x\}[t]^p$ satisfy 
$$f(t,x,\ovl{y}(t),\ovl{z}(t,x))\in I+(t)^{d_0}.$$
Let us assume that for any $d\geq d_0$ and for any $(y^{(d)}(t),z^{(d)}(t,x))\in \C[t]^{m}\times\C\{x\}[t]^p$ such that, $\ovl{y}(t)-y^{(d)}(t)\in (t)^{d_0}$ et $\ovl{z}(t,x)-z^{(d)}(t,x)\in (t)^{d_0}$, and such that
$$f\left(t,x,y^{(d)}(t),z^{(d)}(t,x)\right)\in I+(t)^{d},$$
there exists $(\e(t),\eta(t,x))\in \C[t]^{m}\times\C\{x\}[t]^p$ homogeneous in $t$ of degree $d$ such that
$$f(t,x,y^{(d)}(t)+\e(t),z^{(d)}(t,x)+\eta(t,x))\in I+(t)^{d+1}.$$
Then there exists $(\wdt{y}(t),\wdt{z}(t,x))\in\C\{t\}^m\times\C\{t,x\}^p$ such that
$$f(t,x,\wdt{y}(t),\wdt{z}(t,x))\in I\ \text{ 
and }\ 
\wdt{y}(t)-\ovl{y}(t),\ \wdt{z}(t,x)-\ovl{z}(t,x)\in (t)^{d_0}.$$
\end{theorem}

The main ingredient of the proof is a result of Functional Analysis called "voisinages privil\'egi\'es" and proven by H. Cartan \cite[Th\'eor\`eme $\a$]{Ca}. We do not give the details here but the reader may consult \cite{JP}.\\
Let us also mention that a completely similar result for differential equations, called Cartan-K\"ahler Theorem, has been proven by B. Malgrange and its proof is based on the same tools used in the proof of Theorem \ref{Grauert_thm} (see the appendix of \cite{Mal}).


  \subsubsection{Gabrielov  Theorem}
  The second positive result about the nested approximation problem in the analytic category is due to A. Gabrielov. Before giving his result, let us explain the context.\\
    Let $\phi :A\lgw B$ be a morphism of complex analytic algebras where $A:=\frac{\C\{x_1,\ldots,x_n\}}{I}$ and $B:=\frac{\C\{y_1,\ldots,y_m\}}{J}$ are analytic algebras. Let us denote by $\phi_i$ the image of $x_i$ by $\phi$ for $1\leq i\leq n$. Let us denote by $\wdh{\phi} : \wdh{A}\lgw \wdh{B}$ the morphism induced by $\phi$. A. Grothendieck \cite{Gro} and S. S. Abhyankar \cite{Ar71} raised the following question: Does $\Ker(\wdh{\phi})=\Ker(\phi).\wdh{A}$?\\
   Without loss of generality we may assume that $A$ and $B$ are regular, i.e. $A=\C\{x_1,\ldots,x_n\}$ and $B=\C\{y_1,\ldots,y_m\}$. \\
 In this case, an element of $\Ker(\phi)$ (resp. of $\Ker(\wdh{\phi})$) is called an \emph{analytic} (resp.  \index{formal relation}\emph{formal}) \emph{relation} \index{analytic relation}between $\phi_1(y),\ldots,\phi_m(y)$.  Hence the previous question is equivalent to the following: is any formal relation $\wdh{S}$ between $\phi_1(y),\ldots,\phi_n(y)$ a linear combination of analytic relations?\\
 This question is also equivalent to the following: may every formal relation between $\phi_1(y),\ldots,\phi_n(y)$ be approximated by analytic relations for the $(x)$-adic topology? In this form the problem is the "dual" problem to the Artin Approximation Problem.\\
  In fact this problem is also a nested approximation problem. Indeed let $\wdh{S}$ be a formal relation between $\phi_1(y),\ldots,\phi_n(y)$. This means that $\wdh{S}(\phi_1(y),\ldots,\phi_n(y))=0$. By Lemma \ref{Taylor}  this is equivalent to the existence of formal power series 
  $$\wdh{h}_1(x,y),\ldots,\wdh{h}_n(x,y)\in\C\llbracket x,y\rrbracket$$
   such that
  $$\wdh{S}(x_1,\ldots,x_n)-\sum_{i=1}^n(x_i-\phi_i(y))\wdh{h}_i(x,y)=0.$$
  Thus we see that the equation
  \begin{equation}\label{S_gab}S-\sum_{i=1}^n(x_i-\phi_i(y))H_i=0\end{equation}
  has a formal nested solution 
  $$(\wdh S(x),\wdh h_1(x,y),\ldots,\wdh h_n(x,y))\in\C\lb x\rb\times\C\lb x,y\rb^n.$$
  On the other hand if this equation has an analytic nested solution 
  $$(S(x),h_1(x,y),\ldots,h_n(x,y))\in\C\{x\}\times\C\{x,y\}^n,$$ this would provide an analytic relation between $\phi_1(y),\ldots,\phi_n(y)$:
  $$S(\phi_1(y),\ldots,\phi_n(y))=0.$$
  Example \ref{Gab_ex} yields a negative answer to this problem by modifying in the following way the example of Osgood (see Example \ref{Gab_ex}):
  
  \begin{example}\cite{Ga1}
 Let us consider now the morphism $$\psi\ :\ \C\{x_1,\,x_2,\,x_3,\,x_4\}\lgw \C\{y_1,\,y_2\}$$ defined by
$$\psi(x_1)=y_1,\ \ \psi(x_2)=y_1y_2,\ \ \psi(x_3)=y_1y_2e^{y_2},\ \ \psi(x_4)=h(y_1,\,y_2)$$
where $h$ is the convergent power series defined in Example \ref{Gab_ex}.\\
Let $\wdh g$ be the power series defined in Example \ref{Gab_ex}.
Then $x_4-\wdh{g}(x_1,\,x_2,\,x_3)\in\Ker(\wdh{\psi})$. On the other hand the morphism induced by $\wdh{\psi}$ on  $\C\llbracket x_1,\ldots,\,x_4\rrbracket/(x_4-\wdh{g}(x_1,\,x_2,\,x_3))$  is isomorphic to $\wdh{\phi}$ (where $\phi$ is the morphism of Example \ref{Gab_ex}) that is injective. Thus we have $\Ker(\wdh{\psi})=(x_4-\wdh{g}(x_1,\,x_2,\,x_3))$. \\
Since $\Ker(\psi)$ is a prime ideal of  $\C\{x\}$,  $\Ker(\psi)\C\llbracket x\rrbracket$ is a prime ideal of $\C\llbracket x\rrbracket$  included in $\Ker(\wdh{\psi})$ by Proposition \ref{int_domain}. Let us assume that $\Ker(\psi)\neq (0)$, then $\Ker(\psi)\C\llbracket x\rrbracket=\Ker(\wdh{\psi})$ since ht$(\Ker(\wdh{\psi}))=1$. Thus $\Ker(\wdh{\psi})$  is generated by one convergent power series  denoted by $f\in\C\{x_1,\ldots,\,x_4\}$ (in unique factorization domains, prime ideals of height one are principal ideals). Since $\Ker(\wdh{\psi})=(x_4-\wdh{g}(x_1,\,x_2,\,x_3))$, there exists $u(x)\in\C\llbracket x\rrbracket$, $u(0)\neq 0$, such that  $f=u(x).(x_4-\wdh{g}(x_1,\,x_2,\,x_3))$. By the uniqueness  of the decomposition given by the Weierstrass Preparation Theorem of $f$ with respect to $x_4$ we see that $u(x)$ and $x_4-\wdh{g}(x_1,\,x_2,\,x_3)$ must be convergent power series, which is impossible since  $\wdh{g}$ is a divergent power series. Hence $\Ker(\psi)=(0)$  but $\Ker(\wdh{\psi})\neq (0)$.
  \end{example}

Nevertheless    A. Gabrielov proved the following theorem: \index{Gabrielov Theorem}
  
  \begin{theorem}\label{Gab}\cite{Ga2}
  Let $\phi :A\lgw B$ be a morphism of complex analytic algebras.  Let us assume that the generic rank of the Jacobian matrix is equal to $\dim(\frac{A}{\Ker(\wdh{\phi})})$. Then $\Ker(\wdh{\phi})=\Ker(\phi).\wdh{A}$. In particular the equation \eqref{S_gab} satisfies the nested approximation property.
    \end{theorem}
    
    \begin{remark}
 A morphism satisfying the hypothesis of Theorem \ref{Gab} is called \emph{regular} (but this notion of regularity has nothing to do with Definition \ref{regular}). The morphisms of analytic spaces being regular on stalks have been widely studied. The reader may consult \cite{B-M,B-M1,B-M3,Paw} for a study of global properties of regular morphisms and the relation with the composite functions problem in the    $\mathcal{C}^{\infty}$ case.   
    \end{remark}

    \begin{remark}
  Let $\phi : \C\{x_1,\ldots, x_n\}\lgw \C\{y_1,\ldots ,y_m\}$ be a morphism of complex analytic algebras.  Assume for simplicity that $\phi$ and $\wdh\phi$ are injective, and let us denote by $\phi_i(y)$ the image of $x_i$ for every $i$. By a Theorem of P. Eakin and G. Harris \cite{EH} if the generic rank of the Jacobian matrix is strictly less than $n$ then there exists a divergent formal series $\wdh g\in\C\lb x\rb\backslash \C\{x\}$ such that
  $$\wdh\phi(\wdh g)=h\in\C\{y\}.$$
  Since $\wdh\phi$ is injective $h$ is not the image of convergent power series. So, exactly as in  Example \ref{Gab_ex} (2), there exist $\wdh k_i(x,y)\in\C\lb x,y\rb$ for $1\leq i\leq n$ such that
  \begin{equation}\label{pp} 
  \wdh g(x)+\sum_{i=1}^n(x_i-\phi_i(y))\wdh k_i(x,y)=h(y)\end{equation}
  but there is no $g\in\C\{x\}$, $k_i(x,y)\in\C\{x,y\}$, $1\leq i\leq n$, such that
  $$  g(x)+\sum_{i=1}^n(x_i-\phi_i(y)) k_i(x,y)=h(y).$$
    In particular if we define $\phi$ (with $n=3$ and $m=2$) by
    $$\phi(x_1)=y_1,\ \phi(x_2)=y_1y_2,\ \phi(x_3)=y_1\xi(y_2)$$
    where $\xi(y_2)$ is a transcendental power series, exactly as done for Osgood's Example  \ref{Gab_ex}, $\phi$ and $\wdh\phi$ are injective, but the generic rank of the Jacobian matrix is 2 since $m=2$ so the preceding result of P. Eakin and G. Harris applies. This gives a systematic way to construct examples of (linear) equations having nested formal solutions but no convergent nested equations.
   \end{remark}
  
  \begin{proof}[Sketch of the proof of Theorem \ref{Gab}]
  We give a sketch of the proof given by J.-Cl. Tougeron \cite{To2}.
  As before we may assume that $A=\C\{x_1,\ldots,x_n\}$ and $B=\C\{y_1,\ldots,y_m\}$. Let us assume that $\Ker(\phi).\wdh{A}\not\subset \Ker(\wdh{\phi})$ (which is equivalent to ht$(\Ker(\phi))\leq$ ht$( \Ker(\wdh{\phi}))$ since both ideals are prime). Using a Bertini type theorem we may assume that $n=3$, $\phi$ is injective and  $\dim(\frac{\C\llbracket x\rrbracket}{\Ker(\wdh{\phi})})=2$ (in particular $\Ker(\wdh{\phi})$ is a principal ideal). Moreover, in this case we may assume that $m=2$. After a linear change of coordinates we may assume that $\Ker(\wdh{\phi})$ is generated by an irreducible Weierstrass polynomial of degree $d$ in $x_3$. Using changes of coordinates and quadratic transforms on $\C\{y_1,y_2\}$ and using changes of coordinates of $\C\{x\}$ involving only $x_1$ and $x_2$, we may assume that $\phi_1=y_1$ and $\phi_2=y_1y_2$. Let us define $f(y):=\phi_3(y)$. Then we have
 $$f(y)^d+\wdh{a}_1(y_1,y_1y_2)f(y)^{d-1}+\cdots +\wdh{a}_d(y_1,y_1y_2)=0$$
 for some $\wdh{a}_i(x)\in\C\llbracket x_1,x_2\rrbracket$, $1\leq i\leq d$.
  Then we want to prove that the $\wdh{a}_i$  may be chosen convergent in order to get a contradiction. Let us denote
  $$P(Z):=Z^d+\wdh{a}_1(x_1,x_2)Z^{d-1}+\cdots+\wdh{a}_d(x_1,x_2)\in\C\llbracket x\rrbracket [Z].$$
  By assumption  $P(Z)$ is irreducible since $\Ker(\wdh{\phi})$ is prime.
  J.-Cl. Tougeron studies the algebraic closure $\ovl{\K}$ of the field $\C(\!( x_1,x_2)\!)$. 
  Let consider the following valuation ring
  $$V:=\left\{\frac{f}{g}\ /\ f,g\in \C\lb x_1,x_2\rb, g\neq 0, \ord(f)\geq \ord(g)\right\},$$
  let $\wdh{V}$ be its completion and $\wdh{\K}$ the fraction field of $\wdh{V}$.
  J.-Cl. Tougeron proves that the algebraic extension $\K\lgw\ovl{\K}$ factors into $\K\lgw \K_1\lgw \ovl{\K}$ where
  $\K_1$ is a subfield of the following field
    \begin{equation*}\begin{split} \LL:=\left\{A\in\wdh{\K}\ / \ \exists \d,\ a_i\in \k[x]  \text{ is homogeneous } \forall i,\right.&\\
      \ord\left(\frac{a_i}{\d^{m(i)}}\right)= i, \ \exists a,b \text{ such that } m(i)\leq ai+b&\left.\ \forall i   \text{ and }A=\sum_{i=0}^{\infty}\frac{a_i}{\d^{m(i)}}\right\}.\end{split}\end{equation*}

   Moreover the algebraic extension $\K_1\lgw \ovl{\K}$ is the extension of $\K_1$ generated by all the roots of polynomials of the form
   $Z^q+g_1(x)Z^{q-1}+\cdots+g_q$ where $g_i\in\C(x)$ are homogeneous rational fractions of degree $ei$, $1\leq i\leq q$, for some integer $e\in\Q$. A root of such a polynomial is called a homogeneous element of degree $e$. For example, square roots of $x_1$ or of $x_1+x_2$ are homogeneous elements of degree $1/2$. We have $\ovl{\K}\cap \LL=\K_1$.\\
   In the same way he proves that the algebraic closure $\ovl{\K^{an}}$ of the field $\K^{an}$, the fraction field of $\C\{x_1,x_2\}$, can be factorized as $\K^{an}\lgw \K_1^{an}\lgw \ovl{\K^{an}}$ with $\K_1^{an}\subset \LL^{an}$
     where $$ \LL^{an}:=\left\{A\in\wdh{\K}\ / \ \exists \d, \ a_i\in \k[x]  \text{ is homogeneous }\forall i,\  \ord\left(\frac{a_i}{\d^{m(i)}}\right)= i, A= \sum_{i=0}^{\infty}\frac{a_i}{\d^{m(i)}}\hspace{5cm}\right.$$
$$\left. \hspace{1.2cm} \exists a,b \text{ such that } m(i)\leq ai+b\ \forall i \ \text{ and }\ \exists r>0 \text{ such that }  \ \sum_i||a_i||r^i<\infty \right\}$$
and $\displaystyle ||a(x)||:=\max_{|z_i|\leq 1}|a(z_1,z_2)|$ for a homogeneous polynomial $a(x)$.\\
Clearly, $\xi:=f(x_1,\frac{x_2}{x_1})$ is  an element of $\ovl{\K}$ since it is a root of $P(Z)$. Moreover $\xi$ may be written $\xi=\sum_{i=1}^q\xi_i\g^i$ where $\g$ is a homogenous element and $\xi_i\in\LL^{an}\cap \ovl{\K}$ for any $i$, i.e. $\xi\in\LL^{an}[\g]$. Thus the problem is to show that $\xi_i\in\K_1^{an}$ for any $i$, i.e. $\LL^{an}\cap \ovl{\K}=\K_1^{an}$.\\
Then the idea is to resolve, by a sequence of blowing-ups, the singularities of the discriminant locus of $P(Z)$ which is the germ of a plane curve. Let us call $\pi$ this resolution map. Then the discriminant of $\pi^*(P)(Z)$ is normal crossing and $\pi^*(P)(Z)$ defines a germ of surface along the exceptional divisor of $\pi$, denoted by $E$. Let $p$ be a point of  $E$. At this point $\pi^*(P)(Z)$ may factor as a product of polynomials in $Z$ with analytic coefficients, $\xi$ is a root of one of these factors denoted by $Q_1(Z)$ and this root is a germ of an analytic (multivalued) function at $p$. Then the other roots of $Q_1(Z)$ are also in  $\LL^{an}[\g']$ according to the Abhyankar-Jung Theorem \cite{P-R}, for some homogeneous element $\g'$. Thus the coefficients of  $Q_1(Z)$ are in $\LL^{an}$ and are analytic at $p$.\\
Then the idea is to use the special form of the elements of $\LL^{an}$ to prove that the coefficients of $Q_1(Z)$ may be extended as  analytic functions along the exceptional divisor $E$ (the main ingredient in this part is  the Maximum Principle). We can repeat the latter procedure at another point $p'$: we take the roots of $Q_1(Z)$  at $p'$ and using Abhyankar-Jung Theorem we construct new roots of $\pi^*(P)(Z)$ at $p'$ and the coefficients of $Q_2(Z):=\prod_i(Z-\s_i)$, where $\s_i$ runs over all these roots, are in $\LL^{an}$ and are analytic at $p'$ (notice that the decomposition of $P(Z)$ into irreducible factors at $p'$ may be completely different from its decomposition at $p$). Then we extend the coefficients of $Q_2(Z)$ everywhere along $E$. Since $\pi^*(P)(Z)$ has exactly $d$ roots, this process stops after a finite number of steps. The polynomial $Q(Z):=\prod(Z-\s_k)$, where the $\s_k$  are the roots of $\pi(P)(Z)$ that we have constructed, is a polynomial whose coefficients are analytic in a neighborhood of $E$ and it divides $\pi^*(P)(Z)$. Thus,  by Grauert  Direct Image Theorem, there exists a monic polynomial $R(Z)\in\C\{x\}[Z]$ such that $\pi^*(R)(Z)=Q(Z)$. Hence $R(Z)$ divides $P(Z)$, but since $P(Z)$ is irreducible and both are monic,  $P(Z)=R(Z)\in\C\{x\}[Z]$ and the result is proven.  \end{proof}

  
   \subsubsection{One variable Nested Approximation}
 
 In the example of A. Gabrielov \ref{Gab_ex} (3) we can remark that there is only one nest and the nested part of the solutions depends on two variables $x_1$ and $x_2$. In the case this nested part depends only on one variable the nested approximation property is true. This is the following theorem that we state in the more general framework of Weierstrass systems:
 
  \begin{theorem}\label{nested_ana}  \cite[Theorem 5.1]{D-L}
Let $\k$ be a  field and  $\k\llceil x\rrceil$ be a W-system over $\k$. Let $t$ be one variable, $x=(x_1,\ldots,x_n)$, $y=(y_1,\ldots, y_{m+k})$, $f\in\k\llceil t,x,y\rrceil^r$. Let  $\wdh{y}_1,\ldots,\wdh{y}_m\in(t)\k\llbracket t\rrbracket  $ and $\wdh{y}_{m+1},\ldots,\wdh{y}_{m+k}\in(t,x)\k\llbracket t,x\rrbracket  $ satisfy 
$$f(t,x,\wdh{y})=0.$$
 Then for any $c\in\N$ there exist $\wdt{y}_1,\ldots,\wdt{y}_m\in(t)\k\llceil t\rrceil$, $\wdt{y}_{m+1}$,...., $\wdt{y}_{m+k}\in(t,x)\k\llceil t, x\rrceil$ such that 
 $$f(t,x,\wdt{y})=0 \text{ and }\wdh{y}-\wdt{y}\in (t,x)^c.$$
 
 \end{theorem}
 
 \begin{example}
 The main example  is the case where $\k$ is a valued field and $\k\llceil x\rrceil$ is the ring of convergent power series over $\k$. When $\k=\C$ this statement is already mentioned as a known result in \cite{Ga1} without any proof or reference.\\
 But even for algebraic power series this statement is interesting since its proof is really easier and more effective than Theorem \ref{nested_alg}. 
 \end{example}
 
 \begin{proof}
 The proof is very similar to the  proof of Theorem \ref{nested_alg}.\\
Set $u:=(u_1,\ldots,u_j)$, $j\in\N$ and set 
 \begin{equation*}\begin{split}\k\llbracket t\rrbracket [\langle u\rangle]&:=\{f(t,u)\in\k\lb t,u\rb\ / 
 \exists s\in\N,\ g(z_1,\ldots,z_s,u)\in\k\llceil z,u\rrceil,\\
 & z_1(t),\ldots, z_s(t)\in(t)\k\lb t\rb \text{ such that } f(t,u)=g(z_1(t),\ldots, z_s(t),u) \}.\end{split}\end{equation*}
 The rings $\k\lb t\rb [\langle u\rangle ]$ form a $W$-system over $\k\lb t\rb$  \cite[Lemma 52]{D-L} (it is straightforward to check it since $\k\llceil x\rrceil$ is a $W$-system over $\k$ - in particular, if char$(\k)>0$, vi) of Definition \ref{W-sys} is satisfied since v) of Definition \ref{W-sys} is satisfied for $\k\llceil x\rrceil$). By Theorem \ref{W} applied to 
 $$f(t,\wdh{y}_1,\ldots,\wdh{y}_m, y_{m+1},\ldots,y_{m+k})=0$$ 
 there exist $\ovl{y}_{m+1},\ldots,\ovl{y}_{m+k}\in\k\lb t\rb [\langle x\rangle ]$ such that 
 $$f(t, \wdh{y}_1,\ldots,\wdh{y}_m,\ovl{y}_{m+1},\ldots,\ovl{y}_{m+k})=0$$ 
 and $\ovl{y}_i-\wdh{y}_i\in (t,x)^c$ for $m<i\leq m+k$.\\
  Let  us write 
  $$\displaystyle \ovl{y}_i=\sum_{\a\in\N^n} h_{i,\a}(\wdh{z})x^{\a}$$
   with $\displaystyle \sum_{\a\in\N^n} h_{i,\a}(z)x^{\a}\in\k\llceil z,x\rrceil$, $z=(z_1,\ldots,z_s)$ is a vector of new variables  and $\wdh{z}=(\wdh{z}_1,\ldots,\wdh{z}_s)\in\k\lb t\rb^s$. We can write
\begin{equation*}\begin{split}
f\left(t,x,y_1,\ldots,y_m, \sum_{\a}h_{m+1,\a}(z)x^{\a},\ldots,  \sum_{\a}h_{m+k,\a}(z)x^{\a} \right)&=\\
=\sum_{\a}G_{\a}(t,y_1,\ldots,&y_m,z)x^{\a}\end{split}\end{equation*}
 where $G_{\a}(t,y_1,\ldots,y_m,z)\in\k\llceil t,y_1,\ldots,y_m,z\rrceil$ for all $\a\in\N^n$. Thus 
 $$(\wdh{y}_1,\ldots,\wdh{y}_m,\wdh{z}_1,\ldots,\wdh{z}_s)\in\k\lb t\rb^{m+s}$$ is a solution of the equations $G_{\a}=0$ for all $\a\in\N^n$. Since $\k\llceil t,y_1,\ldots,y_m,z\rrceil$ is Noetherian, this system of equations is equivalent to a finite system 
 $G_{\a}=0$ with $\a\in E$ where $E$ is a finite subset of $\N^n$. Thus by Theorem \ref{W} applied to the system $G_{\a}(t,y_1,\ldots,y_m,z)=0$, $\a\in E$, there exist $\wdt{y}_1,\ldots,\wdt{y}_m$, $\wdt{z}_1,\ldots,\wdt{z}_s\in\k\llceil t\rrceil$ such that $\wdt{y}_i-\wdh{y}_i$, $\wdt{z}_j-\wdh{z}_j\in (t)^c$, for $1\leq i\leq m$ and $1\leq j\leq s$, and $G_{\a}(t, \wdt{y}_1,\ldots,\wdt{y}_m,\wdt{z})=0$ for all $\a\in E$, thus $G_{\a}(t,\wdt{y}_1,\ldots,\wdt{y}_m,\wdt{z})=0$ for all $\a\in\N^n$.\\
 Set $\displaystyle \wdt{y}_i=\sum_{\a\in\N^n}h_{i,\a}(\wdt{z})x^{\a}$ for $m<i\leq m+k$. Then $\wdt{y}_1,\ldots,\wdt{y}_{m+k}$ satisfy the conclusion of the theorem.

 \end{proof}

 \begin{remark}
 The proof of this theorem uses in an essential way the Weierstrass Division Property (in order to show that $\k\llbracket t\rrbracket [\langle u\rangle]$ is a Noetherian local ring, which is the main condition to use Theorem \ref{Pop_app}. The Henselian and excellent conditions may be deduced quite easily from the Weierstrass Division Property).\\
  On the other hand the Weierstrass Division Property (at least in dimension 2) is necessary to obtain this theorem. Indeed if $\k\llceil x\rrceil$ is a family of rings satisfying Theorem \ref{nested_ana} and $f(t,y)\in\k\llceil t,y\rrceil$ is $y$-regular of order $d$ ($t$ and $y$ being single variables) and $g(t,y)\in\k\llceil t,y\rrceil$ is another series, by the Weierstrass Division Theorem for formal power series we can write  in a unique way
 $$g(t,y)=\wdh{q}(t,y)f(t,y)+\wdh{r}_0(t)+\wdh{r}_1(t)y+\cdots+\wdh{r}_{d-1}(t)y^{d-1}$$
 where  $\wdh{q}(t,y)\in\k\lb t,y\rb$ and $\wdh{r}_i(t)\in\k\lb t\rb$ for all $i$. Thus by Theorem \ref{nested_ana}, $\wdh{q}(t,y)\in\k\llceil t,y\rrceil$ and $\wdh{r}_i(t)\in\k\llceil t\rrceil$ for all $i$. This means that $\k\llceil t,y\rrceil$ satisfies the Weierstrass Division Theorem.\\
 \\ 
  \index{Denjoy-Carleman function germs} For example let $C_n\subset \k\lb x_1,\ldots,x_n\rb$ be the ring of  germs of $\k$-valued Denjoy-Carleman functions defined at the origin of $\R^n$, where $\k=\R$ or $\C$. One can look at   \cite{Th} for the precise definitions and properties of these rings. Roughly speaking these are  germs of $\k$-valued $\mathcal{C}^{\infty}$-functions whose derivatives at each point of a neighborhood of the origin satisfy inequalities of the form \eqref{mouze} in Remark \ref{mouze'} for some given logarithmically convex sequence of positive numbers $(m_k)_k$, but we need to require additional properties on $(m_k)_k$  in order to insure that these classes of functions are quasi-analytic, i.e. such that the Taylor map is injective. For a given logarithmically convex sequence $\underline m=(m_k)_k$ we denote by  $C_n(\underline m)$ these rings of function germs.\\
   If $\k\{x_1,\ldots,x_n\}\not\subset C_n(\underline m)$ it is still an open problem to know if $C_n(\underline m)$ is Noetherian or not for $n\geq 2$ ($C_1(\underline m)$ is always a discrete valuation ring, thus it is Noetherian). These rings have similar properties to the Weierstrass systems: these are Henselian local rings whose maximal ideal is generated by $x_1,\ldots,x_n$, the completion of $C_n(\underline m)$ is $\k\lb x_1,\ldots,x_n\rb$, for every $n$ $C_n(\underline m)$ is  stable by partial derivatives, by division by coordinates functions or by composition. The only difference with Weierstrass systems is that $C_n(\underline m)$ does not satisfy the  Weierstrass Division Theorem. \\
   For instance, there exist $f\in C_1(\underline m)$ and $\wdh{g}\in\k\llbracket t\rrbracket \backslash C_1(\underline m)$ such that $f(t)=\wdh{g}(t^2)$ (see the proof of  \cite[Proposition 2]{Th}). This implies that
 \begin{equation}\label{eqDC}f(t)=(t^2-y)\wdh{h}(t,y)+\wdh{g}(y)\end{equation}
for some formal power series $\wdh{h}(t,y)\in\k\llbracket t,y\rrbracket$, but Equation (\ref{eqDC}) has no nested solution  $(g,h)\in C_1(\underline m)\times C_2(\underline m)$. \\
 On the other hand if the rings $C_n(\underline m)$ were Noetherian, since their completions are regular local rings they would be regular. Then using Example \ref{ex_excellent} iii) we see that they would be excellent (see also \cite{ElKh}). Thus these rings would satisfy Theorem \ref{Pop_app} but they do not satisfy Theorem \ref{nested_ana} since Equation \eqref{eqDC} has no solutions in $C_1(\underline m)\times C_2(\underline m)$.
  \end{remark}

 \subsection{Other examples of approximation with constraints}
 \subsubsection{Some examples}
 We present here some examples of positive or negative answers to Problems 1 and 2 in several contexts.
  \begin{example}[Cauchy-Riemann equations]\cite{Mi2}\index{Cauchy-Riemann equations}
 P. Milman proved the following theorem:
 \begin{theorem}\label{CR_app}
 
 Let $f\in\C\{x,y,u,v\}^r$ where $x:=(x_1,\ldots,x_n)$, $y=(y_1,\ldots,y_n)$, $u:=(u_1,\ldots, u_m)$, $v:=(v_1,\ldots,v_m)$. Then the set of convergent solutions of the following system:
 \begin{equation}\label{CR}\left\{\begin{aligned}&f(x,y,u(x,y),v(x,y))=0\\
 &\frac{\partial u_k}{\partial x_j}(x,y)-\frac{\partial v_k}{\partial y_j}(x,y)=0\\
  &\frac{\partial v_k}{\partial x_j}(x,y)+\frac{\partial u_k}{\partial y_j}(x,y)=0
  \end{aligned}\right.
  \end{equation}
 is dense (for the $(x,y)$-adic topology) in the set of formal solutions of this system.
  \end{theorem}
 \end{example}
 \begin{proof}[Hints on the proof]
 
 Let $(\wdh{u}(x,y),\wdh{v}(x,y))\in\C\llbracket x,y\rrbracket^{2m}$ be a solution of (\ref{CR}). Let us set $z:=x+iy$ and $w:=u+iv$. In this case the Cauchy-Riemann equations of (\ref{CR}) are equivalent to saying that $\wdh{w}(z,\ovl{z}):=\wdh{u}(x,y)+i\wdh{v}(x,y)$ does not depend on $\ovl{z}$. Let $\phi : \C\{z,\ovl{z},w,\ovl{w}\}\lgw \C\llbracket z,\ovl{z}\rrbracket$ and $\psi:\C\{z,w\}\lgw\C\llbracket z\rrbracket$ be the morphisms defined by
 $$\phi(h(z,\ovl{z},w,\ovl{w})):=h(z,\ovl{z},\wdh{w}(z),\ovl{\wdh{w}(z)})\ \text{ and } \ \psi(h(z,w)):=h(z,\wdh{w}(z)).$$
 Then
 $$f\left(\frac{z+\ovl z}{2}, \frac{z-\ovl z}{2i},\frac{w+\ovl w}{2},\frac{w-\ovl w}{2i}\right)\in \Ker(\phi).$$
 Milman proved that 
 $$\Ker(\phi)=\Ker(\psi).\C\{z,\ovl{z},w,\ovl{w}\}+\ovl{\Ker(\psi)}.\C\{z,\ovl{z},w,\ovl{w}\}.$$
 Since $\Ker(\psi)$ (as an ideal of $\C\{z,w\}$) satisfies Theorem \ref{Ar68}, the result follows.
  \end{proof}
 This proof does not give the existence of an Artin function for this kind of system, since the proof consists in reducing Theorem \ref{CR_app} to Theorem \ref{Ar68}, and this reduction depends on the formal solution of (\ref{CR}).  Nevertheless in \cite{Hi-Ro}, it is proven that such a system admits an Artin function using ultraproducts methods. The survey \cite{Mir} is a good introduction for applications of Artin Approximation in CR geometry.
 
 \begin{example}[Approximation of equivariant solutions]\cite{BM79} Let $G$ be a reductive algebraic group. Suppose that $G$ acts linearly on $\C^n$ and $\C^m$. We say that $y(x)\in\C\llbracket x\rrbracket^m$ is equivariant if $y(\s x)=\s y(x)$ for all $\s\in G$. E. Bierstone and P. Milman proved that, in Theorem \ref{Ar68}, the constraint for the solutions of being equivariant may be preserved for convergent solutions:
 
   \begin{theorem}\cite{BM79}\index{equivariant Artin approximation}
 Let $f(x,y)\in\C\{x,y\}^r$. Then the set of equivariant convergent solutions of $f=0$ is dense in the set of equivariant formal solutions of $f=0$ for the $(x)$-adic topology.\\
 This result remains true is we replace $\C$ (resp. $\C\{x\}$ and $\C\{x,y\}$) by any field of characteristic zero $\k$ (resp. $\k\langle x\rangle$ and $\k\langle x,y\rangle$).
 \end{theorem}
 Using ultraproducts methods we may probably prove  that Problem 2 has a positive answer in this case.
 \end{example}
 \begin{example}\label{ex_Q}\cite{BDLvdD}
 Let $\k$ be a characteristic zero field. Let us consider the following differential equation:
 \begin{equation}\label{EQ1}a^2x_1\frac{\partial f}{\partial x_1}(x_1,x_2)-x_2\frac{\partial f}{\partial x_2}(x_1,x_2)=\sum_{i,j\geq 1}x_1^ix_2^j\ \ \left(=\left(\frac{x_1}{1-x_1}\right)\left(\frac{x_2}{1-x_2}\right)\right).\end{equation}
 For $a\in\k$, $a\neq 0$, this equation has only the following solutions
 $$f(x_1,x_2):=b+\sum_{i,j\geq 1}\frac{x_1^ix_2^j}{a^2i-j},\ \ b\in\k$$
 which are well defined if and only if $a\notin\Q$.
Let us consider the following system of equations with constraints (where $x=(x_1,\ldots,x_5)$):
\begin{equation}\label{EQ2}\left\{\begin{aligned}
&  y_8^2x_1y_5(x_1,x_2)-x_2y_7(x_1,x_2)=\sum_{i,j\geq 1}x_1^ix_2^j\\
  & y_1(x_1,x_2)=y_2(x_3,x_4,x_5)+(x_1-x_3)z_1(x)+(x_2-x_4)z_2(x)\\
& \begin{split}y_2(x_3,x_4,x_5)=y_1(x_1,x_2)+x_5y_4&(x_1,x_2)+\\
          x_5^2y_5(x)+&(x_3-x_1-x_5)z_3(x)+(x_4-x_2)z_4(x)\end{split}\\
 & \begin{split} y_3(x_3,x_4,x_5)=y_1(x_1,x_2)+x_5y_6&(x_1,x_2)+\\
 x_5^2y_7(x)+&(x_3-x_1)z_5(x)+(x_4-x_2-x_5)z_6(x)\end{split}\\
& y_8\in\k\ \text{ and }\  y_8y_{9}=1.
 \end{aligned}\right.\end{equation}
It is straightforward, using the tricks of  Lemma \ref{Taylor} and Example \ref{diff}, to check that $(a,f(x_1,x_2))$ is a solution of the equation (\ref{EQ1}) if and only if the system (\ref{EQ2}) has a constrained solution $(y_1,\ldots,y_{9},z_1,\ldots, z_6)$ with $y_1=f$ and $y_8=a$. Moreover, if $(y_1,\ldots,y_{9},z_1,\ldots, z_6)$ is a constrained solution of  Equation (\ref{EQ2}), then $(y_8,y_1)$ is a solution of (\ref{EQ1}). \\
Thus  (\ref{EQ2}) has no constrained solution in $\Q\llbracket x\rrbracket$. But clearly,  (\ref{EQ1}) has constrained solutions in $\frac{\Q[x]}{(x)^c}$ for any $c\in\N$ and the same is true for  (\ref{EQ2}). This shows that Proposition \ref{prop_mod} is not valid if the base field is not $\C$.
 
 \end{example}
 
 \begin{example}\cite{BDLvdD}\label{lastone}
 Let us assume that $\k=\C$ and consider the previous example.
 The system of equations (\ref{EQ2}) does not admit an Artin function. Indeed, for any $c\in\N$, there is $a_c\in\Q$, such that (\ref{EQ2}) has a solution modulo $(x)^c$ with $y_8=a_c$. But there is no solution in $\C\llbracket x\rrbracket $ with $y_8=a_c$ modulo $(x)$, otherwise $y_8=a_c$ which is not possible.\\
 Thus systems of equations with constraints do not satisfy Problem 2 in general.
 
 \end{example}

 \begin{example}\cite{Ro8} Let $\phi :\C\{x\}\lgw \C\{y\}$ be a morphism of complex analytic algebras and let  $\phi_i(y)$ denote the image of $x_i$ by $\phi$. Let us denote by $\wdh{\phi} :\C\llbracket x\rrbracket \lgw \C\llbracket y\rrbracket$ the induced morphism between the completions. According to a lemma of Chevalley (Lemma 7 of \cite{Ch}), there exists a function $\b:\N\lgw \N$ such that $\wdh\phi^{-1}((y)^{\b(c)})\subset \Ker(\wdh{\phi})+(x)^{c}$ for any $c\in\N$. It is called the \emph{Chevalley function} \index{Chevalley function} of $\phi$. Using Lemma \ref{Taylor} we check easily that this function $\b$ satisfies the following statement (in fact the two statements are equivalent \cite{Ro8}):
 Let $\ovl{f}(x)\in\C\llbracket x\rrbracket$ and $\ovl{h}_i(x,y)\in\C\llbracket x,y\rrbracket$, $1\leq i\leq n$, satisfy
 $$\ovl{f}(x)+\sum_{i=1}^n(x_i-\phi_i(y))\ovl{h}_i(x,y)\in (x,y)^{\b(c)}.$$
 Then there exist $\wdt{f}(y)\in\C\llbracket x\rrbracket$, $\wdt{h}_i(x,y)\in\C\llbracket x,y\rrbracket$, $1\leq i\leq n$, such that
  \begin{equation}\label{gab_iz}\wdt{g}(x)+\sum_{i=1}^n(x_i-\phi_i(y))\wdt{h}_i(x,y)=0\end{equation}
  and $\wdt{f}(x)-\ovl{f}(x)\in (x)^c$, $\wdt{h}_i(x,y)-\ovl{h}_i(x,y)\in(x,y)^c$, $1\leq i\leq n$.\\
  In particular Problem 2 has a positive answer for Equation (\ref{gab_iz}), but not Problem 1 (see Example \ref{Gab_ex}).
  In fact, the conditions of Theorem \ref{Gab} are equivalent to the fact that $\b$ is bounded by a linear function \cite{Iz2}.\\
 \\
The following example is given in  \cite{Ro8} and is inspired by Example \ref{Gab_ex}. Let  $\a\ :\ \N\lgw \N$ be an increasing function. Let $(n_i)_i$ be  a sequence of integers such that $n_{i+1}>\a(n_i+1)$ for all  $i$ and such that the convergent power series  $\xi(Y):=\sum_{i\geq 1}Y^{n_i}$ is not algebraic over $\C(Y)$. Then we define the morphism  $\phi\ : \C\{x_1,\,x_2,\,x_3\}\lgw \C\{y_1,\,y_2\}$ in the following way:
$$(\phi(x_1),\,\phi(x_2),\,\phi(x_3))=(y_1,\,y_1y_2,\,y_1\xi(y_2)).$$
It is easy to prove that $\wdh{\phi}$ is injective exactly as in Example \ref{Gab_ex}.
For any integer $i$ we define:
$$\ovl{f}_i:=x_1^{n_i-1}x_3-\left(x_2^{n_1}x_1^{n_i-n_1}+\cdots+x_2^{n_{i-1}}x_1^{n_i-n_{i-1}}+x_2^{n_i}\right).$$
Then
$$\phi(\ovl{f}_i)=y_1^{n_i}\xi(y_2)-y_1^{n_i}\sum_{k=1}^{i}y_2^{n_k}\in(y)^{n_i+n_{i+1}}\subset (y)^{\a(n_i+1)}$$
but $\ovl{f}_i\notin (x)^{n_i+1}$ for any $i$. Thus the Chevalley function of $\phi$ satisfies $\b(n_i+1)>\a(n_i+1)$ for all $i\in\N$. Hence $\limsup \frac{\b(c)}{\a(c)}\geq 1$.
In particular if the growth of $\a$ is too big, then $\b$ is not recursive. 
  \end{example}
  \subsubsection{Artin Approximation for differential equations}

 These examples along with the trick of Example \ref{diff} are a motivation to study the Artin Approximation Property for systems of differential equations. These examples show that there is no direct generalization of the Artin Approximation Theorems to differential equations,  in the sense that the formal solutions of a system of differential equations with polynomial coefficients cannot be approximated in general by convergent solutions. Nevertheless J. Denef and L. Lipshitz showed that there exist differential analogues of them in the one variable case. Let us explain this.
 
 \begin{definition}\cite{DL2}
Let $R\subset \k\lb x\rb$ be a differential ring where $x$ is a single variable (i.e. the  differential $\frac{\partial}{\partial x} : \k\lb x\rb\lgw \k\lb x\rb$ send $R$ into $R$). 
A power series  $f$ is called \emph{differentially algebraic} over $R$ if there is a non-zero differential polynomial in one variable over $R$ which vanishes on $f$.
\end{definition}
When $R\subset \k\lb x\rb$ is a differential ring and $y=(y_1,\ldots, y_m)$ is a vector of variables we denote by $R[y,\partial y]$ the ring of differential polynomials in $y$, i.e. the ring of polynomials in the countable number of variables $y$ and $y_{i,n}$ where $1\leq i\leq m$ and $n\in \N\backslash \{(0)\}$. The differential operator $\frac{\partial}{\partial x}$ extend to $R[y,\partial y]$ by
$$\frac{\partial}{\partial x} y_j=y_{j,1} \text{ and } \frac{\partial}{\partial x}y_{j,n}:=y_{j,n+1}\ \ \forall i, n.$$
Then we have the following analogue of the Artin Approximation Theorem in the differential case:

\begin{theorem}\cite{DL2}
Let $R\subset \k\lb x\rb$ be a differential ring and let $f(x,y)\in R[y, \partial y]^r$. Let $c\in\N$ and $y(x)$ be a formal power series solution
$$f(x,y(x))=0.$$
Then there exists a solution $\wdt y(x)$, whose components are differentially algebraic over $R$,
$$f(x,\wdt y(x))=0$$
and $\wdt y(x)-y(x)\in (x)^c.$

\end{theorem}
As seen in Example \ref{euler} a  power series differentially algebraic over $\C[x]$ is not convergent in general, but they are always Gevrey power series \cite{Ma03}. But these have also  good combinatorial properties (see \cite{St} or \cite{D-L} for instance).\\
\\
Moreover we have the following analogue of the Strong Artin approximation Theorem for differential equations (see \cite{D-L} for the one variable case and \cite{PR} for the several variables case):

\begin{theorem}\label{DLDL}\cite{D-L,PR}
Let $f(x,y)$ be a system of differential polynomials in $y$ with formal power series coefficients over a field $\k$. 
\begin{enumerate}

\item If $x$ is a single variable, assume that $\k$ is a characteristic zero field which is either algebraically closed, a real closed field or a Henselian valued field.

\item If $x$ is a vector of variables, assume that $\k$ is a finite field, an uncountable algebraically closed field or an ultraproduct of fields.
\end{enumerate}
If $f=0$ has approximate solutions up to any order then $f=0$ has a solution in $\k[[x]]^m$.
\end{theorem}

\begin{remark}
The case when $\k=\C$ is quite trivial as indicated in Remark 2.11 \cite{D-L}.\\
Let us mention that there are examples of  partial differential equations with coefficients in $\R[x_1,\ldots, x_n]$ (resp. $\ovl\Q[x_1,\ldots, x_n]$) with $n\geq 2$, that do not satisfy the conclusion of Theorem \ref{DLDL} (see \cite{D-L}). These show that there is really a difference between the univariate case and the case of several variables.
\end{remark}

\begin{remark}
Let us consider the partial differential equation of Example \ref{ex_Q} where $\k=\C$. Let us replace $a$ by a function $g(x_1,x_2)$. The condition for $a$ to be in $\C$ is equivalent to saying that $\frac{\partial g}{\partial x_1}=\frac{\partial g}{\partial x_2}=0$.\\
Thus for the following system of partial differential equations
$$\left\{\begin{array}{c} g^2x_1\frac{\partial f}{\partial x_1}-x_2\frac{\partial f}{\partial x_2}=\sum_{i,j\geq 1}x_1^ix_2^j\\
\frac{\partial g}{\partial x_1}=\frac{\partial g}{\partial x_2}=0\end{array}\right.$$
Example \ref{lastone} shows that there is no integer $\b$ such that every approximate solution $(f,g)$ of this system up to order $\b$ has a solution $(\wdt f,\wdt g)$ such that $\wdt g-g\in (x)$. Hence the analogue of the Strong Artin Approximation Theorem \ref{SAP_formal} for partial differential equations is not valid in general.

\end{remark}


 \section{Appendix A: Weierstrass Preparation Theorem}\label{weierstrass}
 In this part we set $x:=(x_1,\ldots,x_n)$ and $x':=(x_1,\ldots,x_{n-1})$. Moreover $R$ always denotes a local ring of maximal ideal $\m$ and residue field $\k$ (if $R$ is a field, $\m=(0)$). A local subring of $R\lb x\rb$ will be a subring $A$ of $R\lb x\rb$ which is a local ring and whose maximal ideal is $(\m+(x))\cap  A$.
 
 \begin{definition}\label{regular-W}
If $f\in R\lb x\rb$ we say that $f$ is $x_n$-regular of order $d$ if the image of $f$ in $\frac{R\lb x\rb}{\m+(x')}\simeq \k\lb x_n\rb$  has the form $u(x_n)x_n^d$ where   $u(x_n)$ is invertible in $\k\lb x_n\rb$.\\
When $R=\k$ is a field this just means that $f(0,\ldots,0,x_n)=u(x_n)x_n^d$ where $u(x_n)$ is invertible.
 \end{definition}
 
 \begin{definition}\label{WDT}\index{Weierstrass Division Property} 
 Let $A$ be a local subring of $R\lb x\rb$. We say that $A$ satifies the \emph{Weierstrass Division Property} if for any $f$, $g\in A$ such that $f$ is $x_n$-regular of order $d$, there exist $q\in A$ and $r\in (A\cap R\lb x'\rb)[x_n]$ such that $\deg_{x_n}(r)<d$ and $g=qf+r$. In this case $q$ and $r$ are unique.

 \end{definition}

 \begin{definition}\label{WPT}\index{Weierstrass Preparation Property}
 Let $A$ be a local subring of $R\lb x\rb$. We say that $A$ satisfies the \emph{Weierstrass Preparation Property} if for any $f\in A$ which is $x_n$-regular, there exist an integer $d$, a unit $u\in A$ and $a_1(x'),\ldots,a_{d}(x')\in A\cap  (\m+(x'))R\lb x'\rb$ such that 
 $$f=u\left(x_n^d+a_1(x')x_n^{d-1}+\cdots+a_d(x')\right).$$
 In this case $f$ is necessarily regular of order $d$ with respect to $x_n$ and $u$ and the $a_i$ are unique. The polynomial $x_n^d+a_1(x')x_n^{d-1}+\cdots+a_d(x')$ is called the \emph{Weierstrass polynomial} of $f$.
 \end{definition}
 
As mentioned in \ref{AAATWDT}, the Weierstrass Preparation Property implies the Implicit Function Theorem:

\begin{lemma}\label{WPPIFT}
Let $A$ be a local subring of $R\lb x\rb$, with $x$ a single variable. If $A$ has the Weierstrass Division Property then $A$ satisfies the Implicit Function Theorem as follows:\\
Let $f(x)\in A$ be such that
$$f(0)\in\m_A\cap R\ \text{ and}\ f'(0)\notin\m_A\cap R.$$
Then there is unique $a\in \m_A\cap R$ such that $f(a)=0$.
\end{lemma}

\begin{proof}
Let us write $f=\sum_{k\in\N} f_kx^k$ with $f_k\in R$ for every $k$. Then we have
$$f(0)=f_0\in \m_A\cap R,\ f'(0)=f_1\notin \m_A\cap R.$$
So $f$ is $x$-regular of order $1$. Thus by the Weierstrass Preparation Property there is a unit $u(x)\in A$, and $a\in \m_A\cap R$ such that
$$f(x)=u(x)\cdot(x-a).$$
Since $a\in\m_A\cap R$, $u(a)$ is well defined, therefore $f(a)=0$.\\
Moreover if $a'\in \m_A\cap R$ satisfies $f(a')=0$ we have 
\begin{equation}\label{ttt} f(a')=f(a)+\frac{\partial f}{\partial x}(a)(a'-a)+\sum_{l\geq 2}\frac{1}{l!}\frac{\partial^l f}{\partial x^l}(a)(a'-a)^l.\end{equation}
If $a\neq a'$, set $k=\ord(a'-a)\geq 1$. Since $\frac{\partial f}{\partial x}(0)\notin \m_ R$ and $a\in\m_R$, we have that $\frac{\partial f}{\partial x}(a)\notin \m_R$. Thus $\frac{\partial f}{\partial x}(a)$ is a unit and $\frac{\partial f}{\partial x}(a)(a'-a)\notin \m_R^{k+1}$. But $f(a)=f(a')=0\in\m_R^{k+1}$ and
$$\sum_{l\geq 2}\frac{1}{l!}\frac{\partial^l f}{\partial x^l}(a)(a'-a)^l\in\m_A^{2k}\subset\m_R^{k+1}.$$
This is a contradiction in view of \eqref{ttt}, therefore $a=a'$.

\end{proof}

 \begin{lemma}\label{D->P}
 A local subring $A$ of $R\lb x\rb$ satisfying the Weierstrass Division Property  satisfies the Weierstrass Preparation Property.
  \end{lemma}
  \begin{proof}
If $A$ has the Weierstrass Division Property and if $f\in A$ is $x_n$-regular of order $d$, then we can write $x_n^d=qf+r$ where    $r\in (A\cap R\lb x'\rb)[x_n]$ such that $\deg_{x_n}(r)<d$. Thus $qf=x_n^d-r$. Because $f$ is $x_n$-regular of order $d$,  $q$ is invertible in $R\lb x\rb$ and $r\in (\m+(x'))$. Thus $q\notin (\m+(x))$ and $q$ is invertible in $A$. Hence $f=q^{-1}(x_n^d-r)$.
 \end{proof}

In fact the converse implication is true under some mild conditions:
\begin{lemma}\label{P->D}\cite{CL}
Let $A_n$ be a subring of $R \lb x_1,\ldots,x_n\rb$ for all $n\in\N$ such that 
\begin{itemize}
\item[i)] $A_{n+m}\cap R \lb x_1,\ldots,x_n\rb=A_n$ for all $n$ and $m$,
 \item[ii)] if $f\in A_{n}$ is written $f=\sum_{k\in\N}f_kx_{n}^k$ with $f_k\in R \lb x'\rb$ for all $k$, then $f_k\in A_{n-1}$ for all $k$.
\item[iii)] $A_n$ is stable by permutation of the $x_i$.
 \end{itemize}
 Then $A_n$ has the Weierstrass Division Property if $A_n$ and $A_{n+1}$ have the Weierstrass Preparation Property.
\end{lemma}

\begin{proof}
Let $f(x)\in A_n$ be $x_n$-regular of order $d$. By the Weierstrass Preparation Property for $A_n$ we may write
$$f=u\left(x_n^d+a_1(x')x_n^{d-1}+\cdots+a_d(x')\right)=uP$$
where $u$ is a unit in $A_n$ and $P\in A_{n-1}[x_n]$.
Now let $g(x)\in A_n$ and   set $h:=P+x_{n+1}g$. Then $h$ is also $x_n$-regular of order $d$, thus by the Weierstrass Preparation Property for $A_{n+1}$ we may write $h=vQ$ where $v$ is a unit and $Q$ a monic polynomial of degre $d$ in $x_n$. Let us write
$$v=\sum_{k\in\N}v_kx_{n+1}^k \ \ \text{ and }\ \ Q=\sum_{k\in\N}Q_kx_{n+1}^k$$
where $v_k\in A_n$ and $Q_k\in A_{n-1}[x_n]$ for all $k$. Since $Q$ is monic in $x_n$ of degree $d$, the polynomial $Q_0$ is  monic  in $x_n$ of degree $d$ and $\deg_{x_n}(Q_k)< d$ for $k\geq 1$.\\
 We deduce from this that
$$v_0Q_0=P \ \ \text{ and } \ \ v_1Q_0+v_0Q_1=g.$$
Since $Q_0$ and $P$ are monic polynomials in $x_n$ of degree $d$ the first equality implies that $v_0=1$ and $Q_0=P$. Thus the second yields $g=v_1P+Q_1$, i.e.
$$g=v_1u^{-1}f+Q_1$$
and $Q_1\in A_{n-1}[x_n]$ is a polynomial in $x_n$ of degree $< d$. Thus the Weierstrass Division Property holds.
 \end{proof}
 
 \begin{theorem}\label{WDP}
 The following rings have the Weierstrass Division Property:
 \begin{enumerate}
 \item[i)] The ring $A=R \lb x\rb$ where $R $ is a complete local ring (\cite{Bo}).
 \item[ii)] The ring $A=R \langle x\rangle$ of algebraic power series where $R $ is a field or a Noetherian Henselian local ring of characteristic zero which is analytically normal \cite{Laf1,Laf,Ro15}.
 \item[iii)] The ring $A=\k\{x\}$ of convergent power series over a valued field $\k$ (see \cite{Na} or \cite{To} where is given a nice short proof using an invertibility criterion of a linear map between complete topological groups).
 \item[iv)] The ring $A$ of germs of $\mathcal C^\infty$-functions at the origin of $\R^n$ \cite{Mal1}. In this case $A$ is not a Noetherian ring.
 \end{enumerate}
 \end{theorem}
 
 \begin{remark}\label{change_x_n-regular}
 Let $f\in\k\lb x\rb$ where $\k$ is an infinite field and let $d:=\ord(f)$. Let $f_d$ be the initial term of $f$. Since $\k$ is infinite there exists $(\la_1,\ldots,\la_{n-1})\in\k^{n-1}$ such that $c:=f_d(\la_1,\ldots,\la_{n-1},1)\neq 0$. So let us consider the linear change of variables defined by
 $$x_i\lgm x_i+\la_ix_n \text{ for }i<n$$
 $$x_n\lgm x_n.$$
 Then under this linear change of variables $f$ is transformed into a new power series $g$ whose initial term has degree $d$ and is of the form
 $$cx_n^d+\e$$
 where $\e\in(x_1,\ldots,x_{n-1})$. Thus $g$ is $x_n$-regular. Hence any formal power series may be transformed into a $x_n$-regular power series of degree $d=\ord(f)$ by a linear change of variables.\\
 In the case when $\k$ is finite we can also transform any formal power series $f$ into a $x_n$-regular one but for this we have to use non-linear changes of variables (see  \cite[Lemma 6.11]{Ar69}). Moreover after this change of coordinates $f$ is $x_n$-regular of degree bigger than $\ord(f)$ (with no equality in general).
 \end{remark}
 
 \section{Appendix B: Regular morphisms and excellent rings}\label{regular_app}
We give here the definitions and the main properties of regular morphisms and excellent rings. For more details the reader may consult \cite[15.32 and 15.42]{ST}  or \cite{Ma'}.
 
 \begin{definition}\label{regular}
  Let $\phi: A\lgw B$ be a morphism of Noetherian rings. We say that $\phi$ is a \index{regular morphism} \emph{regular morphism} if it is flat and if for any prime ideal $\p$ of $A$, the $\kappa(\p)$-algebra $B\otimes_A \kappa(\p)$ is geometrically regular (where $\kappa(\p):=\frac{A_{\p}}{\p A_{\p}}$ is the residue field of  $A_{\p}$). This means that $B\otimes_{A}\K$ is a regular Noetherian ring for any finite field extension $\K$ of $\kappa(\p)$.
   \end{definition}
   
   \begin{example}\label{ex_reg} 
  \begin{enumerate}
\item[i)]   If $A$ and $B$ are fields, $A\lgw B$ is regular if and only if $B$ is a separable field extension of $A$.
\item[ii)] If $A$ is excellent (the definition of an excellent ring is given below), for any ideal $I$ of $A$, the morphism $A\lgw \wdh{A}$ is regular where $\wdh{A}:=\underset{{\longleftarrow}}{\lim} \frac{A}{I^n}$ denotes the $I$-adic completion of $A$ \cite[7-8-3]{EGAIV-2}.
\item[iii)] If $V$ is a discrete valuation ring,  the completion morphism $V\lgw \wdh{V}$ is regular if and only if $\Frac(V)\lgw \Frac(\wdh{V})$ is separable. Indeed, $V\lgw \wdh{V}$ is always flat and this morphism induces an isomorphism on the residue fields.

\item[iv)] Let $X$ be a compact Nash manifold, let $\mathcal{N}(X)$ be the ring of Nash functions on $X$ and let $\O(X)$ be the ring of real analytic functions on $X$. Then the natural inclusion $\mathcal{N}(X)\lgw \O(X)$ is regular (cf. \cite{C-R-S} or \cite{CRS} for a survey on the applications of General NŽron Desingularization to the theory of sheaves of Nash functions on Nash manifolds).
\item[v)] Let $L\subset \C^n$ be a compact polynomial polyhedron and $B$ the ring of holomorphic function germs at $L$. Then the morphism of constants $\C\lgw B$ is regular (cf. \cite{Lem}). This example and the previous one enable the use of Theorem \ref{Popescu} to show global approximation results in complex geometry or real geometry. The paper \cite{Bi} also provides a proof of a global Artin approximation theorem whose proof is based on basic methods of analytic geometry and not on the General N\'eron desingularization Theorem.  \index{global Artin approximation} The papers \cite{Bi2,BP} give stronger forms of this global Artin approximation theorem.
\end{enumerate}
 \end{example}
 
 In the case of the Artin Approximation problem, we will be mostly interested in the morphism $A\lgw \wdh{A}$. Thus we need to know what is an excellent ring by Example \ref{ex_reg} ii).
 
 \begin{definition}
 A Noetherian ring $A$ is \index{excellent ring} \emph{excellent} if the following conditions hold:
 \begin{enumerate}
 \item[i)] $A$ is universally catenary.
 \item[ii)] For any $\p\in\Spec(A)$, the formal fibre of $A_{p}$ is geometrically regular.
 \item[iii)] For any $\p\in\Spec(A)$ and for any finite separable extension $\Frac\left(\frac{A}{\p}\right)\lgw \K$, there exists a finitely generated sub-$\frac{A}{\p}$-algebra $B$ of $\K$, containing $\frac{A}{\p}$,  such that $\Frac(B)=\K$ and the set of regular points of $\Spec(B)$ contains a non-empty open set.
 \end{enumerate}
 
 \end{definition} This definition may be a bit obscure at first sight and difficult to catch.  Thus we give here the main examples of excellent rings:
 \begin{example}\label{ex_excellent}
   \begin{enumerate}
 \item[i)] Local complete rings (in particular any field) are excellent. Dedekind rings of characteristic zero (for instance $\Z$)  are excellent. Any ring  which is essentially of finite type over an excellent ring is excellent  \cite[7-8-3]{EGAIV-2}.
\item[ii)] If $\k$ is a complete valued field, then the ring of convergent power series $\k\{x_1,\ldots,x_n\}$ is excellent \cite{K}.
\item[iii)] We have the following result: let $A$ be a regular  local ring containing a field of characteristic zero denoted by $\k$. Suppose  that there exists an integer $n$ such that for any maximal ideal $\m$, the field extension $\k\lgw \frac{A}{\m}$ is algebraic and ht$(\m)=n$. Suppose moreover that there exist $D_1,\ldots,D_n\in \text{Der}_k(A)$ and $x_1,\ldots,x_n\in A$ such that $D_i(x_j)=\d_{i,j}$. Then $A$ is excellent  \cite[Theorem 102]{Ma'}. 
\item[iv)] A Noetherian local ring $A$ is excellent if and only if it is universally catenary and $A\lgw \wdh{A}$ is regular \cite[7-8-3 i)]{EGAIV-2}. In particular, if $A$ is a quotient of a local regular ring, then $A$ is excellent if and only if $A\lgw \wdh{A}$ is regular \cite[5-6-4]{EGAIV-2}.
\end{enumerate}
 \end{example}

\begin{example}\cite{Na,Ma'}
Let $\k$ be a field of characteristic $p>0$ such that $[\k:\k^p]=\infty$ (for instance let us take $\k=\mathbb{F}_p(t_1,\ldots,t_n,\ldots)$ where $(t_n)_n$ is a sequence of indeterminates). Let $V:=\k^p\lb x\rb[\k]$ where $x$ is a single variable, i.e. $V$ is the ring of power series $\displaystyle\sum_{i=0}^{\infty}a_ix^i$ such that $[\k^p(a_0,a_1,\ldots):\k^p]<\infty$.  Then $V$ is a  discrete valuation ring whose completion is $\k\lb x\rb$. We have $\wdh{V}^p\subset V$, thus $[\Frac(\wdh{V}):\Frac(V)]$ is purely inseparable. Hence  $V$ is a Henselian ring by Remark \ref{separable}  since  $\wdh V$ is Henselian by Example \ref{ex_henselian}.\\
Since $[\Frac(\wdh{V}):\Frac(V)]$ is purely inseparable, $V\lgw \wdh{V}$ is not regular by Example \ref{ex_reg} and $V$ is not excellent by Example \ref{ex_excellent} iv).\\
 On the other hand, let $f$ be a power series of the form $\displaystyle\sum_{i=0}^{\infty}a_ix^i$, $a_i\in\k$ such that 
 $$[\k^p(a_0,a_1,...):\k^p]=\infty.$$ Then $f\in \wdh{V}$ but $f\notin V$, and $f^p\in V$. Thus $f$ is the only root of the polynomial $y^p-f^p$. This shows that the polynomial $y^p-f^p\in V[y]$ does not satisfy Theorem \ref{Popescu}.
\end{example}


\section{Appendix C: \'Etale morphisms and Henselian rings}\label{etale_app}
The material presented here is very classical and has first been studied by G. Azumaya \cite{Az} and M. Nagata \cite{Na1,Na2}. We will give a quick review of the definitions and properties that we need for the understanding of the rest of the paper. Nevertheless, the reader may consult \cite{Na,EGAIV-4,Ra70,Iv,Mil} for more details, in particular for the proofs we do not give here. 

\begin{example}\label{ex_elliptic} In classical algebraic geometry, the Zariski topology has too few open sets. For instance, there is no Implicit Function Theorem. Let us explain this problem through the following example:\\
Let $X$ be the zero set of the polynomial $y^2-x^2(x+1)$ in $\C^2$. On an affine open neighborhood of 0, denoted by $U$,  $X\cap  U$ is equal to $X$ minus a finite number of points, thus $X\cap  U$ is irreducible since $X$ is irreducible. In the euclidean topology, we can find an open neighborhood of $0$, denoted by $U$,  such that $X\cap  U$ is reducible, for instance take $U=\{(x,y)\in\C^2\ /\ |x|^2+|y|^2< 1/2\}$. This comes from the fact that $x^2(1+x)$ is  the square of an analytic function defined on $U\cap  (\C\times\{0\})$. Let $z(x)$ be such an analytic function, $z(x)^2=x^2(1+x)$.\\
In fact we can construct $z(x)$ by using the Implicit Function Theorem as follows. We see that $z(x)$ is a root of the polynomial $Q(x,z):=z^2-x^2(1+x)$. We have $Q(0,0)=\frac{\partial Q}{\partial z}(0,0)=0$, thus we can not use directly the Implicit Function Theorem to obtain $z(x)$ from its minimal polynomial.\\ 
Nevertheless let us  define $P(x,t):=(t+1)^2-(1+x)=t^2+2t-x$. Then $P(0,0)=0$ and $\frac{\partial P}{\partial t}(0,0)=2\neq 0$. Thus, from the Implicit function Theorem, there exists $t(x)$ analytic on a neighborhood of 0 such that $t(0)=0$  and $P(x,t(x))=0$. If we set $z(x):=x(1+t(x))$, we have $z^2(x)=x^2(1+x)$. In fact $z(x)\in B:=\frac{\C[x,t]_{(x,t)}}{(P(x,t))}$. The morphism $\C[x]\lgw B$ is a typical example of an \'etale morphism.
\end{example}

\begin{definition}\index{smooth morphism}\index{Žtale morphism}
Let $\phi : A\lgw B $ be a ring morphism essentially of finite type. We say that $\phi$ is a \emph{smooth morphism} (resp. \emph{\'etale morphism}) if for every $A$-algebra $C$ along with an ideal $I$ such that $I^2=(0)$ and any morphism of $A$-algebras $\psi : B\lgw \frac{C}{I}$ there exists a morphism $\s : B\lgw C$ (resp. a unique morphism) such that the following diagram commutes:
$$\xymatrix{A \ar[r]^{\phi} \ar[d] & B \ar[d]^{\psi} \ar@{.>}[ld]^{\s} \\
C \ar[r] & \frac{C}{I}}$$
\end{definition}

Let us mention that the definition of an \'etale morphism is not the same depending on the sources and some authors require that an \'etale morphism be of finite type (as in \cite{Ra70} or \cite{ST} for instance). But we are mainly interested in local morphisms $A\lgw B$ which are hardly of finite type, so we prefer to choose this definition. Moreover this allows the \'etale neighborhoods (see the definition below) to be \'etale morphisms, which is not the case if we impose the finite type condition. 

\begin{example}
Let $\k:=\R$ or $\C$ and let us assume that $A=\frac{\k[x_1,\ldots,x_n]}{J}$ and $B=\frac{A[y_1,\ldots,y_m]}{K}$ for some ideals $J$ and $K$. Let $X$ be the zero locus of $J$ in $\k^n$  and $Y$ be the zero locus of $K$ in $\k^{n+m}$. The morphism $\phi : A\lgw B$ defines a regular map $\Phi : Y\lgw X$. Let $C:=\frac{\k[t]}{(t^2)}$ and $I:=(t)$. Let $f_1(x),\ldots,f_r(x)$ be generators of $J$.\\
\\
 A morphism $A\lgw C$ is given by elements $a_i$, $b_i\in \k$ such that $f_j(a_1+b_1t,\ldots,a_n+b_nt)\in (t)^2$ for $1\leq j\leq r$. We have
$$f_j(a_1+b_1t,\ldots,a_n+b_nt)=f_j(a_1,\ldots,a_n)+\left(\sum_{i=1}^n\frac{\partial f_j}{\partial x_i}(a_1,\ldots,a_n)b_i\right)t \text{ mod. } (t)^2.$$
Thus a morphism $A\lgw C$ is given by a point $x:=(a_1,\ldots,a_n)\in X$ (i.e. such that $f_j(a_1,\ldots,a_n)=0$ for all $j$) and a tangent vector $u:=(b_1,\ldots,b_n)$ to $X$ at $x$ (i.e. such that $\displaystyle\sum_{i=1}^n\frac{\partial f_j}{\partial x_i}(a_1,\ldots,a_n)b_i=0$ for all $j$). In the same way a $A$-morphism $B\lgw \frac{C}{I}=\k$ is given by a point $y\in Y$. Moreover the first diagram is commutative if and only if $\Phi(y)=x$.\\
\\
 Then $\phi$ is smooth if for every $x\in X$, every $y\in Y$ and every tangent vector $u$ to $X$ at $x$ such that $\Phi(y)=x$,  there exists a tangent vector $v$ to $Y$ at $y$ such that $D_y(\Phi)(v)=u$, i.e. if $D_y(\phi)$ is surjective. And $\phi$ is \'etale if and only if $v$ is unique, i.e. if $D_y(\Phi)$ is bijective. This shows that smooth morphisms correspond to submersions in differential geometry and \'etale morphisms to local diffeomorphisms.

\end{example}

\begin{example}\label{standard}
Let $\phi : A\lgw B_{S}$ be the canonical morphism where $B:=\frac{A[x]}{(P(x))}$ and $S$ is a multiplicative system of $ B$ containing $\frac{\partial P}{\partial x}$. If we have  a commutative diagram
$$\xymatrix{A \ar[r]^{\phi} \ar[d] & B_{S} \ar[d]^{\psi} \\
C \ar[r] & \frac{C}{I}}$$
with $I^2=(0)$, 
the morphism $B_{S}\lgw \frac{C}{I}$ is given by an element $c\in C$ such that $P(c)\in I$. Looking for a lifting of $\psi$ is equivalent to find $\e\in I$ such that $P(c+\e)=0$. We have 
\begin{equation}\label{lifting}P(c+\e)=P(c)+\frac{\partial P}{\partial x}(c)\e\end{equation}
since $I^2=(0)$. Since $\frac{\partial P}{\partial x}$ is invertible in $B_{S}$, $\frac{\partial P}{\partial x}(c)$ is invertible in $\frac{C}{I}$, i.e. there exists $a\in C$ such that $a\frac{\partial P}{\partial x}(c)=1$ mod. $I$. Let $i:=a\frac{\partial P}{\partial x}(c)-1$. Then
$$a(1-i)\frac{\partial P}{\partial x}(c)=1$$
since $i^2=0$.
Thus there exists a unique $\e$ satisfying \eqref{lifting} and $\e$ is given by:
$$\e=-P(c)a(1-i).$$
 This proves that $\phi$ is \'etale. Compare this example with Example \ref{ex_elliptic}.\end{example}

\begin{definition}\label{standard_etale}\index{standard \'etale morphism}
Morphisms of the form $\phi : A\lgw \frac{A[x]}{(P(x))}_{\p}$  where $P(x)$ is monic, $\p\in \Spec(A[x])$  contains $P(x)$ but not $\frac{\partial P}{\partial x}(x)$ and $\p\cap A$ is maximal are \'etale morphisms  and these are called  \emph{standard \'etale morphisms}. 

\end{definition}

\begin{theorem}\label{iversen}\cite[II.2]{Iv} 
If $A$ and $B$  are local rings,  any \'etale morphism from $A$ to $B$ is standard Žtale.
\end{theorem}

\begin{example}[Jacobian Criterion]
If $\k$ is a field and $\phi : \k\lgw B:=\frac{\k[x_1,\ldots,x_n]_{\m}}{(g_1,\ldots,g_r)}$,  where $\m:=(x_1-c_1,\ldots,x_n-c_n)$ for some $c_i\in \k$, the morphism $\phi$ is smooth if and only if the jacobian matrix $\left(\frac{\partial g_i}{\partial x_j}(c)\right)$ has rank equal to the height of $(g_1,\ldots,g_r)$. This is equivalent to saying that $V(I)$ has a non-singular point at the origin. Let us recall that the fibers of submersions are always smooth.
\end{example}

\begin{definition}\index{\'etale neighborhood}\index{Henselization}
Let $(A,\m_A)$ be a local ring. An \emph{\'etale neighborhood} of $A$ is an
Žtale local morphism $A\lgw B$
inducing an isomorphism between the residue fields. \\
If $A$ is a local ring, the \'etale neighborhoods of $A$ form a filtered inductive system and the limit of this system is called the \emph{Henselization} of $A$ (cf. \cite[III. 6]{Iv}  or \cite[VIII]{Ra}) and is denoted by $A^h$.\\
We say that $A$ is \emph{Henselian} if $A=A^h$. The morphism $\imath_A : A\lgw A^h$ is universal among all the morphisms $A\lgw B$ inducing an isomorphism on the residue fields and where $B$ is a Henselian local ring. The morphism $\imath_A$ is called the \emph{Henselization morphism} of $A$.
\end{definition}

\begin{remark}\label{separable}
If $A$ is a local domain, then $\Frac(A)\lgw \Frac(A^h)$ is an algebraic separable extension. Indeed $A^h$ is the limit of Žtale neighborhoods of $A$ which are localizations of Žtale morphisms by Theorem \ref{iversen}, thus $A^h$ is a limit of separable algebraic extensions.
\end{remark}

\begin{proposition}\label{section_hens}
If $A$ is a Noetherian local ring,  its Henselization $A^h$ is a Noetherian local ring and $\imath_A:A\lgw A^h$ is faithfully flat (in particular it is injective). If $\phi : A^h\lgw B$ is an \'etale morphism  there is a section $\s: B\lgw A^h$, i.e. $\s\circ\phi=id_{A^h}$.  
\end{proposition}

\begin{remark}

\begin{itemize}
\item[i)] Let $\phi :A\lgw B$ be a morphism of local rings. We denote by $\imath_A :A\lgw A^h$ and $\imath_B :B\lgw B^h$ the Henselization morphisms. By the universal property of the Henselization the morphism $\imath_B\circ A:A\lgw B^h$ factors through $A^h$ in a unique way, i.e. there exists a unique morphism $\phi^h:A^h\lgw B^h$ such that $\phi^h\circ\imath_A=\imath_B\circ\phi$.
\item[ii)] If $\phi :A\lgw B$ is an Žtale morphism between two local rings, $\phi^h$ is an isomorphism. Indeed $\phi$ being Žtale  $\imath_A$ factors through $\phi$, i.e. there exists a unique morphism $s:B\lgw A^h$ such that $s\circ\phi=\imath_A$. The morphism $s$ induces a morphism $s^h:B^h\lgw A^h$ as above. Thus we have the following commutative diagram:
$$\xymatrix{ A \ar[r]^{\phi} \ar[d]_{\imath_A} & B \ar[ld]_{s} \ar[d]^{\imath_B}  \\ A^h \ar[r]^{\phi^h} & B^h \ar@/^/@{->}[l]^{s^h} }$$
Since $s\circ\phi=\imath_A$, $(s\circ\phi)^h=s^h\circ\phi^h=\text{id}_A$. On the other hand $(\phi^h\circ s)^h=\phi^h\circ s^h=\imath_B^h=\text{id}_B$. This shows that $\phi^h$ is an isomorphism and $s^h$ is its inverse.
\item[iii)] If $\phi :A\lgw B$ is an Žtale morphism between two local rings where $A$ is Henselian, the previous remark implies that $\phi$ is an isomorphism since $\imath_B :B\lgw B^h$ is injective.
\end{itemize}
\end{remark}

\begin{proposition}\label{lift}
Let $A$ be a Henselian local ring and let $\phi:A\lgw B$ be an \'etale morphism that admits a section in $\frac{A}{\m_A^c}$ for some $c\geq 1$, i.e. a morphism of $A$-algebras  $s:B\lgw \frac{A}{\m_A^c}$. Then there exists a section $\wdt{s}: B\lgw A$ such that $\wdt{s}=s$ modulo $\m^c$.
\end{proposition}

\begin{proof}
Let $\m:=s^{-1}(\m_A)$. Since $s$ is a $A$-morphism, $\m\cap A=\m_A$,  $\m$ is a maximal ideal of $B$ and $\frac{B}{\m}$ is isomorphic to $\frac{A}{\m_A}$. Because $A$ is Henselian and the morphism $\psi:A\lgw B_{\m}$ induced by $\phi$ is an Žtale neighborhood,  $\psi$ is an isomorphism. Then $\psi^{-1}$ composed with the localization morphism $B\lgw B_{\m}$ gives the desired section.
 \end{proof}

 
\begin{remark}
Let $(A,\m_A)$ be a local ring. Let $P(y)\in A[y]$  and $a\in A$ satisfy $P(a)\in\m_A $ and $\frac{\partial P}{\partial y}(a)\notin \m_A$. If $A$ is Henselian,  $A\lgw \frac{A[y]}{(P(y))}\, _{\m_A+(y-a)}$ is an \'etale neighborhood of $A$, thus it admits a section. This means that  there exists $\wdt{y}\in \m_A$ such that $P(a+\wdt{y})=0$.\\
In fact this characterizes Henselian local rings:
\end{remark}

\begin{proposition}\label{propC12}
Let $A$ be a local ring. Then $A$ is Henselian if and only if for any $P(y)\in A[y]$  and $a\in A$ such that $P(a)\in\m_A $ and $\frac{\partial P}{\partial y}(a)\notin \m_A$ there exists $\wdt{y}\in \m_A$ such that $P(a+\wdt{y})=0$.\end{proposition}

We can generalize this proposition as follows:

 \begin{theorem}[Implicit Function Theorem]\label{IFT}
 Set $y=(y_1,\ldots,y_m)$ and let $f(y)\in A[y]^r$ with $r\leq m$. Let $J$ be the ideal of $A[y]$ generated by the $r\times r$ minors of the Jacobian matrix of $f(y)$. Assume that $A$ is Henselian, $f(0)=0$ and $J\not\subset \m_A.\frac{A[y]}{(y)}$. Then there exists $\wdt{y}\in \m_A^m$ such that $f(\wdt{y})=0$.\\
 \end{theorem}
 
  \begin{example}\label{ex_henselian}
  The following rings are Henselian local rings:
  \begin{itemize}
  \item Any complete local ring is Henselian.
\item The ring of germs of $\mathcal{C}^{\infty}$ functions at the origin of $\R^n$ is a Henselian local ring but it is not Noetherian. 
\item The ring of germs of analytic functions at the origin of $\C^n$ is a Noetherian Henselian local ring; it is isomorphic to the ring of convergent power series. 
 \item By Proposition \ref{propC12} any quotient of a local Henselian ring is again a local Henselian ring.
 \item The next example shows that the rings of algebraic power series over a field are Henselian.
 \end{itemize}
 \end{example}

 \begin{example}
 If $A=\k\llceil x_1,\ldots,x_n\rrceil$ for some Weierstrass system over $\k$, then $A$ is a Henselian local ring by Proposition \ref{propC12}. Indeed, let $P(y)\in A[y]$ satisfy $P(0)=0$ and $\frac{\partial P}{\partial y}(0)\notin (\p,x)$. Then $P(y)$ has a non-zero term of the form $cy$, $c\in\k^*$. So we have, by the Weierstrass Division Property, 
 $$y=P(y)Q(y)+r$$ where $r\in \m_A$. By considering the derivatives with respect to $y$ of both terms of this equality and evaluating at 0 we see that $Q(0)\neq 0$, i.e. $Q(y)$ is a unit. Thus $Q(r)\neq 0$ and  $P(r)=0$.
 \end{example}
 
   We have the following generalization of Proposition \ref{propC12}:

 \begin{proposition}[Hensel's Lemma]\label{Henselian}
 Let $(A,\m_A)$ be a local ring. Then $A$ is Henselian if and only if for any monic polynomial $P(y)\in A[y]$ such that $P(y)=f(y)g(y)$ mod $\m_A$ for some monic polynomials $f(y)$, $g(y)\in A[y]$ which are coprime modulo $\m_A$, there exist monic polynomials $\wdt{f}(y)$, $\wdt{g}(y)\in A[y]$ such that $P(y)=\wdt{f}(y)\wdt{g}(y)$ and $\wdt{f}(y)-f(y)$, $\wdt{g}(y)-g(y)\in \m_A[y]$.
 
  \end{proposition}

  \begin{proof}
 Let us prove the sufficiency of the condition.   Let $P(y)\in A[y]$  and $a\in A$ satisfy $P(a)\in\m_A $ and $\frac{\partial P}{\partial y}(a)\notin \m_A$. This means that $P(X)=(X-a)Q(X)$ where $X-a$ and $Q(X)$ are coprime modulo $\m$. Then this factorization lifts to $A[X]$, i.e. there exists
  $\wdt{y}\in \m_A$ such that $P(a+\wdt{y})=0$. This proves that $A$ is Henselian.\\
To prove that the condition is necessary, let $P(y)\in A[y]$ be a monic polynomial, $P(y)=y^d+a_1y^{d-1}+\cdots+a_d$. Let $\k:=\frac{A}{\m_A}$ be the residue field of $A$. For any $a\in A$ let us denote by $\ovl{a}$  the image of $a$ in $\k$. Let us assume that $\ovl{P}(y)=f(y)g(y)$ mod $\m_A$ for some $f(y)$, $g(y)\in \k[y]$ which are coprime in $\k[y]$. Let us write
\begin{align*}
&  f(y)=y^{d_1}+b_1 y^{d_1-1}+\cdots+b_{d_1},  \quad g(y)=y^{d_2}+c_1 y^{d_2-1}+\cdots+c_{d_2}
\end{align*}
where  $b=(b_1,\ldots, b_{d_1})\in \k^{d_1}$,  
$c=(c_1,\ldots, c_{d_2})\in \k^{d_2}$.
The product of polynomials $\ovl{P}=fg$ defines  a map  $ \Phi :\k^{d_1}\times\k^{d_2} \to \k^d$,  that is polynomial 
in $b$ and $c$  with integer coefficients, and $\Phi (b,c) = \ovl{a}:=(\ovl{a}_1,\ldots,\ovl{a}_d)$.  
  The determinant of the Jacobian matrix $\frac {\partial \Phi }{\partial (b,c)}$ 
 is the resultant of $f(y)$ and $g(y)$, and hence is non-zero  at $(b,c)$.   By the Implicit Function Theorem (Theorem \ref{IFT}), there exist  $\wdt{b}\in A^{d_1}$, $\wdt{c}\in A^{d_2}$ such that $P(y)= P_1(y) P_2(y)$ where $ P_1(y)=y^{d_1}+\wdt{b}_1 y^{d_1-1}+\cdots+\wdt{b}_{d_1}$ and $ P_2(y)=y^{d_2}+\wdt{c}_1 y^{d_2-1}+\cdots+\wdt{c}_{d_2}$.
   \end{proof}

  \begin{proposition}\label{henz_exc}\cite[18-7-6]{EGAIV-4} 
 Given  an excellent local ring $A$, its Henselization $A^h$ is also an excellent local ring.

 \end{proposition}

 \printindex
%

\end{document}